\documentclass[a4paper,11pt]{amsart}
\usepackage{amsfonts}
\usepackage{a4wide}
\usepackage{amsthm}
\usepackage{amsmath}
\usepackage{amssymb}
\usepackage[final]{graphicx}
\usepackage{graphics}
\usepackage{graphicx}
\usepackage{epsfig}
\usepackage{tabularx}
\usepackage{enumerate}
\usepackage[all]{xy}
\usepackage{eepic}
\usepackage[UKenglish]{babel}

\newtheorem{theorem}{Theorem}[section]
\newtheorem{corollary}[theorem]{Corollary}

\newtheorem{lemma}[theorem]{Lemma}
\newtheorem{proposition}[theorem]{Proposition}

\theoremstyle{definition}
\newtheorem{definition}[theorem]{Definition}
\newtheorem{remark}[theorem]{Remark}

\title{Topological dynamics of the doubling map with asymmetrical holes}
\author{Rafael Alcaraz Barrera}
\date{\today}
\address{School of Mathematics\\
The University of Manchester\\
Oxford Road, Man\-ches\-ter, M13 9PL, UK}
\email{rafael.alcarazbarrera@manchester.ac.uk} 
\thanks{Some results of this paper are contained in the author's doctoral dissertation \cite{yomero2}. Research was supported by CONACyT scholarship for Doctoral Students  no. 213600.}

\begin{document}
\begin{abstract}
We study the dynamics of the attractor of the doubling map with an asymmetrical hole by associating to each hole an element of the lexicographic world. A description of the topological entropy function is given. We show that the set of parameters $(a,b)$ such that the dynamics of the mentioned attractor corresponds to a subshift of finite type is open and dense. Using the connections between this family of open dynamical systems, intermediate $\beta$-expansions and Lorenz maps we study the topological transitivity and the specification property for such maps.
\end{abstract}
\maketitle

\section{Introduction and Summary}
\label{intro}

\noindent The purpose of this paper is to study the dynamical properties of a family of \textit{open dynamical systems} corresponding to the doubling map $2x \mod 1$. Let us remind the reader our setting: let $f: S^1 \to S^1$ denote the doubling map. Following \cite{urbanski2}, $0$ will denote the fixed point of $f$ and if $a, b \in S^1$ w with $a \neq b$ then $(a,b)$ denotes the open arc anticlockwise oriented from $a$ to $b$ and it is called open interval. Given $a,b \in S^1$ we say that $a < b$ if $l((0,a)) < l((0,b))$ where $l$ denotes the length of the segments $(x,y)$ with $x,y \in S^1$. Given $(a,b) \subset S^1$, $X_{(a,b)}$ is \textit{the exceptional set corresponding to $(a,b)$}, that is $$X_{(a,b)} = \left\{x \in S^1 \mid f^n(x) \notin (a,b) \hbox{\rm{ for every }} n \in \mathbb{N}\right\}.$$ Since $X_{(a,b)}$ is a forward $f$-invariant set it is possible to consider the map $f_{(a,b)} = f\mid_{X_{(a,b)}}$. We call the pair $(X_{(a,b)}, f_{(a,b)})$ an \textit{open dynamical system} or \textit{exclusion system} for the doubling map. There is a recent body of works which studies particular features of such systems- see \cite{dettman, sidorov3, sidorov1, nilsson, sidorov2}.

\vspace{1em}We will concentrate in studying the dynamical properties of the attractor of $X_{(a,b)}$ denoted by $\Lambda_{(a,b)}$ when $(a,b)$ is a centred interval, that is an open interval $(a,b)$ such that $\frac{1}{2} \in (a,b)$. Observe that $\Lambda_{(a,b)} = [2b-1,2a]\cap X_{(a,b)}$. In particular we are interested in understanding when $(\Lambda_{(a,b)}, f_{(a,b)})$ is transitive, has specification property whenever $\Lambda_{(a,b)}$ has positive Hausdorff dimension as asked in \cite{sidorov4}. A natural question to pose for the doubling map is if the mentioned properties will follow because of the symbolic representation of the boundary points of the interval -see \cite{bundfuss}. Also, it is natural to ask if the topological entropy of these systems change, and if it does, if it has dependence on size or position of the hole $(a,b)$. This questions have been studied for the symmetric case, i.e. intervals of the form $(a, 1-a)$ in \cite{yomero1} and for non centred intervals, i.e. $a = 0$ or $b =1$ in \cite{nilsson}. Both cases have a strong relationship with the dynamical properties of the set of points with unique $\beta$-expansion when $\beta \in (1,2)$ and with the classical $\beta$-shift respectively.

\vspace{1em}Our study is developed by using tools of symbolic dynamics and most of our proofs are essentially symbolic. In Proposition \ref{attractor1} and Theorem \ref{asociar} we show that an open map can be represented by a unique element of the \textit{lexicographic world} \footnote{It is worth to mention that the lexicographic world was introduced by Gan in \cite{gan}. However, this notion was studied previously -see \cite{guckwill} among others.}. The lexicographic world also provides a tool to study the dynamical features of Lorenz maps -see \cite{glendinning, hall, hofbauer2, hofbauer1, hubbardsparrow}- by their \textit{kneading invariants}. One of the novelties of our approach is to extend the lexicographic world to extremal sequences, i.e. sequences not considered in the lexicographic world -see Definition \ref{lexworld}, but that correspond to the binary expansion of the end points of an open interval $(a,b)$.

\vspace{1em}The structure of the paper is as follows: In Section \ref{basic} we introduce all the tools from symbolic dynamics used during the paper. Also we give show that the attractors $(\Lambda_{(a,b)},f_{(a,b)})$ are essentially lexicographic subshifts. In Section \ref{generic} we show in Theorem \ref{opendense1} that the set of parameters $(a,b)$ such that the corresponding attractor is conjugated to a subshift of finite type is open and dense. As a corollary we obtain that the topological entropy of $(\Lambda_{(a,b)}, f_{(a,b)})$ is locally constant - see Corollary \ref{opendense}. Also, in Theorem \ref{asociar1} we extend the lexicographic world to extremal sequences see (Definition \ref{lexworld}). In Section \ref{yasequedaasi} we also discuss the relation between intermediate $\beta$-expansions, Lorenz maps and open dynamical systems for the doubling map. In Section \ref{transitivity} we study the topological transitivity of $(\Lambda_{(a,b)}, f_{(a,b)})$ using a notion of renormalisation. In Theorem \ref{aproximaporabajogeneral} we characterise symbolically the intervals $(a,b)$ whose corresponding attractor $(\Lambda_{(a,b)}, f_{(a,b)})$ is topologically transitive. In Section \ref{specification} we study the specification property for $(\Lambda_{(a,b)},f_{(a,b)})$. In Theorem \ref{sispec} we give a sufficient condition for a lexicographic subshift $(\Sigma_{(\boldsymbol{\alpha},\boldsymbol{\beta})},\sigma_{(\boldsymbol{\alpha},\boldsymbol{\beta})})$ to have the specification property. We construct a family of examples of lexicographic subshifts $(\Sigma_{(\boldsymbol{\alpha},\boldsymbol{\beta})},\sigma_{(\boldsymbol{\alpha},\boldsymbol{\beta})})$ with no specification and we also provide a sufficient condition on the pairs $(\boldsymbol{\alpha},\boldsymbol{\beta})$ which guarantee that $(\Sigma_{(\boldsymbol{\alpha},\boldsymbol{\beta})},\sigma_{(\boldsymbol{\alpha},\boldsymbol{\beta})})$ has no specification in Theorem \ref{nospec}.   

\section[Preliminaries]{Symbolic Dynamics, the Lexicographic World and the doubling map with holes}
\label{basic}

\subsection*{Symbolic Dynamics}
\noindent For the convenience of the reader, we give all the relevant concepts from Symbolic Dynamics to develop our study. A detailed exposition in Symbolic Dynamics can be consulted in \cite{lindmarcus}. 

\vspace{1em} We will restrict our attention to subshifts defined on the alphabet $\{0,1\}$. The elements of $\{0,1\}$ will be referred as \textit{symbols} or \textit{digits}. We denote by $\Sigma_2$ to the set of all one sided sequences with symbols in $\{0,1\}$, that is $\Sigma_2 = \mathop{\prod}\limits_{n=1}^{\infty}\{0, 1\}$. It is well known that $\Sigma_2$ is a compact metric space with the distance given by: $$d(\boldsymbol{x},\boldsymbol{y}) = \left\{
\begin{array}{clrr}      
2^{-j} & \hbox{\rm{ if }}  \boldsymbol{x} \neq \boldsymbol{y};& 
\hbox{\rm{where }} j = \min\{i \mid x_i \neq y_i \}\\
0 & \hbox{\rm{ otherwise.}}&\\
\end{array}
\right.$$ 

Let $\pi: \Sigma_2 \to \left[0,1\right]$ be \textit{the projection map} given by $\pi(\boldsymbol{x}) = \mathop{\sum}\limits_{i=1}^{\infty} \frac{x_i}{2^i}$ and consider $\sigma: \Sigma_2 \to \Sigma_2$ to be the \textit{one sided full shift map}, that is $\sigma((x_i)_{i=1}^{\infty}) = (x_{i+1})_{i=1}^{\infty}$. Let $A \subset \Sigma_2$. We say that $(A, \sigma_A)$ is a \textit{subshift of $\Sigma_2$} if $A$ is a closed $\sigma$-invariant set and $\sigma_A$ is defined to be $\sigma_A = \sigma \mid_A$.  

\vspace{1em}A \textit{word} is a finite sequence of symbols $\omega = w_1,\ldots w_n$ where $w_i \in \{0,1\}$, and denote the \textit{length of $\omega$} by $\ell(\omega)$. Given two finite words $\omega = w_1,\ldots w_n$ and $\nu = u_1 \ldots u_m$ we write $\omega\nu$ to denote their \textit{concatenation}, i.e. $\omega\nu = w_1,\ldots w_n u_1 \ldots u_m$ and $\omega^n$ stands for the word $\omega$ concatenated to itself $n$ times. Given $\boldsymbol{x} \in \Sigma_2$ and a word $\omega$ we say that \textit{$\omega$ is a factor of $\boldsymbol{x}$} or \textit{$\omega$ occurs in $\boldsymbol{x}$} if there are coordinates $i$ and $j$ such that $\omega = x_i \ldots x_j$. Note that the same definition holds if $x$ is a finite word. Consider $\mathcal{F}$ to be a set of words and let $$\Sigma_{\mathcal{F}} = \left\{\boldsymbol{x} \in \Sigma_2 \mid \upsilon \hbox{\rm{ is not a factor of }} \boldsymbol{x} \hbox{\rm{ for any word }} \upsilon \in \mathcal{F}\right\}.$$ The set $\Sigma_{\mathcal{F}}$ is always a closed, $\sigma$-invariant set. Therefore, the dynamical system given by $(\Sigma_{\mathcal{F}}, \sigma\mid_{\Sigma_{\mathcal{F}}})$ is a subshift of $\Sigma_2$ and the set $\mathcal{F}$ is called a \textit{set of forbidden factors}. Conversely, for every compact and $\sigma$-invariant set $A$, there always exist a set of forbidden factors $\mathcal{F}$ such that $A = \Sigma_{\mathcal{F}}$ \cite[Theorem 6.1.21]{bundfuss}. We say that a subshift $(\Sigma_{\mathcal{F}},\sigma\mid_{\Sigma_{\mathcal{F}}})$ is a \textit{subshift of finite type} if $\mathcal{F}$ is finite. A subshift $(\Sigma_{\mathcal{F}},\sigma\mid_{\Sigma_{\mathcal{F}}})$ is said to be \textit{sofic} if there is a subshift of finite type $(X, \sigma_{X})$ and a semi-conjugacy $h:X \to \Sigma_{\mathcal{F}}$. 

\vspace{1em}Let $\boldsymbol{\alpha} = (a_i)_{i=1}^{\infty}$ and $\boldsymbol{\beta} = (b_i)_{i=1}^{\infty} \in \Sigma_2$. We say that $\boldsymbol{\alpha}$ is \textit{lexicographically less than} $\boldsymbol{\beta}$, denoted by $\boldsymbol{\alpha} \prec \boldsymbol{\beta}$ if there exists $k \in \mathbb{N}$ such that $a_j = b_j$ for $i < k$ and $a_k < b_k$. Note that the lexicographic order is induced on finite words of the same length using the same definition. If $\boldsymbol{\alpha} \prec \boldsymbol{\beta}$ then the \textit{lexicographic open interval from $\boldsymbol{\alpha}$ to $\boldsymbol{\beta}$} is the set $$(\boldsymbol{\alpha},\boldsymbol{\beta})_{\prec} = \left\{ 
\boldsymbol{x} \in \Sigma_2 \mid \boldsymbol{\alpha} \prec \boldsymbol{x} \prec \boldsymbol{\beta} \right\}.$$ Similarly, it is possible to consider the \textit{lexicographic closed interval 
from $\boldsymbol{\alpha}$ to $\boldsymbol{\beta}$}, $\left[\boldsymbol{\alpha}, \boldsymbol{\beta}\right]_{\prec}$, by changing $\prec$ for $\preccurlyeq$.  

\vspace{1em}Set $\Sigma^1_2 = \{\boldsymbol{x} \in \Sigma_2 \mid x_1 = 1\}$ and $\Sigma^0_2 = \{\boldsymbol{x} \in \Sigma_2 \mid x_1 = 0 \}$. Note that $\Sigma^0_2 = \{\boldsymbol{x} \in \Sigma_2 \mid \boldsymbol{x} = \bar{\boldsymbol{\alpha}} \hbox{\rm{ with }} \boldsymbol{\alpha} \in \Sigma_2^1\},$ where $\bar{\boldsymbol{x}}$ denotes \textit{the mirror image of $\boldsymbol{x}$}, that is $\bar{\boldsymbol{x}} = (1 - x_i)_{i=1}^{\infty}$. Let $(\boldsymbol{\alpha},\boldsymbol{\beta})\in \Sigma_2^1 \times \Sigma^0_2$ and $\Sigma_{(\boldsymbol{\alpha},\boldsymbol{\beta})}$ be given by: $$\Sigma_{(\boldsymbol{\alpha},\boldsymbol{\beta})} = \{\boldsymbol{x} \in \Sigma_2 \mid \boldsymbol{\beta} \preccurlyeq \sigma^{n}(\boldsymbol{x}) \preccurlyeq \boldsymbol{\alpha} \hbox{\rm{ for every }} n \geq 0\}.$$ Since $\Sigma_{(\boldsymbol{\alpha},\boldsymbol{\beta})}$ is a closed and $\sigma$-invariant subset of $\Sigma_2$ we call the pair $(\Sigma_{(\boldsymbol{\alpha},\boldsymbol{\beta})}, \sigma_{(\boldsymbol{\alpha},\boldsymbol{\beta})})$ \textit{the $(\boldsymbol{\alpha},\boldsymbol{\beta})$-lexicographic subshift} or simply a \textit{lexicographic subshift}, where $\sigma_{(\boldsymbol{\alpha},\boldsymbol{\beta})} = \sigma\mid_{\Sigma_{(\boldsymbol{\alpha},\boldsymbol{\beta})}}$.

\vspace{1em}A sequence $\boldsymbol{\alpha} \in  \Sigma^1_2$ is said to be a \textit{Parry sequence} if $\sigma^n(\boldsymbol{\alpha}) \preccurlyeq \boldsymbol{\alpha}$ for every $n \in \mathbb{N}$. We denote the set of Parry sequences by $P$. Note that for every $\boldsymbol{\alpha} \in P$ $\overline{\sigma^n(\boldsymbol{\alpha})} = \sigma^n(\bar{\boldsymbol{\alpha}}) \succ \bar{\boldsymbol{\alpha}}$. Denote by $\bar P = \{\boldsymbol{x} \in \Sigma^0_2 \mid \bar{\boldsymbol{x}} \in P\}$. Observe that $\bar P$ coincides with the set defined by Nilsson in \cite[p. 105]{nilsson}. By \cite[Theorems 3.7, 3.8]{nilsson}, $\pi(P)$ and $\pi(\bar P)$ are sets of Lebesgue measure zero and $\dim_H(\pi(P)) = \dim_H(\pi(\bar P)) = 1$ where $\dim_H$ denotes the Hausdorff dimension. Note that if $\boldsymbol{x} \in P$ then $\boldsymbol{x} \in [1^n0^{\infty}, 1^{n+1}0^{\infty}]_{\prec}$ for some $n \in \mathbb{N}$ and if $\boldsymbol{x} \in \bar P$ then $\boldsymbol{x} \in [0^{n+1}10^{\infty}, 0^{n}10^{\infty}]_{\prec}$ for some $n \in \mathbb{N}$. From \cite[Theorem 3.6]{nilsson}, we are also sure that every $\boldsymbol{\alpha} \in P$ is a limit point of $P$, therefore $P$ is a perfect set. Since $P$ is totally disconnected and compact, then $P$ and $\bar{P}$ are Cantor set with $\dim_H{P}=\dim_H(\bar{P}) = 1$. 

\begin{definition}
Let $(\boldsymbol{\alpha},\boldsymbol{\beta}) \in \Sigma_2^1 \times \Sigma_2^0$ such that: 
\begin{enumerate}[$i)$]
\item $\boldsymbol{\alpha} \in P$ and $\boldsymbol{\beta} \in \bar{P}$;
\item $\sigma^{n}(\boldsymbol{\alpha}) \succcurlyeq \boldsymbol{\beta}$ and $\sigma^{n}(\boldsymbol{\beta})  \preccurlyeq \boldsymbol{\alpha}$ for every $n \in \mathbb{N}$.
\end{enumerate}
A pair $(\boldsymbol{\alpha},\boldsymbol{\beta}) \in \Sigma^1_2 \times \Sigma^0_2$ is said to be \textit{extremal} if $(\boldsymbol{\alpha},\boldsymbol{\beta})$ does not satisfy $i)$ or $ii)$ and the family of non-extremal pairs is called \textit{the lexicographic world} and we will denote it as $\mathcal{LW}$.\label{lexworld}
\end{definition}

It is clear that $\mathcal{LW} \subset P \times \bar{P}$. Therefore, \cite[Theorem 3.6]{nilsson} implies that $\pi(\mathcal{LW})$ has Lebesgue measure zero.
   
\vspace{1em}Notice that if $(\boldsymbol{\alpha},\boldsymbol{\beta}) \in \mathcal{LW}$ then neither $\boldsymbol{\alpha}$ nor $\boldsymbol{\beta}$ have arbitrarily long strings of $0$'s or $1$'s unless $\boldsymbol{\alpha} = 1^{\infty}$ or $\boldsymbol{\beta} = 0^{\infty}$. Also, it is clear that if $\boldsymbol{\alpha} = 1^{\infty}$ and $\boldsymbol{\beta} = 0^{\infty}$ then $\Sigma_{(\boldsymbol{\alpha},\boldsymbol{\beta})} = \Sigma_2$. Given a sequence $\boldsymbol{\alpha} \in \Sigma_2$, consider $$0_{\boldsymbol{\alpha}} = \max\{n \in \mathbb{N} \mid 0^n \hbox{\rm{ is a factor of }} \boldsymbol{\alpha}\},$$ and $$1_{\boldsymbol{\alpha}} = \max\{n \in \mathbb{N}\mid 1^n \hbox{\rm{ is a factor of }} \boldsymbol{\alpha}\}.$$ Observe that $0_{\boldsymbol{\alpha}}$ and $1_{\boldsymbol{\alpha}}$ are well defined if $((\boldsymbol{\alpha},\boldsymbol{\beta})) \in \mathcal{LW} \setminus \{0^{\infty}, 1^{\infty}\}.$ Moreover, it is clear that for every $((\boldsymbol{\alpha},\boldsymbol{\beta}) \in \mathcal{LW}$, $0_{\boldsymbol{\alpha}} \leq 0_{\boldsymbol{\beta}}$ and $1_{\boldsymbol{\beta}} \leq 1_{\boldsymbol{\alpha}}$.  We say that a lexicographic subshift $(\Sigma_{(\boldsymbol{\alpha},\boldsymbol{\beta})}, \sigma_{(\boldsymbol{\alpha},\boldsymbol{\beta})})$ is \textit{asymmetric} if there exist $\boldsymbol{x} \in \Sigma_{(\boldsymbol{\alpha},\boldsymbol{\beta})}$ such that $\bar{\boldsymbol{x}} \notin \Sigma_{(\boldsymbol{\alpha},\boldsymbol{\beta})}$.

\vspace{1em}Given a lexicographic subshift $(\Sigma_{(\boldsymbol{\alpha},\boldsymbol{\beta})},\sigma_{(\boldsymbol{\alpha},\boldsymbol{\beta})})$, \textit{the set of admissible words of length $n$ of $(\Sigma_{(\boldsymbol{\alpha},\boldsymbol{\beta})}, \sigma_{(\boldsymbol{\alpha},\boldsymbol{\beta})})$} will be denoted by $B_n(\Sigma_{(\boldsymbol{\alpha},\boldsymbol{\beta})})$ and $$\mathcal{L}(\Sigma_{(\boldsymbol{\alpha},\boldsymbol{\beta})}) = \mathop{\bigcup}\limits_{n=1}^{\infty}B_n(\Sigma_{(\boldsymbol{\alpha},\boldsymbol{\beta})})$$ stands for \textit{language of $\Sigma_{(\boldsymbol{\alpha},\boldsymbol{\beta})}$}. Given a lexicographic subshift $(\Sigma_{(\boldsymbol{\alpha},\boldsymbol{\beta})}, \sigma_{(\boldsymbol{\alpha},\boldsymbol{\beta})})$, we define the \textit{topological entropy of $(\Sigma_{(\boldsymbol{\alpha},\boldsymbol{\beta})}, \sigma_{(\boldsymbol{\alpha},\boldsymbol{\beta})})$} by $$h_{top}(\sigma_{(\boldsymbol{\alpha},\boldsymbol{\beta})}) = \mathop{\lim}\limits_{n \to \infty} \dfrac{1}{n}\log \left|B_n(\Sigma_{(\boldsymbol{\alpha},\boldsymbol{\beta})})\right|$$ where $\log$ is always considered to be $\log_2$. 

\vspace{1em}We say that a lexicographic subshift $(\Sigma_{(\boldsymbol{\alpha},\boldsymbol{\beta})}, \sigma_{(\boldsymbol{\alpha},\boldsymbol{\beta})})$ is \textit{topologically transitive} if for any two words $\omega, \nu \in \mathcal{L}(\Sigma_{(\boldsymbol{\alpha},\boldsymbol{\beta})})$ there exist a word $\upsilon \in \mathcal{L}(\Sigma_{(\boldsymbol{\alpha},\boldsymbol{\beta})})$ such that $\omega \upsilon \nu \in \mathcal{L}(\Sigma_{(\boldsymbol{\alpha},\boldsymbol{\beta})})$. Also we say that a lexicographic subshift has \textit{the specification property} if $(\Sigma_{(\boldsymbol{\alpha},\boldsymbol{\beta})}, \sigma_{(\boldsymbol{\alpha},\boldsymbol{\beta})})$ is transitive and there exist $M \in \mathbb{N}$ such that for every $\omega, \nu$, $\ell(\upsilon) = M$.  We say that $(\Sigma_{(\boldsymbol{\alpha},\boldsymbol{\beta})},\sigma_{(\boldsymbol{\alpha},\boldsymbol{\beta})})$ is \textit{coded} if there exist a sequence of transitive lexicographic subshifts of finite type $\left\{(\Sigma_{(\boldsymbol{\alpha}_n,\boldsymbol{\beta}_n)},\sigma_{(\boldsymbol{\alpha}_n,\boldsymbol{\beta}_n)})\right\}_{n=1}^{\infty}$ such that $\Sigma_{(\boldsymbol{\alpha}_n,\boldsymbol{\beta}_n)} \subset \Sigma_{(\boldsymbol{\alpha}_{n+1},\boldsymbol{\beta}_{n+1})}$ for every $n \in \mathbb{N}$ and $\Sigma_{(\boldsymbol{\alpha},\boldsymbol{\beta})} = \overline{\mathop{\bigcup}\limits_{n=1}^{\infty} \Sigma_{(\boldsymbol{\alpha}_n,\boldsymbol{\beta}_n)}}$. 

\vspace{1em}To complete our exposition, we need tools form combinatorics of words. A detailed exposition can be consulted in \cite[Chapter 2]{lothaire}. 

\vspace{1em}Let $\omega$ be a word. We denote by $0\hbox{\rm{-max}}_{\omega}$ to the lexicographically biggest cyclic permutation of $\omega$ starting with $0$ and $1\hbox{\rm{-min}}_{\omega}$ to the lexicographically smallest cyclic permutation of $\omega$ starting with $1$. Also, $\hbox{\rm{max}}_{\omega}$ denotes the lexicographically largest cyclic permutation of $\omega$ and $\hbox{\rm{max}}_{\omega}$ denotes the lexicographically smallest cyclic permutation of $\omega$. It is clear that $\sigma(1\hbox{\rm{-min}}_{\omega}^\infty) = \hbox{\rm{min}}_{\omega}^{\infty}$ and 
$\sigma(0\hbox{\rm{-max}}_{\omega}^{\infty}) = \hbox{\rm{max}}_{\omega}^{\infty}$. 

\begin{proposition}
Let $\omega$ be a word such that $\omega \neq 1^n$ and $\omega \neq 0^m$ for any $n,m \in \mathbb{N}$. Then $\hbox{\rm{max}}_{\omega}$ ends with $0$ and $\hbox{\rm{min}}_{\omega}$ ends with $1$. \label{endings}
\end{proposition}
\begin{proof}
If $\hbox{\rm{max}}_{\omega}$ ends with $1$ then $1\hbox{\rm{max}}_{{\omega}_1},\ldots \hbox{\rm{-
max}}_{{\omega}_{\ell(\omega)-1}} \succ \hbox{\rm{max}}_{\omega}$ which contradicts that $\hbox{\rm{max}}_{\omega}$ is maximal. Therefore, $\hbox{\rm{max}}_{\omega}$ ends with $0$. The proof of $\hbox{\rm{min}}_{\omega}$ ending with $1$ is similar.
\end{proof}

\begin{proposition}
Let $\omega$ be a word such that $\omega \neq 1^n$ and $\omega \neq 0^m$ for any $n,m \in \mathbb{N}$. Then there exist words $\upsilon$ and $\nu$ such that $u_1 = 0$, $v_1 = 1$, $0\hbox{\rm{-max}}_{\omega}^{\infty} = (\upsilon\nu)^{\infty}$ and $1\hbox{\rm{-min}}_{\omega}^\infty = (\nu\upsilon)^{\infty}$ \label{palabrasciclicas}
\end{proposition}
\begin{proof}
Let $\omega$ be a word. Observe that $(0\hbox{\rm{-max}}_{\omega})^{\infty}$ and $(1\hbox{\rm{-min}}_{\omega})^{\infty}$ are cyclic permutations of each other. Then, there exist $n < \ell(\omega)$ such that $\sigma^n(0\hbox{\rm{-max}}_{\omega}^{\infty}) = (1\hbox{\rm{-min}}_{\omega})^{\infty}$. Then $$\upsilon = {0\hbox{\rm{-max}}_{\omega}}_1 \ldots {0\hbox{\rm{-max}}_{\omega}}_{n-1}$$ and $$\nu =  {0\hbox{\rm{-max}}_{\omega}}_n \ldots {0\hbox{\rm{-max}}_{\omega}}_{\ell(\omega)}$$ satisfy the desired conclusion.
\end{proof}

Given a word $\omega$, $\left| \omega \right|_1$ denotes to the number of $1$'s that occur in $\omega$. The \textit{$1$-ratio of $\omega$}, denoted by $1(\omega)$ is defined to be $1(\omega) = \dfrac{\left| \omega \right|_1}{\ell(\omega)}$.  For a sequence $\boldsymbol{\alpha} \in \Sigma_2$ we define \textit{the $1$-ratio of $\boldsymbol{\alpha}$} to be $1(\boldsymbol{\alpha}) = \mathop{\lim}\limits_{n \to \infty}\dfrac{\left|a_1 \ldots a_n\right|_1}{n}$ if such limit exists.

\vspace{1em}A word $\omega$ is said to be \textit{balanced}, if for any factors $\upsilon$ and $\nu$ of $\omega$ with $\ell(\upsilon) = \ell(\nu) = n$, $\left| \left| \upsilon \right|_1 - \left| \nu \right|_1 \right| \leq 1$. A word $\omega$ is said to be \textit{cyclically balanced} if $\omega^2$ is balanced. Notice if $\omega$ is cyclically balanced then $\omega^{\infty}$ is balanced. It is well known that if $\omega$ and $\nu$ are two cyclically balanced words with $\left|\omega\right|_{1}= \left|\nu\right|_{1} = p$ and $\ell(\omega) =\ell(\nu)=q$ with $\gcd(p,q)=1$, then $\omega$ is a cyclic permutation of $\nu$. This implies that there exists only $q$ distinct cyclically balanced words of length $q$ with $p$ $1$'s. Given $r=\frac{p}{q} \in \mathbb{Q}\cap(0,1)$ we define $\omega_r$ to be the lexicographically largest cyclically balanced word with length $q$ and $1(\omega) = r$ beginning with 0, and $\nu_r$ to be lexicographically smallest cyclically balanced word with length $q$ and $1(\nu)=r$ beginning with 1. In particular, $\omega_r = 0\hbox{\rm{-max}}_{\omega_r}$ and $\nu_r = 1\hbox{\rm{-min}}_{\omega_r}$. There is an explicit way to construct $\omega_r$ and $\nu_r$ for any given $r$ using the continued fraction expansion of $r$. This construction can be found in \cite{sidorov1, sidorov6}. 

\begin{remark}
For any $r_1, r_2 \in \mathbb{Q}\cap(0,1)$ such that $r_1 < r_2$ then $\omega_{r_1}^{\infty} \prec \omega_{r_2}^{\infty}$ and $\nu_{r_1} \prec \nu_{r_2}$.\label{laobservacion}
\end{remark}  

We say that $r_1, r_2 \in \mathbb{Q} \cap (0,1)$ with $r_2 < r_1$ are \textit{Farey neighbours} if $p_1q_2 - p_2q_1 = 1$. Given two Farey neighbours $r_1, r_2 \in \mathbb{Q} \cap (0,1)$ we say that $r_3$ is \textit{the mediant of $r_1$ and $r_2$} if $r_3 = \frac{p_1 + p_2}{q_1 + q_2}$. 

\begin{lemma}{\normalfont{\cite[Lemma 3.2]{sidorov6}}}
Consider two Farey neighbours $r_1, r_2$  with $r_2 < r_1$. Let $r_3$ the mediant for $r_1$ and $r_2$and consider $\omega_{r_1}, \omega_{r_2}, \nu_{r_1}$ and $\nu_{r_2}$. Then $\omega_{r_3} = \omega_{r_1}\omega{r_2}$, $\omega_{r_3} = \omega_{r_2}\nu_{r_1}$, $\nu_{r_3} = \nu_{r_2}\nu_{r_1}$ and $\nu_{r_3} = \nu_{r_1}\omega_{r_2}$.
\label{lemadenik} 
\end{lemma}

\subsection*{Open dynamical systems and the lexicographic world}

\noindent Now we formalise the relation between $(\Lambda_{(a,b)}, f_{(a,b)})$ and the lexicographic world $\mathcal{LW}$. Firstly, we need to mention some results obtained by Glendinning and Sidorov in \cite{sidorov1} and by Hare and Sidorov in \cite{sidorov6} in order to define the set of parameters where our study is developed. We confine ourselves to studying open dynamical systems parametrised by $(a,b) \in (\frac{1}{4},\frac{1}{2})\times(\frac{1}{2},\frac{3}{4})$ based on the following results.

\begin{lemma}{\normalfont{\cite[Lemma 1.1]{sidorov1}}}
Let $a,b \in S^1$ with $a < b$. Then:
\begin{enumerate}[i)]
 \item If $0 < a < \frac{1}{4}$ and $\frac{1}{2} < b < 1$, or $0 < a < \frac{1}{2}$ and $\frac{3}{4} < b < 1$ then $X_{(a,b)} = \{0\}$.  
 \item If $\frac{1}{2} < a < 1$ or $0 < b < \frac{1}{2}$, then $\dim_{H}X_{(a,b)} > 0$.  
\end{enumerate}
\label{acotar}
\end{lemma}

Note that if $a \in (\frac{1}{2}, 1)$ and $b = 1$ or $a = 0$ and $b \in (0, \frac{1}{2})$ then $(X_{(a,b)}, f_{(a,b)})$ is topologically conjugated to a $\beta$-shift \cite{nilsson}. As a direct consequence of this observation, it is possible to obtain a different proof of Lemma \ref{acotar} \textit{ii)} than one the presented in \cite{sidorov1}.

\vspace{1em}Let $\phi:(\frac{1}{4},\frac{1}{2}) \to (\frac{1}{2},\frac{3}{4})$ and $\chi:(\frac{1}{4},\frac{1}{2}) \to 
(\frac{1}{2},\frac{3}{4})$ be defined by:
\begin{itemize} 
\item[] $\phi(a) = \mathop{\sup}\{b \in S^1 \mid X_{(a,b)} \neq \{0\}\};$
\item[] $\chi(a) = \mathop{\sup}\{b \in S^1 \mid X_{(a,b)} \hbox{\rm{ is uncountable}}\}.$
\end{itemize}

\vspace{1em}The functions $\phi$ and $\chi$ were introduced in \cite{sidorov1} and \cite{gugu}. It is clear that $\chi(a) \leq \phi(a)$. Moreover, in \cite[Theorem 2.13]{sidorov1} the authors give an 
explicit formula to calculate them. It is worth pointing out that both $\phi$ and $\chi$ were studied symbolically and $f$ was considered as a transformation of the unit interval. Nonetheless, the results remain valid for $f$ defined in $S^1$ if $0 \notin (a,b)$. It is not our purpose to study $\phi$ and $\chi$. Nonetheless, they determine the parameter space in question. We denote $D_0 = \{(a,b) \in R \mid b \leq \phi(a)\},$ and $D_1 = \{(a,b) \in R \mid b \leq \chi(a)\}$. 

\vspace{1em}Now, we introduce the following definitions as in \cite{sidorov6}. Let $(a,b) \in D_0$. We call $n \geq 3$ \textit{bad for $(a,b)$} if every periodic orbit of period $n$ of $f$ intersects $(a,b)$. Let us denote by $B(a,b) = \{n \geq 3 \mid n \hbox{\rm{ is bad for }}(a,b)\}.$  Let $D_2$ be given by $$D_2 = \left\{(a,b) \in \left(\frac{1}{4},\frac{1}{2}\right) \times \left(\frac{1}{2}, \frac{3}{4}\right) \mid B(a,b) \hbox{\rm{ is finite }}\right\}$$ and let $D_3$ be defined by $$D_3 = \{(a,b) \in R \mid B(a,b) = \emptyset\}.$$ Hare and Sidorov in \cite{sidorov6} shown that $$D_3 \subset D_2 \subset D_1 \subset D_0.$$ Since $\{(a,b) \in (\frac{1}{4},\frac{1}{2})\times(\frac{1}{2}, \frac{3}{4}) \mid h_{top}(f_{(a,b)}) > 0\} = D_1 \setminus  \partial D_1$ \cite[Theorem 2]{gugu} and $X_{(a,b)}$ is countable for $(a,b) \in D_0 \setminus D_1$, we will restrict ourselves to studying the pairs $(a,b) \in D_1$. 

\begin{proposition}
Let $(a,b) \in D_1$. Then $\Lambda_{(a,b)}$ is topologically conjugated to $\Lambda_{(1-b,1-a)}$.\label{elhomeomorfismo}
\end{proposition}
\begin{proof}
Let $h:S^1 \to S^1$ given by $h(x) = 1-x$. Observe that $h$ is a homeomorphism of the unit circle that reverses orientation. Moreover, $$h\circ f = f\circ h = \left\{
\begin{array}{clrr}      
1-2x & \hbox{\rm{ if }} & x \in [\frac{1}{2},1] \\
2-2x & \hbox{\rm{ if }} & x \in [0,\frac{1}{2})\\
\end{array}
\right.$$

Let $x \in \Lambda_{(a,b)}$. Then, $2b-1 \leq x \leq 2a$. Note that $1-2a \leq h(x) \leq 2-2b$. Also, since $x \notin (a,b)$, $h(x) \notin (1-b, 1-a)$. Then $h(\Lambda_{(a,b)}) = \Lambda_{(1-b,1-a)}$, and our result follows.
\end{proof}

The following theorem gives a lexicographic characterisation of $\Lambda_{(a,b)}$.

\begin{theorem}
Let $(a,b) \in D_1$. Then $$\pi^{-1}(\Lambda_{(a,b)}) = \Sigma_{(\boldsymbol{\alpha},\boldsymbol{\beta})}$$ where $\boldsymbol{\alpha} = \pi^{-1}(2a)$ and $\boldsymbol{\beta} = \pi^{-1}(2b-1)$.
\label{attractor1}
\end{theorem}
\begin{proof}
Let $\boldsymbol{x} \in \pi^{-1}(\Lambda_{(a,b)})$. Note that $\pi(\boldsymbol{x}) \in \Lambda_{(a,b)}$, i.e. $$\pi(\boldsymbol{x}) \in \left[2b-1, 2a\right] \cap X_{(a,b)}.$$ Then, $f^n(\pi(\boldsymbol{x})) \notin (a,b)$ and $2b-1 < f^n(\pi(\boldsymbol{x})) < 2a$ for every $n \geq 0$. By substituting $f^n(\pi(\boldsymbol{x}))$ by $\pi(\sigma^n(\boldsymbol{x}))$ it is clear that $\sigma^n(\boldsymbol{x}) \in \Sigma_{(\boldsymbol{\alpha},\boldsymbol{\beta})}$ for every $n \geq 0$, which implies that $$\pi^{-1}(\Lambda_{(a,b)}) \subset \Sigma_{(\boldsymbol{\alpha},\boldsymbol{\beta})}.$$

Now, consider $\boldsymbol{x} \in \Sigma_{(\boldsymbol{\alpha},\boldsymbol{\beta})}.$ Then $2b-1 < \pi(\sigma^n(\boldsymbol{x})) < 2a$ for every $n \geq 0$. This implies that $f^n(\pi(\boldsymbol{x})) \in \left[2b-1, 2a\right]$ for every $n \geq 0$. Suppose that $\pi(\boldsymbol{x}) \notin \Lambda_{(a,b)}$. Then, there exists $n \geq 0$ such that $f^n(\pi(\boldsymbol{x})) \in (a,b)$. Recall that $(a,b) = (a, \frac{1}{2})\cup[\frac{1}{2},b)$. Suppose that $f^n(\pi(\boldsymbol{x})) \in (a, \frac{1}{2})$ then $f^{n+1}(\pi(\boldsymbol{x})) \in (2a,1)$. Therefore $\sigma^{n+1}(\boldsymbol{x})) \succcurlyeq \boldsymbol{\alpha}$ which contradicts our assumption on $\boldsymbol{x}$. If $f^n(\pi(\boldsymbol{x})) \in [\frac{1}{2},b)$ then $f^{n+1}(\pi(\boldsymbol{x})) \in [0,2b-1)$. Therefore, $\sigma^{n+1}(\boldsymbol{x})) \preccurlyeq \boldsymbol{\beta}$ which, again, contradicts our assumption on $\boldsymbol{x}$.
\end{proof}

Observe that the pair $(\boldsymbol{\alpha}, \boldsymbol{\beta})$ where $\boldsymbol{\alpha} = \pi^{-1}(2a)$ and $\boldsymbol{\beta} = \pi^{-1}(2b-1)$ might be extremal. Nonetheless, as a consequence of Theorem \ref{attractor1} the following theorem is true.  

\begin{theorem}
For every $(\boldsymbol{\alpha},\boldsymbol{\beta}) \in \mathcal{LW}$ there exists $(a,b) \in D_1$ such that $(\Lambda_{(a,b)}, f_{(a,b)})$ is topologically conjugated to $(\Sigma_{(\boldsymbol{\alpha},\boldsymbol{\beta})},\sigma_{(\boldsymbol{\alpha},\boldsymbol{\beta})})$. 
\label{asociar}
\end{theorem}

\section{Genericity Results and the extended Lexicographic World}
\label{generic}

\noindent In this section we describe the pairs $(a,b) \in D_1$ such that the corresponding exclusion subshift is a \textit{subshift of finite type}. In Theorem \ref{opendense1} we show that for almost every $(a,b) \in D_1$ the corresponding attractor $(\Lambda_{(a,b)}, f_{(a,b)})$ is topologically conjugated to a subshift of finite type. This implies directly that the function which associates to every $(a,b) \in D_1$ the topological entropy of $(\Lambda_{(a,b)}, f_{(a,b)})$ is locally constant - see Corollary \ref{opendense}.

\vspace{1em} As we observed in Section \ref{basic}, the pair $(\pi^{-1}(2a), \pi^{-1}(2b-1))$ can be a extremal pair. This corresponds to intervals $(a,b)$ such that $a$ or $b$ fall into the hole under iterations.  In Theorem \ref{asociar1} it we prove that there is always a pair $(\boldsymbol{\alpha}, \boldsymbol{\beta}) \in \mathcal{LW}$ which describes completely the dynamics of $(\Lambda_{(a, b)}, f_{(a,b)})$ in such situation. Observe that Theorem \ref{asociar1} also extends the lexicographic world defined in \cite{gan} and \cite{gugu}. As a corollary, we obtain that $(\Lambda_{(a,b)}, f_{(a,b)})$ is always conjugated to an element of the lexicographical world. 

\subsection*{Subshifts of finite type and topological entropy}  Recall that a subshift $(A, \sigma_{A})$ is a \textit{subshift of finite type} if the forbidden set of factors $\mathcal{F}$ corresponding to $A$ is finite. Recall that for every $(a,b) \in D_1$ we associate the pair $(\boldsymbol{\alpha},\boldsymbol{\beta}) \in \Sigma^{1}_2 \times \Sigma^{0}_2$ where $\boldsymbol{\alpha} = \pi^{-1}(2a)$ and $\boldsymbol{\beta} = \pi^{-1}(2b-1)$. Let $$S = \left\{(a,b) \in D_1 \mid \hbox{\rm{ there exist }} n \hbox{\rm{ and }} m \hbox{\rm{ such that }} f^n(a), f^m(b) \in (a,b)\right\}.$$   

\begin{lemma}
The set $S$ is open and dense in $D_1$.\label{dense}
\end{lemma}
\begin{proof}
Let $(a,b) \in S$ and recall that $f$ denotes the doubling map. Let $$n = \min\left\{j \in \mathbb{N} \mid f^j(a) \in 
(a,b)\right\}$$ and $$m = \min \left\{i \in \mathbb{N} \mid f^i(b) \in (a,b) \right\}.$$ Recall that $f^j$ is continuous for every $j \geq 0$. Let $\varepsilon^{\prime} > 0$. Then there exists $\varepsilon_a > 0$ such that if $d(a, a^{\prime}) \leq \varepsilon_a$ then $d(f^n(a), f^n(a^{\prime})) < \varepsilon^{\prime}$. Similarly, there exists $\varepsilon_b > 0$ such that if $d(b, b^{\prime}) \leq \varepsilon_b$ then $d(f^m(b), f^m(b^{\prime})) < \varepsilon^{\prime}$. Let $\varepsilon = \mathop{\min}\{\varepsilon_a, \varepsilon_b\}.$ Then 
$B_{\varepsilon}((a,b)) \subsetneq S$. Therefore $S$ is open. 

To show that $S$ is dense consider $(a,b) \in D_1 \setminus \partial D_1$. Let $\varepsilon > 0$ such that  $\frac{1}{2} \notin (a-\varepsilon, a + \varepsilon)$ and $\frac{1}{2} \notin (b- \varepsilon,b+\varepsilon)$. Moreover, since $S$ is open without losing generality we can assume that $(a-\varepsilon, a + \varepsilon)\times (b-\varepsilon,b+\varepsilon) \subset D_1\setminus \partial D_1$. Recall that the Lebesgue measure of $(a-\varepsilon, a + \varepsilon)$ and $(b- \varepsilon, b+\varepsilon)$ is positive. Then, by the Poincar\'e recurrence theorem (see e.g. \cite[Theorem 2.2]{mane}), there exists $a^{\prime} \in (a-\varepsilon, a + \varepsilon)$ and $j \in\mathbb{N}$ such that $f^j(a^{\prime}) \in (a-\varepsilon, a + \varepsilon)$. Similarly, there exists a point $b^{\prime} \in (b-\varepsilon, b+\varepsilon)$ and $i \in \mathbb{N}$ such that $f^i(b^{\prime}) \in (b-\varepsilon, b+\varepsilon)$. Since $f^j$ is continuous for every $j \geq 0$ there exist $\varepsilon_a > 0$ and $\varepsilon_b > 0$ such that for every $a^{\prime \prime} \in (a^{\prime}-\varepsilon_a, a^{\prime}_a+\varepsilon)$, $f^j(a^{\prime \prime}) \in (a^{\prime}-\varepsilon_a, a^{\prime}+\varepsilon_a)$ and for every $b^{\prime \prime} \in (b^{\prime}-\varepsilon_b, b^{\prime}_b+\varepsilon)$, $f^i(b^{\prime \prime}) \in (b^{\prime}-\varepsilon_b, b^{\prime}+\varepsilon_b)$.  Then $(a^{\prime \prime},b^{\prime \prime}) \in S$ and $d((a,b),(a^{\prime\prime},b^{\prime\prime})) < \varepsilon$. Hence, $S$ is dense.   
\end{proof}

In \cite[Proposition 4.1]{bundfuss} and \cite[Theorem 2.4]{simon2} it is shown that if $(a,b) \in S$ then $(\Sigma_{(\boldsymbol{\alpha},\boldsymbol{\beta})},\sigma_{(\boldsymbol{\alpha},\boldsymbol{\beta})})$ is a subshift of finite type. We can see immediately that for every $(\boldsymbol{\alpha},\boldsymbol{\beta}) \in N \times \bar {N}$ there is a pair $(a,b)\in S$ such that $\pi^{-1}(2a) = \boldsymbol{\alpha}$ and $\pi^{-1}(2b-1) = \boldsymbol{\beta}$. Then Lemma \ref{dense} and \cite[Theorems 3.7, 3.8]{nilsson} imply that $D_1 \cap S$ is a set of full Lebesgue measure. Thus, we have proved the following theorem.  

\begin{theorem}
The set $$T = \left\{(a,b) \in D_1 \mid (\Sigma_{(\boldsymbol{\alpha},\boldsymbol{\beta})}, \sigma_{(\boldsymbol{\alpha},\boldsymbol{\beta})}) \hbox{\rm{ is a subshift of finite type}}\right\}$$ is open and dense. Moreover, it has full Lebesgue measure.\label{opendense1}
\end{theorem}

It is worth pointing out that there exist pairs $(\boldsymbol{\alpha},\boldsymbol{\beta}) \in P \times \bar{P}$ such that $(\pi(0\boldsymbol{\alpha}),\pi(1\boldsymbol{\beta})) \in S$. Such pairs are called \textit{two sided extremal} -see Definition \ref{tiposdeextremos}. 
  
\begin{theorem}
For every $(a,b) \in S$ there is an open set $U \subset D_1$ such that $(a,b) \in U$ and $(X_{(a,b)}, f_{(a,b)})  = 
(X_{(a^{\prime},b^{\prime})}, f_{(a^{\prime}, b^{\prime})})$ for every $(a^{\prime},b^{\prime}) \in U$ \footnote{\normalfont{After completion of this proof the author was made aware of the work of 
Baker, Dajani and Jiang \cite[Theorem 2.4]{simon2}. Our result follows immediately from the proof of the mentioned theorem.}}. 
\label{devil1}
\end{theorem}
\begin{proof}
The argument is essentially the same as the one of the proof of \cite[Theorem 1]{urbanski1} which proves the case when $a = 0$. Let $(a,b) \in S$ be fixed. Then, there exist $n = n(a)$ and $m = m(b) \in \mathbb{N}$ such that $f^n(a), f^m(b) \in (a,b)$. Without loosing generality, assume that such $n$ and $m$ are minimal. Consider $$C_1 = \left\{b^{\prime} \in \left(\frac{1}{2}, \chi(a)\right) \mid f^m(b^{\prime}) \in (a,b)\right\}.$$ Observe that $C_1 \neq \emptyset$ since $b \in C_1$. Since $f$ is continuous, there exists $\varepsilon > 0$ such that for any $b^{\prime} \in \overline{B_{\varepsilon}(b)}$, $f^m(b^{\prime}) \in (a,b^{\prime})$. Note that for any $c \in B_{\varepsilon}(b)$, $f^m(c) \in (a, b-\frac{\varepsilon}{2})$. Recall that $X_{(a, b+\frac{\varepsilon}{2})} \subset X_{(a,b)} \subset X_{(a, b-\frac{\varepsilon}{2})}$. Let $x \in S^1 \setminus X_{(a, b+\frac{\varepsilon}{2})}$. Then there exists $m \in \mathbb{N}$ such that $f^m(x) \in (a, b+\frac{\varepsilon}{2})$. If $x \in (b-\frac{\varepsilon}{2}, b + \frac{\varepsilon}{2})$ then there exists $n \in \mathbb{N}$ such that $f^n(f^m(x)) \in (a, b - \frac{\varepsilon}{2})$. This shows that $X_{(a, b-\frac{\varepsilon}{2})} \subset X_{(a, b+ \frac{\varepsilon}{2})}$ and $$X_{(a, b-\frac{\varepsilon}{2})} = X_{(a, b+ \frac{\varepsilon}{2})} = X_{(a,b)}.$$ Similarly, we consider $$C_2 = \left\{a^{\prime} \in \left(\frac{1}{4},\frac{1}{2}\right) \mid f^n(a^{\prime}) \in (a,b)\right\}.$$ Then, we can show in a similar way that there exists $\varepsilon^{\prime} > 0$ such that 
$X_{(a - \frac{\varepsilon^{\prime}}{2}, b)} \subset X_{(a + \frac{\varepsilon^{\prime}}{2}, b)}$ and $$X_{(a-\frac{\varepsilon^{\prime}}{2}, b)} = X_{(a + \frac{\varepsilon^{\prime}}{2}, b)} = X_{(a,b)}.$$ Then, $U = (a-\epsilon, a+\epsilon)\times(b-\epsilon, b+\epsilon)$ where $\epsilon = \min\{\frac{\varepsilon}{2},\frac{\varepsilon^{\prime}}{2}\}$.
\end{proof}    

For every $(a,b) \in D_1$ it is natural to associate it the topological entropy of the corresponding attractor $\Lambda_{(a,b)}$. To formalise this, let $H:D_1 \to [0,1]$ given by $$H((a,b)) = h_{top}(f_{(a,b)}).$$ From \cite[Theorem 4]{urbanski2} we have that $H$ is a continuous function. Furthermore, note that for every $a \in [0,\frac{1}{2})$, if $b = a$ then $X_{(a,b)} = S^1$ and $h_{top}(f_{(a,b)}) = 1$. Moreover, if $b = \chi(a)$ then $h_{top}(f_{(a,b)}) = 0$ as a direct consequence of \cite[Theorem 2.16]{sidorov1}. An immediate consequence of Theorem \ref{devil1} is that $H$ is constant almost everywhere.

\begin{corollary}
The set $$T = \left\{(a,b) \in R \mid h_{top}(f_{(a,b)}) \hbox{\rm{ is constant}}\right\}$$ is open and dense. Moreover, it has full measure.   
\label{opendense}
\end{corollary}

As a result of Corollary \ref{opendense} and the well know formula $$\dim_HX_{(a,b)} = \dfrac{h_{top}(f_{(a,b)})}{\lambda},$$ where $\lambda$ is the Lyapunov exponent of $2x\mod 
1$, we obtain the following corollary. The proof is omitted.

\begin{corollary}
The set of pairs $(a, b) \in D_1$ such that $\dim_HX_{(a,b)}$ is constant is open and dense with full measure. 
\end{corollary}

\subsection*{The extended Lexicographic World}
\label{extremal}

Recall that Theorem \ref{attractor1} implies that for every $(\boldsymbol{\alpha},\boldsymbol{\beta}) \in \mathcal{LW}$, $\pi^{-1}(\Lambda_{(\pi(0\boldsymbol{\alpha}),\pi(1\boldsymbol{\beta}))}) = \Sigma_{(\boldsymbol{\alpha},\boldsymbol{\beta})}$. The rest of the section is devoted to prove Theorem {asociar1}, that is, for every $(\boldsymbol{\alpha}, \boldsymbol{\beta})\in \Sigma^1_2 \times \Sigma^0_2$ there is a pair $(\boldsymbol{\alpha}^{\prime}, \boldsymbol{\beta}^{\prime}) \in \mathcal{LW}$ such that $(\Sigma_{(\boldsymbol{\alpha}, \boldsymbol{\beta}}, \sigma_{(\boldsymbol{\alpha}, \boldsymbol{\beta})}) = (\Sigma_{(\boldsymbol{\alpha}^{\prime}, \boldsymbol{\beta}^{\prime})}, \sigma_{(\boldsymbol{\alpha}^{\prime}, \boldsymbol{\beta}^{\prime})})$. Then we obtain that for every $(a,b) \in D_1$ there is a unique element $(\boldsymbol{\alpha},\boldsymbol{\beta}) \in \mathcal{LW}$ such that $\Sigma_{(\boldsymbol{\alpha},\boldsymbol{\beta})} = \pi^{-1}(\Lambda_{(a,b)})$ (Corollary \ref{resultadazo}). 

\vspace{1em}Observe that Lemma \ref{dense} implies that for almost every $(a,b) \in D_1$, the pair $(\boldsymbol{\alpha},\boldsymbol{\beta})$ where $\boldsymbol{\alpha} = \pi^{-1}(2a)$ and $\boldsymbol{\beta} = \sigma(\pi^{-1}(2b-1)))$ is extremal. Moreover, given $(\boldsymbol{\alpha},\boldsymbol{\beta}) \in P \times \bar{P}$, $(\boldsymbol{\alpha},\boldsymbol{\beta})$ might be extremal. As mentioned before, extremal pairs are elements of $\Sigma^1_2 \times \Sigma^0_2$ corresponding under projection to the end points of $(a,b) \in S$ or $(a,b) \in D_1 \setminus S$ such that there exist $n(a), m(b) \in \mathbb{N}$ such that either $f^{n(a)}(a) \in (a,b)$, or $f^{m(b)}(b) \in (a,b)$. The possible cases of extremal pairs $(\boldsymbol{\alpha},\boldsymbol{\beta}) \in P \times \bar{P}$ are defined as follows:

\begin{definition}
Let $(\boldsymbol{\alpha},\boldsymbol{\beta})\in P \times \bar P$.  
\begin{itemize}
\item  We say that $(\boldsymbol{\alpha},\boldsymbol{\beta})$ is \textit{two sided extremal} if:
\begin{enumerate}[$i)$]
\item $\sigma^n(\boldsymbol{\beta}) \succ \boldsymbol{\alpha}$ for some $n \in \mathbb{N}$;
\item $\sigma^m(\boldsymbol{\alpha}) \prec \boldsymbol{\beta}$ for some $m \in \mathbb{N}$.
\end{enumerate}

\item We say that $(\boldsymbol{\alpha},\boldsymbol{\beta})$ is \textit{right extremal} if:
\begin{enumerate}[$i)$]
\item $\sigma^m(\boldsymbol{\alpha}) \succ \boldsymbol{\beta}$ for every $m \in \mathbb{N}$;
\item $\sigma^n(\boldsymbol{\beta}) \succ \boldsymbol{\alpha}$ for some $n \in \mathbb{N}$.
\end{enumerate}

\item We say that $(\boldsymbol{\alpha},\boldsymbol{\beta})$ is \textit{left extremal} if:
\begin{enumerate}[$i)$]
\item $\sigma^n(\boldsymbol{\beta}) \prec \boldsymbol{\alpha}$ for every $n \in \mathbb{N}$;
\item $\sigma^m(\boldsymbol{\alpha}) \prec \boldsymbol{\beta}$ for some $m \in \mathbb{N}$.
\end{enumerate}
\end{itemize}
\label{tiposdeextremos}
\end{definition}
 
In \cite[Theorem 1.3]{samuel} it is shown that if $(\boldsymbol{\alpha},\boldsymbol{\beta}) \in \mathcal{LW}$ satisfies that $\left(\Sigma_{(\boldsymbol{\alpha},\boldsymbol{\beta})}, \sigma_{(\boldsymbol{\alpha},\boldsymbol{\beta})}\right)$ is a subshift of finite type then $\boldsymbol{\alpha}$ and $\boldsymbol{\beta}$ are periodic sequences and vice versa. Taking this into account, a suitable way to associate an element of the lexicographic world to every $(\boldsymbol{\alpha},\boldsymbol{\beta}) \in \Sigma_2^1 \times \Sigma_2^0$ is needed. The idea is to find a pair of periodic sequences $({\boldsymbol{\alpha}}^{\prime},{\boldsymbol{\beta}}^{\prime}) \in \mathcal{LW}$ which reflects faithfully the dynamical behaviour of the subshift associated to the extremal pair $(\boldsymbol{\alpha},\boldsymbol{\beta})$. Firstly, in Lemma \ref{fuera1} we show that every $(a,b) \in D_1$ can be represented by an element of $P \times \bar P$ where $P$ is the set of Parry sequences.   

\begin{proposition}
The set $P \cap Per(\sigma)$ is dense in $P$.\label{denseperiodicparry}
\end{proposition}
\begin{proof}
Let $\boldsymbol{\alpha} \in P$ and $\varepsilon > 0$. Firstly, note that if $\boldsymbol{\alpha}$ is a finite sequence, then the sequence ${\boldsymbol{\alpha}}^{\prime} = (a_1 \ldots a_{\ell(a)}0^{[\frac{1}{\varepsilon}]+1})^{\infty}$ is a Parry sequence with $d(\boldsymbol{\alpha}, 
{\boldsymbol{\alpha}}^{\prime}) < \varepsilon$. Assume that $\boldsymbol{\alpha}$ is not finite. Without losing generality we can assume that $\boldsymbol{\alpha} \notin P \cap Per(\sigma)$. Consider $n \in \mathbb{N}$ such that $\frac{1}{2^n} < \varepsilon$, $a_n = 0$ and $a_{n+1} = 1$. Consider the sequence ${\boldsymbol{\alpha}}^{\prime} = (a_1 \ldots a_{n})^{\infty}$. Observe that $d(\boldsymbol{\alpha}, {\boldsymbol{\alpha}}^{\prime}) \leq \frac{1}{2^n} < \varepsilon$ and ${\boldsymbol{\alpha}}^{\prime}$. Thus, it suffices to show that ${\boldsymbol{\alpha}}^{\prime} \in P$. Assume, on the contrary that ${\boldsymbol{\alpha}}^{\prime} \notin P$. Then, there exists $m \in \mathbb{N}$ such that $\sigma^{m}({\boldsymbol{\alpha}}^{\prime}) \succ {\boldsymbol{\alpha}}^{\prime}$. Since ${\boldsymbol{\alpha}}^{\prime}$ is periodic, we can be sure that $m \in \{1, \ldots, n\}$. This implies that $\sigma^{m}({\boldsymbol{\alpha}}^{\prime}) \succ {\boldsymbol{\alpha}}^{\prime}$. Therefore, $\sigma^m(\boldsymbol{\alpha}) \succ \boldsymbol{\alpha}$ which is a contradiction. Therefore ${\boldsymbol{\alpha}}^{\prime} \in P$ and hence $P \cap Per(\sigma)$ is dense in $P$.
\end{proof}

Let $N = \Sigma_2 \setminus (P \cup \bar P)$. As we observed previously, $\pi(N)$ has full Lebesgue measure. 

\begin{proposition}
$N$ is an open and dense subset of $\Sigma_2$. \label{propertiesofn}
\end{proposition}
\begin{proof}
Firstly, we show that $N$ is an open subset of $\Sigma_2$. Let $\boldsymbol{x} \in N$ such that $\boldsymbol{x} \in \Sigma^1_2$. Note that, $\boldsymbol{x} \in [1^n,1^{n+1}]_{\prec}$ for some $n \in \mathbb{N}$. Let $j = \min \left\{i \in \mathbb{N} \mid \sigma^i(\boldsymbol{x}) \succ \boldsymbol{x}\right\}$ and $n_j = \min\left\{n \in \mathbb{N} \mid \sigma^j(\boldsymbol{x})_n \succ x_{j+n}\right\}$. Consider the set 
\begin{align*}
N(\boldsymbol{x}) = \left\{\boldsymbol{y} \in \Sigma_2 \mid y_i = x_i \hbox{\rm{ for every }}  1 \leq i \leq  n_j\right\}.
\end{align*}
Observe that $N(\boldsymbol{x}) \subset N$ and $\boldsymbol{x} \in N(\boldsymbol{x})$. Let us show that $N(\boldsymbol{x})$ is open. Observe that for every $\boldsymbol{\alpha} \in N(\boldsymbol{x})$ $d(\boldsymbol{x},\boldsymbol{\alpha}) < \frac{1}{2^{n_j+1}}$. Let $\varepsilon > 0$ sufficiently small, that is $\varepsilon < \frac{1}{2^{n_j+1}}$. Then, there exists $k \in \mathbb{N}$ such that $\frac{1}{2^k} < \varepsilon < \frac{1}{2^{n_j+1}}$. Let $N_{\varepsilon}(\boldsymbol{x}) = \left\{\boldsymbol{y} \in N(\boldsymbol{x}) \mid y_i = x_i \hbox{\rm{ for every }} 1 \leq i \leq k\right\}$. Then, $d(\boldsymbol{x}, \boldsymbol{y}) < \varepsilon$ for every $\boldsymbol{y} \in N_{\varepsilon}(\boldsymbol{x})$. Thus, $N(\boldsymbol{x})$ is an open set. Observe for $\boldsymbol{x}$ such that $x \in \Sigma_2^0 \cap N$ we can do a similar construction. Then we can be sure that $N$ is an open set.

In order to show that $N$ is dense, let $\boldsymbol{x} \in \Sigma^{1}_2$ and $\varepsilon > 0$. Observe that $\boldsymbol{x} \in [1^n,1^{n+1}]_{\prec}$ for some $n \in \mathbb{N}$ and recall there exists $j \in \mathbb{N}$ such that $\frac{1}{2^j} < \varepsilon$. Then, $\boldsymbol{\alpha} \in N$ given by $\boldsymbol{\alpha} = x_1 \ldots x_j 0 1^{n+2}0^{\infty}$ satisfies that $d(\boldsymbol{x},\boldsymbol{\alpha}) < \varepsilon$. A similar construction can be done for $\boldsymbol{x} \in \Sigma^{0}_2$ such that $x_1=0$.  Therefore $N$ is a dense subset of $\Sigma_2$.
\end{proof}

Let $N_0 = \{\boldsymbol{x} \in \Sigma_2^0 \cap N\}$ and $N_1 =  \{\boldsymbol{x} \in \Sigma_2^1 \cap N\}$. Note that $N = N_0 \cup N_1$. For every $\boldsymbol{\alpha} \in N_1$ consider $$n_{\boldsymbol{\alpha}} = \mathop{\min}\{n \in \mathbb{N} \mid \sigma^n(\boldsymbol{\alpha}) \succ \boldsymbol{\alpha}\},$$ and for every $\boldsymbol{\beta} \in N_0$ consider $$m_{\boldsymbol{\beta}} = \mathop{\min}\{n \in \mathbb{N} \mid \sigma^n(\boldsymbol{\beta}) \prec \boldsymbol{\beta}\}.$$ Note that both $n_{\boldsymbol{\alpha}}$ and $n_{\boldsymbol{\beta}}$ exist for every $\boldsymbol{\alpha} \in N_1$ and $\boldsymbol{\beta} \in N_0$ respectively. 

\vspace{1em}We define $\varsigma: \Sigma^1_2 \to \Sigma^1_2$ by $$\varsigma(\boldsymbol{\alpha}) = \left\{
\begin{array}{clrr}      
\boldsymbol{\alpha} & \hbox{\rm{ if }}  \boldsymbol{\alpha} \in P\\
(a_1 \ldots a_{n_a-1}0)^{\infty} & \hbox{\rm{ otherwise. }}\\
\end{array}
\right.$$

\vspace{1em}Observe that $\varsigma$ is well defined, since $n_{\boldsymbol{\alpha}}$ exists for every $\boldsymbol{\alpha} \in N$. Moreover, $n_{\boldsymbol{\alpha}}$ is unique. Also, $\varsigma(\Sigma^1_2) = P$ and $\varsigma(\boldsymbol{\alpha}) \preccurlyeq \boldsymbol{\alpha}$ for every $\boldsymbol{\alpha} \in \Sigma^1_2$. Note that it is possible to define $\varsigma^{\prime}:\Sigma^0_2 \to \Sigma^0_2$ as 
$\varsigma^{\prime}(\boldsymbol{\beta}) = \overline{\varsigma(\bar{\boldsymbol{\beta}})}$. 

\begin{lemma}
Let $\boldsymbol{\alpha} \in P \cap Per(\sigma)$. Then, $\varsigma$ is constant on intervals of the form 
$$\left[\boldsymbol{\alpha}, a_1 \ldots a_{n}1^{\infty}\right]_{\prec}$$ where $n$ is the period of $\boldsymbol{\alpha}$. Moreover, if ${\boldsymbol{\alpha}}^{\prime} \notin [\boldsymbol{\alpha}, a_1 \ldots a_{n}1^{\infty}]_{\prec}$ then $\varsigma(\boldsymbol{\alpha}) \neq \varsigma ({\boldsymbol{\alpha}}^{\prime})$. \label{varsigma}
\end{lemma}
\begin{proof}
It is clear that $\varsigma(\boldsymbol{\alpha}) = \boldsymbol{\alpha}$ and $\varsigma(a_1 \ldots a_{n)}1^{\infty}) = \boldsymbol{\alpha}$ from the definition of $\varsigma$. Let ${\boldsymbol{\alpha}}^{\prime} \in (\boldsymbol{\alpha}, a_1 \ldots a_{n}1^{\infty})_{\prec}$. Then there exists $m > n$ such that $a^{\prime}_m = 1$ and $a_m =0$. This implies that ${\boldsymbol{\alpha}}^{\prime} \notin P$. Consider $\varsigma({\boldsymbol{\alpha}}^{\prime})$ and assume that $\varsigma({\boldsymbol{\alpha}}^{\prime}) \neq \boldsymbol{\alpha}$. Suppose first that $\varsigma({\boldsymbol{\alpha}}^{\prime}) \prec \boldsymbol{\alpha}$. Then, there exists $m^{\prime}$ such that $\varsigma({\boldsymbol{\alpha}}^{\prime})_{m^{\prime}} < a_{m^{\prime}}$. This implies that ${a}^{\prime}_1 \ldots a^{\prime}_{n_{a^{\prime}}} \prec \boldsymbol{\alpha}$ which contradicts that $\boldsymbol{\alpha} \in (\boldsymbol{\alpha}, a_1 \ldots a_{n}1^{\infty})_{\prec}$. If $\varsigma({\boldsymbol{\alpha}}^{\prime}) \succ \boldsymbol{\alpha}$ then $a^{\prime}_1 \ldots a^{\prime}_{n_{a^{\prime}}} \succ \boldsymbol{\alpha}$ which again contradicts that ${\boldsymbol{\alpha}}^{\prime} \in (\boldsymbol{\alpha}, a_1 \ldots a_{n}1^{\infty})_{\prec}$. 
\end{proof}

\begin{remark}
In the proof of Lemma \ref{varsigma} we considered the sequence $a_1 \ldots a_{n}1^{\infty}$ instead of $a_1 \ldots a_{n-1}10^{\infty}$. The purpose of this modification is to not change the provided definition of $\varsigma$. However, if $a_1 \ldots a_{n-1}10^{\infty}$ is considered, then $\varsigma(a) = (a_{1} \ldots a_{\ell(a)-1}0)^{\infty}$ (see \cite[Definition 2.1]{sidorov5}).  \label{observaciondevarsigma}
\end{remark}

An immediate consequence of Lemma \ref{varsigma} is that $\varsigma\mid_{P}$ is a decreasing function with respect to the lexicographic order. Also, note that $\varsigma^{\prime}$ will satisfy the same properties as $\varsigma$, except that $\varsigma^{\prime}$ is an increasing function.

\begin{theorem}
The function $\varsigma$ is continuous. Moreover, $\pi \circ \varsigma \circ \pi^{-1}:[\frac{1}{2}, 1] \to [\frac{1}{2}, 1]$ is a devil's staircase.\label{varsigmacontinuous}
\end{theorem}
\begin{proof}
From Lemma \ref{varsigma}, $\varsigma$ is continuous and constant on intervals of the form $[\boldsymbol{\alpha}, a_1 \ldots a_{n}1^{\infty}]_{\prec}$ where $n = Per(\boldsymbol{\alpha})$. Observe that if ${\boldsymbol{\alpha}}^{\prime} \in \Sigma^1_2 \setminus \mathop{\bigcup}\limits_{\boldsymbol{\alpha} \in P \cap Per(\sigma)} [\boldsymbol{\alpha}, a_1 \ldots a_{n}1^{\infty}]_{\prec}$ then $\varsigma({\boldsymbol{\alpha}}^{\prime}) = \boldsymbol{\alpha}^{\prime}$. Moreover, for every ${\boldsymbol{\alpha}}^{\prime} \in \Sigma_2^1 \setminus P$ there exists $\boldsymbol{\alpha} \in P\cap Per(\sigma)$ such that ${\boldsymbol{\alpha}}^{\prime} \in [\boldsymbol{\alpha}, a_1 \ldots a_{n}1^{\infty}]_{\prec}$. Then from Proposition \ref{propertiesofn} $\varsigma$ is continuous and constant on a dense set of $\Sigma^1_2$. Then, it is just needed to show that $\varsigma$ is continuous in $P \setminus P\cap Per(\sigma)$. Let $\boldsymbol{\alpha} \in P \setminus P\cap Per(\sigma)$ and $\{\boldsymbol{\alpha}_i\}_{i=1}^{\infty}$ such that $\boldsymbol{\alpha}_i \mathop{\longrightarrow}\limits_{i \to \infty} \alpha$. Let $\varepsilon > 0$. Then, there exists $k \in \mathbb{N}$ such that $d(\boldsymbol{\alpha}_i,\boldsymbol{\alpha}) < \frac{1}{2^{k+1}} < \frac{\varepsilon}{2}$. Recall that $d(\boldsymbol{\alpha}_n,\varsigma(\boldsymbol{\alpha}_i)) \leq \frac{1}{2^{n_{\alpha_i}}}$. Also, since $\boldsymbol{\alpha}_i$ converge to $\boldsymbol{\alpha}$ then $n_{\boldsymbol{\alpha}_i} > k$ for every $i \geq k$. This gives that $$d(\varsigma(\boldsymbol{\alpha}_i),\varsigma(\boldsymbol{\alpha})) \leq \frac{1}{2^{n_{\boldsymbol{\alpha}_i}}} + \frac{\varepsilon}{2} < \varepsilon.$$  Thus, $\varsigma$ is continuous.

Let $a \in [\frac{1}{2}, 1]$ such that $\pi^{-1}(\{a\})$ has two elements. By Remark \ref{observaciondevarsigma} we can consider $\boldsymbol{x} \in \pi^{-1}(\{a\})$ such that $\boldsymbol{x} = x_1 \ldots x_n 1^{\infty}$. Then, $\pi \circ \varsigma \circ \pi^{-1}$ is a well defined function and it is continuous, increasing, constant on $\pi([[\boldsymbol{\alpha}, a_1 \ldots a_{\ell(a)}1^{\infty}]_{\prec}])$. Moreover, since $\pi(P)$ is a set of Lebesgue measure zero, $$\Sigma^1_2 \setminus \mathop{\bigcup}\limits_{\boldsymbol{\alpha} \in P \cap Per(\sigma)}[\boldsymbol{\alpha}, a_1 \ldots a_{\ell(a)}1^{\infty}]_{\prec}$$ has Lebesgue measure zero and $\pi \circ\varsigma \circ \pi^{-1}$ is not differentiable in $$\pi\left(\Sigma^1_2 \setminus \mathop{\bigcup}\limits_{\boldsymbol{\alpha} \in P \cap Per(\sigma)} [\boldsymbol{\alpha}, a_1 \ldots a_{\ell(a)}1^{\infty}]_{\prec}\right).$$ Thus, $\pi \circ \varsigma \circ \pi^{-1}$ is a devil's staircase.
\end{proof}

\begin{lemma}
Let $(\boldsymbol{\alpha},\boldsymbol{\beta}) \in N_1 \times N_0 \setminus P \times \bar{P}$ such that $(\boldsymbol{\alpha},\boldsymbol{\beta})$ is extremal, $\boldsymbol{\alpha} \in N$ and $\boldsymbol{\beta} \in \bar N$ and $\sigma^{n}(\boldsymbol{\beta}) \prec \boldsymbol{\alpha}$ and $\sigma^{n}(\boldsymbol{\alpha}) \succ \boldsymbol{\beta}$ for every $n \in \mathbb{N}$. Then $\Sigma_{(\boldsymbol{\alpha},\boldsymbol{\beta})} = \Sigma_{(\varsigma(\boldsymbol{\alpha}), \varsigma^{\prime}(\boldsymbol{\beta}))}.$\label{fuera1}
\end{lemma}
\begin{proof}
Note that it suffices to show our result for $(\boldsymbol{\alpha},\boldsymbol{\beta}) \in N_0 \times N_1$. Observe that $\boldsymbol{\beta} \prec \varsigma^{\prime}(\boldsymbol{\beta}) \prec \varsigma(\boldsymbol{\alpha}) \prec \boldsymbol{\alpha}$. Then $\Sigma_{(\varsigma(\boldsymbol{\alpha}), \varsigma^{\prime}(\boldsymbol{\beta}))} \subset \Sigma_{(\boldsymbol{\alpha},\boldsymbol{\beta})}.$ Assume that $\Sigma_{(\varsigma(\boldsymbol{\alpha}),\varsigma^{\prime}(\boldsymbol{\beta}))} \subsetneq \Sigma_{(\boldsymbol{\alpha},\boldsymbol{\beta})}.$ Let $\boldsymbol{x} \in \Sigma_{(\boldsymbol{\alpha},\boldsymbol{\beta})} \setminus \Sigma_{(\varsigma(\boldsymbol{\alpha}), \varsigma^{\prime}(\boldsymbol{\beta}))}$. Then $$\sigma^n(\boldsymbol{x}) \in \left[\boldsymbol{\beta}, \varsigma^{\prime}(\boldsymbol{\beta})\right)_{\prec} \cup (\varsigma(\boldsymbol{\alpha}),\boldsymbol{\alpha}]_{\prec} \hbox{ for every } n \geq 0.$$ Then there exists $m^{\prime}$ such that either $\sigma^{m^{\prime}}(\boldsymbol{x})_{m_b} < (\varsigma^{\prime}(\boldsymbol{\beta}))_{m_b}$ or $\sigma^{m^{\prime}}(\boldsymbol{x})_{n_a} > (\varsigma(\boldsymbol{\alpha}))_{n_a}$. This implies that $\sigma^{m_b + j}(\boldsymbol{x}) \prec \boldsymbol{\beta}$ or $\sigma^{n_a+j}(\boldsymbol{x}) \succ \boldsymbol{\alpha}$, which is a contradiction. Then $\Sigma_{(\boldsymbol{\alpha},\boldsymbol{\beta})} = \Sigma_{(\varsigma(\boldsymbol{\alpha}),\varsigma^{\prime}(\boldsymbol{\beta}))}$.   
\end{proof}

Now we will study two sided extremal pairs $(\boldsymbol{\alpha},\boldsymbol{\beta}) \in P \times \bar P$. Let $(\boldsymbol{\alpha},\boldsymbol{\beta}) \in P \times \bar P$ be two sided extremal. Consider $$M_{\boldsymbol{\alpha}}(\boldsymbol{\beta})= \min\{m \in \mathbb{N} \mid \sigma^m(\boldsymbol{\beta}) \succ \boldsymbol{\alpha}\},$$ and $$N_{\boldsymbol{\beta}}(\boldsymbol{\alpha})= \min\{n \in \mathbb{N} \mid \sigma^n(\boldsymbol{\alpha}) \prec \boldsymbol{\beta}\}.$$ Then we define $$\iota(\boldsymbol{\alpha}) = (a_1 \ldots a_{N_{\boldsymbol{\beta}}(\boldsymbol{\alpha}) -1 }0)^{\infty}$$ and $$\iota(\boldsymbol{\beta})= (b_1 \ldots b_{M_{\boldsymbol{\alpha}}(\boldsymbol{\beta}) - 1}1)^{\infty}.$$ 

\vspace{1em}Observe that $\iota(\boldsymbol{\alpha})$ and $\iota(\boldsymbol{\beta})$ are periodic sequences. Also, $\iota(\boldsymbol{\alpha})$ and $\iota(\boldsymbol{\beta})$ satisfy that $\iota(\boldsymbol{\alpha}) \prec \boldsymbol{\alpha}$, $\boldsymbol{\beta} \succ \iota(\boldsymbol{\beta})$,  $d(\boldsymbol{\alpha}, \iota(\boldsymbol{\alpha})) \leq \frac{1}{2^{N_{\boldsymbol{\beta}}(\boldsymbol{\alpha})}}$ and $d(\boldsymbol{\beta},\iota(\boldsymbol{\beta})) \leq \frac{1}{2^{M_{\boldsymbol{\alpha}}(\boldsymbol{\beta})}}$.

\begin{proposition}
Let $(\boldsymbol{\alpha},\boldsymbol{\beta}) \in P \times \bar P$. Then $(\iota(\boldsymbol{\alpha}),\iota(\boldsymbol{\beta})) \in \mathcal{LW}$. \label{propertiesiota1}
\end{proposition}
\begin{proof}
From the definition of $(\iota(\boldsymbol{\alpha}), \iota(\boldsymbol{\beta}))$ it is clear that $(\iota(\boldsymbol{\alpha}), \iota(\boldsymbol{\beta})) \in P \times \bar P$. Assume that $(\iota(\boldsymbol{\alpha}), \iota(\boldsymbol{\beta}))$ is extremal. Firstly, suppose that $\sigma^n(\iota(\boldsymbol{\alpha})) \prec \iota(\boldsymbol{\beta})$ for some $n \in \mathbb{N}$. Consider $N_{\iota(\boldsymbol{\beta})}(\iota(\boldsymbol{\alpha}))$. Since $\iota(\boldsymbol{\alpha})$ is periodic, then $1 < N_{\iota(\boldsymbol{\beta})}(\iota(\boldsymbol{\alpha})) < N_{\boldsymbol{\beta}}(\boldsymbol{\alpha}).$ This contradicts the choice of $N_{\boldsymbol{\beta}}(\boldsymbol{\alpha})$. Thus, $\boldsymbol{\beta} \prec \sigma^n(\iota(\boldsymbol{\alpha}))$ for every $n \in \mathbb{N}$. Similarly, $\sigma^n(\boldsymbol{\beta}) \prec \boldsymbol{\alpha}$ for every $n \in \mathbb{N}$. Therefore $(\iota(\boldsymbol{\alpha}),\iota(\boldsymbol{\beta})) \in \mathcal{LW}.$ 
\end{proof}

\begin{lemma}
If $(\boldsymbol{\alpha},\boldsymbol{\beta}) \in P \times \bar P$ is two sided extremal, then $\Sigma_{(\boldsymbol{\alpha},\boldsymbol{\beta)}} = \Sigma_{(\iota(\boldsymbol{\alpha}),\iota(\boldsymbol{\beta}))}$. 
\label{twosidedextremal}
\end{lemma}
\begin{proof}
Consider a two sided extremal pair $(\boldsymbol{\alpha},\boldsymbol{\beta})$. Since, $\iota(\boldsymbol{\alpha}) \prec \boldsymbol{\alpha}$ and $\boldsymbol{\beta} \prec \iota(\boldsymbol{\beta})$ then $\Sigma_{(\iota(\boldsymbol{\alpha}),\iota(\boldsymbol{\beta}))} \subset \Sigma_{(\boldsymbol{\alpha},\boldsymbol{\beta})}$. 

\vspace{0.5em}Assume that $\Sigma_{(\iota(\boldsymbol{\alpha}),\iota(\boldsymbol{\beta}))} \subsetneq \Sigma_{(\boldsymbol{\alpha},\boldsymbol{\beta})}.$ Let $\boldsymbol{x} \in \Sigma_{(\boldsymbol{\alpha},\boldsymbol{\beta})} \setminus \Sigma_{(\iota(\boldsymbol{\alpha}), \iota(\boldsymbol{\beta}))}$. Then $$\sigma^n(\boldsymbol{x}) \in \left[\boldsymbol{\beta}, \iota(\boldsymbol{\beta})\right)_{\prec} \cup (\iota(\boldsymbol{\alpha}),\boldsymbol{\alpha}]_{\prec} \hbox{ for every } n \geq 0.$$ Without losing generality we can assume that $\boldsymbol{x} \neq \boldsymbol{\alpha}$ and $\boldsymbol{x} \neq \boldsymbol{\beta}$. Then there exists $m^{\prime}$ such that either $$\sigma^{m^{\prime}}(\boldsymbol{x})_{M_{\boldsymbol{\alpha}}(\boldsymbol{\beta})} < (\iota(\boldsymbol{\beta}))_{M_{\boldsymbol{\alpha}}(\boldsymbol{\beta})} \hbox{\rm{ or }} \sigma^{m^{\prime}}(\boldsymbol{x})_{N_{\boldsymbol{\alpha}}(\boldsymbol{\beta})} > (\iota(\boldsymbol{\alpha}))_{N_{\boldsymbol{\alpha}}(\boldsymbol{\beta})}.$$ This implies that $\sigma^{m_b + j}(\boldsymbol{x}) \prec \boldsymbol{\beta}$ or $\sigma^{n_a+j}(\boldsymbol{x}) \succ \boldsymbol{\alpha}$ for some $j \in \mathbb{N}$, which is a contradiction. Thus, $\Sigma_{(\boldsymbol{\alpha},\boldsymbol{\beta})} = \Sigma_{(\iota(\boldsymbol{\alpha}),\iota(\boldsymbol{\beta}))}$.
\end{proof}

Now, we will concentrate in studying right and left extremal pairs. Given $\boldsymbol{\alpha} \in P$, we define $\xi_{\boldsymbol{\alpha}}: P^{\prime} \to P^{\prime}$ by: $$\xi_{\boldsymbol{\alpha}}(\boldsymbol{\beta}) = \left\{
\begin{array}{clrr}      
\boldsymbol{\beta} & \hbox{\rm{ if }} \sigma^n(\boldsymbol{\beta}) \prec \boldsymbol{\alpha} \hbox{\rm{ for every }} n \geq 0;\\
(b_1, \ldots b_{M_{\boldsymbol{\alpha}}(\boldsymbol{\beta})-1}1)^{\infty}& \hbox{\rm{ otherwise. }}\\
\end{array}
\right.$$  

\vspace{1em}Similarly, if $(\boldsymbol{\alpha},\boldsymbol{\beta}) \in P \times \bar  P$ is left extremal, we define $\xi^{\prime}_{\boldsymbol{\beta}}: P \to P$ by: $$\xi^{\prime}_{\boldsymbol{\beta}}(\boldsymbol{\alpha}) = \left\{
\begin{array}{clrr}      
\boldsymbol{\alpha} & \hbox{\rm{ if }}  
\sigma^n(\boldsymbol{\alpha}) \succ \boldsymbol{\beta} \hbox{\rm{ 
for every }} n \geq 0;\\
(a_1 \ldots a_{N_{\boldsymbol{\beta}
(\boldsymbol{\alpha})}-1}0)^{\infty} & \hbox{\rm{ otherwise.}} \\
\end{array}
\right.$$  

\begin{proposition}
Let $(\boldsymbol{\alpha}, \boldsymbol{\beta}) \in P \times \bar{P}$. Then 
\begin{enumerate}[i)]
\item $\xi_{\boldsymbol{\alpha}}(\boldsymbol{\beta}) \in \bar{P}$ if $(\boldsymbol{\alpha}, \boldsymbol{\beta})$ is right extremal; 
\item $\xi^{\prime}_{\boldsymbol{\beta}}(\boldsymbol{\alpha}) \in P$ if $(\boldsymbol{\alpha}, \boldsymbol{\beta})$ is left extremal.
\end{enumerate}
\label{sobrelasfunciones}
\end{proposition}
\begin{proof}
Let $(\boldsymbol{\alpha}, \boldsymbol{\beta})$ be right extremal. Assume that $\xi_{\boldsymbol{\alpha}}(\boldsymbol{\beta}) \notin \bar{P}$. Then there exist $j \in \mathbb{N}$ such that $\sigma^{j}(\xi_{\boldsymbol{\alpha}}(\boldsymbol{\beta})) \prec \xi_{\boldsymbol{\alpha}}(\boldsymbol{\beta})$. Since $\xi_{\boldsymbol{\alpha}}(\boldsymbol{\beta})$ is a periodic sequence of period $M_{\boldsymbol{\alpha}}(\boldsymbol{\beta})$ and $\xi_{\boldsymbol{\alpha}}(\boldsymbol{\beta})_i = b_i$ for every $i \in \{1,\ldots,  M_{\boldsymbol{\alpha}}(\boldsymbol{\beta})-1$ then $j = M_{\boldsymbol{\alpha}}(\boldsymbol{\beta})$. Then $\sigma^j(\xi_{\boldsymbol{\alpha}}(\boldsymbol{\beta})) \prec \xi_{\boldsymbol{\alpha}}(\boldsymbol{\beta})$. This implies that $\sigma^j(\xi_{\boldsymbol{\alpha}}(\boldsymbol{\beta}))_1 = 0$ which contradicts $\sigma^j(\xi_{\boldsymbol{\alpha}}(\boldsymbol{\beta})) \prec \xi_{\boldsymbol{\alpha}}(\boldsymbol{\beta})$. Thus, $\sigma^j(\xi_{\boldsymbol{\alpha}}(\boldsymbol{\beta})) \in \bar{P}$.

The proof of $ii)$ is analogous. 
\end{proof}

\begin{lemma}
Let $(\boldsymbol{\alpha},\boldsymbol{\beta}) \in P \times \bar P$. Then:
\begin{enumerate}[i)]
\item $(\boldsymbol{\alpha},\xi_{\boldsymbol{\alpha}}(\boldsymbol{\beta})) \in \mathcal{LW}$ if $(\boldsymbol{\alpha},\boldsymbol{\beta})$ is right extremal;
\item $(\xi^{\prime}_{\boldsymbol{\beta}}(\boldsymbol{\alpha}),\boldsymbol{\beta}) \in \mathcal{LW}$ if $(\boldsymbol{\alpha},\boldsymbol{\beta})$ is left extremal.
\end{enumerate} 
\label{propertiesiota2}
\end{lemma}
\begin{proof}
It suffices to show case $i)$. Assume that $(\boldsymbol{\alpha},\boldsymbol{\beta})$ is right extremal. To show that $(\boldsymbol{\alpha}, \xi_{\boldsymbol{\alpha}}(\boldsymbol{\beta})) \in \mathcal{LW}$ we need to show that $\sigma^n(\boldsymbol{\alpha}) \succcurlyeq \xi_{\boldsymbol{\alpha}}(\boldsymbol{\beta})$ and $\xi_{\boldsymbol{\alpha}}(\boldsymbol{\beta}) \preccurlyeq \sigma^m(\xi_{\boldsymbol{\alpha}}(\boldsymbol{\beta})) \preccurlyeq \boldsymbol{\alpha}$ for every $n, m \in \mathbb{N}$. From Proposition \ref{sobrelasfunciones} we have that $\sigma^m(\xi_{\boldsymbol{\alpha}}(\boldsymbol{\beta})) \succcurlyeq(\xi_{\boldsymbol{\alpha}}(\boldsymbol{\beta})$ for every $m \in \mathbb{N}$. Assume that there is $j \in \mathbb{N}$ such that $\sigma^j(\xi_{\boldsymbol{\alpha}}(\boldsymbol{\beta})) \succcurlyeq \boldsymbol{\alpha}.$ From the definition of $\xi_{\boldsymbol{\alpha}}$ we have that $j = M_{\boldsymbol{\alpha}}(\boldsymbol{\beta})$. Then there is $j^{\prime} \in \mathbb{N}$ such that $(\sigma^{j}(\xi_{\boldsymbol{\alpha}}(\boldsymbol{\beta})))_{j^{\prime}} > a_{j^{\prime}}$ which contradicts that $M_{\boldsymbol{\alpha}}(\boldsymbol{\beta})$ is minimal.

We need to show now that $\sigma^n(\boldsymbol{\alpha}) \succcurlyeq \xi_{\boldsymbol{\alpha}}(\boldsymbol{\beta})$ for every $n \in \mathbb{N}$. Suppose that there is $n \in \mathbb{N}$ such that $\sigma^n(\boldsymbol{\alpha}) \prec \xi_{\boldsymbol{\alpha}}(\boldsymbol{\beta})$. Since $(\boldsymbol{\alpha}, \boldsymbol{\beta})$ is right extremal, $\boldsymbol{\beta} \prec \sigma^{j}(\boldsymbol{\alpha}) \prec \xi_{\boldsymbol{\alpha}}(\boldsymbol{\beta})$. Recall that $d(\boldsymbol{\beta}, \xi_{\boldsymbol{\alpha}}(\boldsymbol{\beta})) = \frac{1}{2^{M_{\boldsymbol{\alpha}
(\boldsymbol {\beta})}}}$. This implies that $\sigma^{j}(\boldsymbol{\alpha})_{M_{\boldsymbol{\alpha}(\boldsymbol{\beta})}} = 0$ which contradicts that $M_{\boldsymbol{\alpha}}(\boldsymbol{\beta})$ is minimal again.
\end{proof}

\begin{lemma}
Let $(\boldsymbol{\alpha},\boldsymbol{\beta}) \in P \times \bar P$. Then:
\begin{enumerate}[i)]
\item $\Sigma_{(\boldsymbol{\alpha},\boldsymbol{\beta})} = \Sigma_{(\boldsymbol{\alpha},\xi_{\boldsymbol{\alpha}}
(\boldsymbol{\beta}))}$ if $(\boldsymbol{\alpha},\boldsymbol{\beta})$ is right extremal;
\item $\Sigma_{(\boldsymbol{\alpha},\boldsymbol{\beta})} = \Sigma_{(\xi^{\prime}_{\boldsymbol{\beta}} (\boldsymbol{\alpha}),\boldsymbol{\beta})}$ if $(\boldsymbol{\alpha},\boldsymbol{\beta})$ is left extremal.
\end{enumerate}
\label{fuera2}
\end{lemma}
\begin{proof}
It suffices to show \textit{i)} since the proof of \textit{ii)} is analogous. Observe that $\Sigma_{(\boldsymbol{\alpha} , \xi_{\boldsymbol{\alpha}}(\boldsymbol{\beta}))} \subset \Sigma_{(\boldsymbol{\alpha}, \boldsymbol{\beta})}$ since  $\xi_{\boldsymbol{\alpha}}(\boldsymbol{\beta}) \succcurlyeq \boldsymbol{\beta}$. We want to show that $\Sigma_{(\boldsymbol{\alpha} ,\boldsymbol{\beta})} \subset \Sigma_{(\boldsymbol{\alpha},\xi_{\boldsymbol{\alpha}}(\boldsymbol{\beta})}$. Assume that the inclusion does not hold. Then there exists $\boldsymbol{x} \in \Sigma_{(\boldsymbol{\alpha},\boldsymbol{\beta})} \setminus \Sigma_{(\boldsymbol{\alpha},\xi_{\boldsymbol{\alpha}}(\boldsymbol{\beta}))}$. This implies that there exists $n \in \mathbb{N}$ such that $\xi_{\boldsymbol{\alpha}}(\boldsymbol{\beta}) \succ \sigma^{n}(\boldsymbol{x}) \succ \boldsymbol{\beta}$. Then there exists $k_{\boldsymbol{\beta}} \in \mathbb{N}$ such that $\sigma^n(\boldsymbol{x})_{k_{\boldsymbol{\beta}}} < {\xi_{\boldsymbol{\alpha}}(\boldsymbol{\beta})}_{k_{\boldsymbol{\beta}}}$ and $\sigma^n(\boldsymbol{x})_i = {\xi_{\boldsymbol{\alpha}}(\boldsymbol{\beta})}_i$ for every $i < k_{\boldsymbol{\beta}}$. Thus, $\sigma^n(\boldsymbol{x}) \succ \boldsymbol{\alpha}$. This contradicts that $\boldsymbol{x} \in \Sigma_{(\boldsymbol{\alpha},\xi_{\boldsymbol{\alpha}}(\boldsymbol{\beta}))}$. Hence $\Sigma_{(\boldsymbol{\alpha}, \boldsymbol{\beta})} \subset \Sigma_{(\boldsymbol{\alpha},\xi_{\boldsymbol{\alpha}}(\boldsymbol{\beta}))}$. 
\end{proof}

We define $I: P \times \bar P \to P \times \bar P$ as: $$I(\boldsymbol{\alpha},\boldsymbol{\beta}) =\left\{
\begin{array}{clrr}      
(\boldsymbol{\alpha},\boldsymbol{\beta}) & \hbox{\rm{ if }} (\boldsymbol{\alpha},\boldsymbol{\beta}) \in \mathcal{LW};\\
(\iota(\boldsymbol{\alpha}), \iota(\boldsymbol{\beta})) & \hbox{\rm{ if }} (\boldsymbol{\alpha},\boldsymbol{\beta}) \hbox{\rm{ is two sided extremal}};\\
(\boldsymbol{\alpha}, \xi_{\boldsymbol{\alpha}}(\boldsymbol{\beta})) & \hbox{\rm{ if }} (\boldsymbol{\alpha},\boldsymbol{\beta}) \hbox{\rm{ is left extremal}};\\
(\xi^{\prime}_{\boldsymbol{\beta}}(\boldsymbol{\alpha}), \boldsymbol{\beta}) & \hbox{\rm{ if }} (\boldsymbol{\alpha},\boldsymbol{\beta}) \hbox{\rm{ is right extremal}}.
\end{array}
\right.$$  

\vspace{1em}The function $I$ provides the sought link between $P \times \bar P$ and the lexicographic world. 

\begin{theorem}
Let $(\boldsymbol{\alpha},\boldsymbol{\beta}) \in P \times \bar P$. Then $I$ is well defined. Moreover, $I(\boldsymbol{\alpha},\boldsymbol{\beta}) \in \mathcal{LW}$ for every $(\boldsymbol{\alpha},\boldsymbol{\beta}) \in P \times \bar P$. \label{dosparametros}
\end{theorem}
\begin{proof}
From Lemmas \ref{propertiesiota1} and \ref{propertiesiota2}, $I$ is well defined and $I(\boldsymbol{\alpha},\boldsymbol{\beta}) \in \mathcal{LW}$ for every $(\boldsymbol{\alpha},\boldsymbol{\beta}) 
\in P \times \bar P$.
\end{proof}

\begin{theorem}
For every $(\boldsymbol{\alpha},\boldsymbol{\beta}) \in P \times \bar P$, $\Sigma_{(\boldsymbol{\alpha},\boldsymbol{\beta})} = \Sigma_{I(\boldsymbol{\alpha},\boldsymbol{\beta})}$. 
\label{elmeropincheteoremota}
\end{theorem}
\begin{proof}
This is a direct consequence of Lemmas \ref{twosidedextremal} and \ref{fuera2}. 
\end{proof}

Using Theorem \ref{dosparametros} and Theorem \ref{elmeropincheteoremota}, we obtain the proof of Theorem \ref{asociar}, i.e. for every $(a,b) \in D_1$, there exists $(\boldsymbol{\alpha}, \boldsymbol{\beta}) \in \mathcal{LW}$ such that $(\Lambda_{(a,b)}, f_{(a,b)})$ is topologically conjugated to the lexicographic subshift $(\Sigma_{(\boldsymbol{\alpha}, \boldsymbol{\beta})}, \sigma_{(\boldsymbol{\alpha}.\boldsymbol{\beta})})$.

\vspace{1em}As a direct consequence of Theorem \ref{attractor1}, Theorem \ref{dosparametros} and Theorem \ref{elmeropincheteoremota} we obtain the following results.

\begin{theorem}
For every $(\boldsymbol{\alpha}^{\prime},\boldsymbol{\beta}^{\prime}) \in \Sigma^1_2 \times \Sigma^0_1$ there exists $(\boldsymbol{\alpha},\boldsymbol{\beta}) \in \mathcal{LW}$ such that $$(\Sigma_{(\boldsymbol{\alpha}^{\prime},\boldsymbol{\beta}^{\prime})},\sigma_{(\boldsymbol{\alpha}^{\prime}\boldsymbol{\beta}^{\prime})}) = (\Sigma_{(\boldsymbol{\alpha},\boldsymbol{\beta})},\sigma_{(\boldsymbol{\alpha},\boldsymbol{\beta})}).$$ 
\label{asociar1}
\end{theorem} 

\begin{corollary}
For every $(a,b) \in D_1$ there exists $(\boldsymbol{\alpha},\boldsymbol{\beta}) \in \mathcal{LW}$ such that $(\Lambda_{(a,b)}, f_{(a,b)})$ is topologically conjugated to $(\Sigma_{(\boldsymbol{\alpha}, \boldsymbol{\beta})}, \sigma_{(\boldsymbol{\alpha}, \boldsymbol{\beta})})$. 
\label{resultadazo}
\end{corollary}

\section{Lorenz Maps, intermediate $\beta$-expansions and Open dynamical systems}
\label{yasequedaasi}

\noindent A map $g:[0,1] \to [0,1]$ is a \textit{topologically expanding Lorenz map} if there exists $c \in (0,1)$ such that:
\begin{enumerate}[i)]
\item $g\mid_{[0.c)}$ and $g\mid_{(c,1]}$ are continuous and strictly increasing;
\item $\mathop{\lim}\limits_{x^+ \to c}g(x) = 1$ and $\mathop{\lim}\limits_{x^- \to c}g(x) = 0$;
\item $\overline{I_c} = [0,1]$ where $I_c = \mathop{\bigcup}\limits_{n=0}^{\infty}g^{-n}(c)$.
\end{enumerate}
We call the pair $([0,1],g)$ an \textit{expanding Lorenz dynamical system}. The dynamical properties of Lorenz maps have been extensively studied - see e.g. \cite{glendinning, glendinning1, hubbardsparrow, samuel} and references therein. It is worth mentioning that in \cite[6.2]{glendinning1} it was stated that Lorenz maps can be studied as open dynamical systems of the doubling map.  

\vspace{1em}It is natural to associate to a Lorenz dynamical system $([0,1],g)$ a symbolic space as follows: Let $x \in [0,1] \setminus I_c$. Then the \textit{kneading sequence of $x$}, denoted by $k_g(x)$ is the sequence of symbols $k_i(x) \in \left\{0,1\right\}$ given by $$k_i(x) = \left\{
\begin{array}{clrr}      
0 & \hbox{\rm{ if }} g^i(x) \in [0,c);\\
1 & \hbox{\rm{ if }} g^i(x) \in (c,1].\\
\end{array}
\right.$$  If $x \in I_c$ the \textit{upper kneading sequence of $x$} denoted by $k^+_g(x)$ is $k^+_g(x) = \mathop{\lim}\limits_{y \to x^+} k_g(y)$ and the \textit{lower kneading sequence of $x$}, $k^-_g(x)$, is $k^-_g(x) = \mathop{\lim}\limits_{y \to x^-} k_g(y)$. Observe that if $x \in [0,1]\setminus I_c$, $k^+_g(x) = k^-_g(x) = k_g(x)$. Finally, the \textit{kneading invariant of $g$} is the pair $(k^-_g(c),k^+_g(c))$. 

\vspace{1em}Hubbard and Sparrow have shown in \cite[Theorem 1]{hubbardsparrow} that given an expanding Lorenz map $g$, $(\sigma(k^+_g(c)),\sigma(k^-_g(c))) \in \mathcal{LW}$. Moreover, every element $(\boldsymbol{\alpha},\boldsymbol{\beta}) \in \mathcal{LW}$ determines a unique Lorenz map $g$ up to topological conjugacy. Also, in \cite[Theorem 2]{hubbardsparrow} it is shown that $(\Sigma_{(\sigma(k^+_g(c)),\sigma(k^-_g(c)))}, \sigma_{(\sigma(k^+_g(c)),\sigma(k^-_g(c)))})$ is semi conjugate to $([0,1],g)$, and the set of points such that the semi-conjugacy is not injective is precisely $I_c$. Therefore, by Theorem \ref{asociar}, the following statement is true.

\begin{proposition}
For every expanding Lorenz map $g$ there exists $(a,b) \in D_1$ such that $(\Lambda_{(a,b)}, f_{(a,b)})$ is semi-conjugate to $([0,1],g)$.\label{asocialorenzo1} 
\end{proposition}

A particular set of Lorenz maps were introduced by Parry in \cite{parry3}. Consider $\beta \in (1,2)$ and $\alpha \in (0,2-\beta)$. Define $$g_{\beta,\alpha}(x) = \left\{
\begin{array}{clrr}      
\beta x + \alpha, & \hbox{\rm{ if }}  x \in [0, \frac{1 -\alpha}{\beta}];\\
\beta x + \alpha - 1 & \hbox{\rm{ if }} x \in [\frac{1 -\alpha}{\beta}, 1].\\
\end{array}
\right.$$ 
Note that $g_{\beta, \alpha}$ is a linear and expanding Lorenz transformation. Maps of this form are often called \textit{linear $\mod 1$ transformations}. It is well known that $h_{top}(g_{\beta, \alpha}) = \log (\beta)$ \cite[p. 452]{frattolagarias}. Also, Glendinning in \cite[Theorem 1]{glendinning} gave the precise conditions under which an expanding Lorenz dynamical system $([0,1],g)$ is topologically conjugated to $([0,1], g_{\beta, \alpha})$ for $\beta \in (1,2)$ and $\alpha \in (0,2-\beta)$. 

\vspace{1em}Let $g_{(\beta, \alpha)}:[0,1]\to [0,1]$ be a linear $\mod 1$ transformation and let $(\boldsymbol{\alpha},\boldsymbol{\beta}) \in \mathcal{LW}$ be such that $(\boldsymbol{\alpha},\boldsymbol{\beta}) =  (\sigma(k^+_g(\frac{1 -\alpha}{\beta})),\sigma(k^-_g(\frac{1 -\alpha}{\beta})))$. Then the projection map $\pi_{\beta,\alpha}: \Sigma_{(\boldsymbol{\alpha},\boldsymbol{\beta})} \to [0,1]$ given by $\pi_{\beta, \alpha}(x) = \frac{\alpha}{\beta-1} +
\mathop{\sum}\limits_{n=1}^{\infty}\frac{x_n}{\beta^n}$ is a semi conjugacy between $(\Sigma_{(\boldsymbol{\alpha},\boldsymbol{\beta})}, \sigma_{(\boldsymbol{\alpha},\boldsymbol{\beta})})$ and $([0,1], g_{\beta, \alpha})$.

\vspace{1em}In \cite{dajani} the authors suggested the notion of \textit{intermediate $\beta$-expansions} or \textit{$(\beta,\alpha)$-expansions}. Let $\alpha \in (0, \frac{2-\beta}{\beta-1}),$ and consider the map: $$T_{\beta, \alpha} = \left\{
\begin{array}{clrr}      
\beta x, & \hbox{\rm{ if }}  x \in [\alpha, \frac{\alpha  + 1}{\beta});\\
\beta x - 1 & \hbox{\rm{ if }} x \in (\frac{\alpha + 1}{\beta}, 1+ \alpha].\\
\end{array}
\right.$$ 

Note that we defined the map $T_{\beta, \alpha}$ in $[\alpha, 1 + \alpha]$ since this set is an attractor for the extended map defined in $[0, \frac{1}{\beta-1}]$ \cite{dajani}. Considering the homeomorphism $\psi:[0,1] \to [\alpha, 1 + \alpha]$ given by $\psi(x)= x+\alpha$ it is easy to show -see \cite{dajani, sidorov4} that $T_{\beta, \alpha}$ is topologically conjugated to the map $g_{\beta, \alpha^{*}}$ given by $$g_{\beta, \alpha^{*}} = \left\{ 
\begin{array}{clrr}      
\beta x + \alpha^{*}, & \hbox{\rm{ if }}  x \in [0, \frac{1 -\alpha^{*}}{\beta});\\
\beta x + \alpha^{*} - 1 & \hbox{\rm{ if }} x \in (\frac{1 -\alpha^{*}}{\beta}, 1],\\
\end{array}
\right.$$ where $\alpha^* = \alpha(\beta-1)$. Then, as a consequence of Proposition \ref{asocialorenzo1} we obtain the following result.

\begin{proposition}
For every $\beta \in (1,2)$ and $\alpha \in (0, 2-\beta)$ the set of $(\beta,\alpha)$-expansions is a factor of an attractor $\Lambda_{(a,b)}$ with $(a,b) \in D_1$. 
\label{asociarintermediate1}
\end{proposition}

\section{Transitivity for $(\Lambda_{(a,b)},f_{(a,b)})$ and $(\Sigma_{(\boldsymbol{\alpha},\boldsymbol{\beta})}, 
\sigma_{(\boldsymbol{\alpha},\boldsymbol{\beta})})$}
\label{transitivity}

\noindent We start the section by giving the notion of \textit{renormalisability} of a pair $(\boldsymbol{\alpha},\boldsymbol{\beta}) \in \mathcal{LW}$ introduced in \cite[p. 27]{glendinning1} and the notion of renormalisation for a Lorenz dynamical system. Firstly, a Lorenz map $g:[0,1] \to [0,1]$ is said to be \textit{renormalisable} if there exist an interval $[a,b] \subset [0,1]$ and $n, m \in \mathbb{N}$ with $n$ and $m$ not both equal to $1$ such that $c \in [a,b]$ and the dynamical system $([a,b], Rg)$ given by $Rg: [a,b] \to [a,b]$ defined as 
$$Rg(x) = \left\{\begin{array}{clrr}
g^n(x) & \hbox{\rm{ if }} x \in [a, c);\\
g^m(x) & \hbox{\rm{ if }} x \in (c, b]\\
\end{array}
\right.$$ 
is topologically conjugated to a Lorenz dynamical system \cite{hall}. 

\vspace{1em}We would like to emphasise that it is a well known result that if $g:[0,1] \to [0,1]$ is a renormalisable expanding Lorenz map then its kneading invariant $(k^-_g(c),k^+_g(c))$ is renormalisable \cite{glendinning1}. Furthermore, it is also well known that an expanding Lorenz dynamical system $([0,1],g)$ is not transitive if and only if the kneading invariant $(k^-_g(c),k^+_g(c))$ is renormalisable \cite[Theorem 2]{glendinning}. Nonetheless, since $([0,1], g)$ is merely a factor of $(\Sigma_{(\boldsymbol{\alpha},\boldsymbol{\beta})}, \sigma_{(\boldsymbol{\alpha},\boldsymbol{\beta})})$ where $\boldsymbol{\alpha} = \sigma(k^-_g(c))$ and $\boldsymbol{\beta} = \sigma(k^+_g(c))$, we cannot be sure if this property is transferred to $(\Sigma_{(\boldsymbol{\alpha},\boldsymbol{\beta})}, \sigma_{(\boldsymbol{\alpha},\boldsymbol{\beta})})$ automatically. 

\begin{definition}
Let $(\boldsymbol{\alpha},\boldsymbol{\beta}) \in \mathcal{LW}$. We say that $(\boldsymbol{\alpha},\boldsymbol{\beta}) \in \mathcal{LW}$ is \textit{renormalisable} if there exist two words $\omega$ and $\nu$ and sequences $\{n^{\omega}_i\}_{i=1}^{\infty}$ $\{n^{\nu}_i\}_{i=1}^{\infty}$, $\{m^{\omega}_j\}_{j=1}^{\infty}$ and $\{m^{\nu}_j\}_{j=1}^{\infty} \subset \mathbb{N}\cup \{\infty\}$ such that $\omega = 0\hbox{\rm{-max}}_{\omega}$, $\nu = 1\hbox{\rm{-min}}_{\nu}$, $(0\hbox{\rm{-max}}_{\nu})^{\infty} \prec \omega^{\infty}$, $\nu^{\infty}\prec(1\hbox{\rm{-min}}_{\omega})^{\infty}$ ,$\ell(\omega\nu) \geq 3$ and $$0\boldsymbol{\alpha} = \omega \nu^{n^{\nu}_1}\omega^{n^{\omega}_1}\nu^{n^{\nu}_2}\omega^{n^{\omega}_2}\nu^{n^{\nu}_3}\ldots$$ and $$1\boldsymbol{\beta} = \nu \omega^{m^{\omega}_1}\nu^{m^{\nu}_1}\omega^{m^{\omega}_2}\nu^{m^{\nu}_2}\omega^{n^{\omega}_3}\ldots .$$ If $\ell(\omega \nu) = 3$ we say that $(\boldsymbol{\alpha},\boldsymbol{\beta})$ is trivially renormalisable. The pair $(\omega, \nu)$ is called \textit{the associated pair of $(\boldsymbol{\alpha}, \boldsymbol{\beta})$}. \label{renormdef}
\end{definition}

\begin{remark}
We will also call a pair $(\boldsymbol{\alpha},\boldsymbol{\beta}) \in \mathcal{LW}$ renormalisable if $\omega$ or $\nu$ is an infinite sequence. In this case, $0\boldsymbol{\alpha} = \omega \nu$ and $1\boldsymbol{\beta} = \nu$ if $\nu$ is an infinite sequence and $0\boldsymbol{\alpha}= \omega$ and $1\boldsymbol{\beta} = \nu \omega$ if $\omega$ is an infinite sequence. In this case, we say that $(\boldsymbol{\alpha},\boldsymbol{\beta})$ is \textit{renormalisable by an infinite sequence}. 
\label{remarkrenorm}
\end{remark}

\vspace{1em}In Definition \ref{renormdef} and Remark \ref{remarkrenorm}, $\omega$ and $\nu$ are always considered to be the shortest choice of renormalisation words. Observe that Definition \ref{renormdef} can be stated considering directly the binary expansion of $\pi^{-1}(a)$ and $\pi^{-1}(b)$ when $(a,b) \in D_0$ and $\left(\sigma(\pi^{-1}(a)),\sigma(\pi^{-1}(b))\right) \in \mathcal{LW}$. However, we will state our results in terms of pairs $(\boldsymbol{\alpha},\boldsymbol{\beta}) \in \mathcal{LW}$.  

\vspace{1em} The aim of this section is to give a symbolic proof of \cite[Theorem 2]{glendinning} characterising transitive subshifts $(\Sigma_{(\boldsymbol{\alpha},\boldsymbol{\beta})}, \sigma_{(\boldsymbol{\alpha},\boldsymbol{\beta})})$ in terms of renormalisability of such pair. In Theorem \ref{renorm1}, Theorem \ref{renorm1prime} and Theorem \ref{renorminfty} it is shown that lexicographic subshifts corresponding to renormalisable pairs $(\boldsymbol{\alpha}, \boldsymbol{\beta})$ are not transitive provided that $\omega$ and $\nu$ are not cyclically balanced words with $1(\omega) = 1(\nu)$. In such a case we show that Theorem \ref{renorm1} is also satisfied by such pairs and in Theorem \ref{balanceadas}. Nonetheless, in Theorem \ref{renorm1prime} we show the subshifts $(\Sigma_{(\boldsymbol{\alpha},\boldsymbol{\beta})},\sigma_{(\boldsymbol{\alpha},\boldsymbol{\beta})})$ corresponding to pairs $(0\boldsymbol{\alpha}, 1\boldsymbol{\beta}) \in \mathcal{LW}$ given by $0\boldsymbol{\alpha} = \omega\nu^{\infty}$ and $1\boldsymbol{\beta} = \nu\omega^{\infty}$ where $\omega = \omega_r$ and $\nu = \nu_r$ for $r \in \mathbb{Q} \cap (0,1)$ are transitive. Finally, in Theorem \ref{renorm2} and Theorem \ref{aproximaporabajogeneral} we show that lexicographic subshifts given by non renormalisable pairs are transitive.     

\vspace{1em}In order to the main results of this section, we prove first some easy claims which provide a useful partition of $\mathcal{LW}$. Also, Proposition \ref{numberofones} and Proposition \ref{numberofceros} generalise \cite[Proposition 4.2]{yomero1}.

\begin{proposition}
Let $n \in \mathbb{N}$ such that $n \geq 2$. Then for every $$\boldsymbol{\alpha} \in ((1^{n-2}0)^{\infty} ,(1^{n-1}0)^{\infty}]_{\prec},$$ $1^n$ is not a factor for any $\boldsymbol{x} \in \Sigma_{(\boldsymbol{\alpha},\boldsymbol{\beta})}$ for every $\boldsymbol{\beta} \in \bar P$. \label{numberofones}
\end{proposition}
\begin{proof}
Let $n \geq 2$ and $\boldsymbol{\alpha} \in ((1^{n-2}0)^{\infty} ,(1^{n-1}0)^{\infty}]_{\prec}$. Then $a_i = 1$ for every $i \in \{1, \ldots n-1\}$. Let $\boldsymbol{\beta} = 0^{\infty}$. Assume that there exist $\boldsymbol{x} \in \Sigma_{(\boldsymbol{\alpha},\boldsymbol{\beta})}$ such that $1^n$ is a factor of $\boldsymbol{x}$. Let $$j = \min\{k \in \mathbb{N} \mid (\sigma^{k}(\boldsymbol{x}))_i = 1 \hbox{\rm{ for }} i \in \{1, \ldots n\}\}.$$ Then $\sigma^j(\boldsymbol{x}) \succ \boldsymbol{\alpha}$ which is a contradiction. 

Since $\Sigma_{(\boldsymbol{\alpha},\boldsymbol{\beta})} \subset \Sigma_{(\boldsymbol{\alpha}, 0^{\infty})}$ for every $\boldsymbol{\beta} \in (0^{\infty}, \pi^{-1}(\chi(a)))_{\prec}$ then $1^n$ is not a factor of any $\boldsymbol{x} \in \Sigma_{(\boldsymbol{\alpha},\boldsymbol{\beta})}$.
\end{proof}

Observe that it is possible to show an analogous result interchanging $1$'s to $0$'s which is stated below. The proof is essentially the same as the one for Proposition \ref{numberofones}, so it is omitted.
\begin{proposition}
Let $n \in \mathbb{N}$ such that $n \geq 3$. Then for every $$\boldsymbol{\beta} \in ((0^{n-1}1)^{\infty}, (0^{n-2}1)^{\infty}],$$ $0^n$ is not a factor for any $\boldsymbol{x} \in \Sigma_{(\boldsymbol{\alpha},\boldsymbol{\beta})}$. 
\label{numberofceros}
\end{proposition}

\begin{lemma}
Let $j,k \geq 2$, then the subshift $(\Sigma_{\mathcal{F}}, \sigma_{\mathcal{F}})$ where $\mathcal{F} = \{0^j, 1^k\}$ is transitive. Moreover, this subshift induces the lexicographic subshift of finite type given by $\boldsymbol{\alpha} = (1^{k-1}0)^{\infty}$ and $\boldsymbol{\beta} = (0^{j-1}1)^{\infty}$. \label{ceroandone}
\end{lemma}
\begin{proof}
Let $n = \max \{j,k\}$. Assume that $n=j$, then $\mathcal{F}= \{0^n, 1^k\}.$ Let $u,v \in \mathcal{L}(\Sigma_{\mathcal{F}})$. If $u_{\ell(u)}=1$ and $v_1 =1$ then the word $0^{n-1}$ is a bridge between $u$ and $v$. For $u_{\ell(u)}=0$ and $v_1 =0$ the word $1^{k-1}$ is a bridge between $u$ and $v$. If $u_{\ell(u)}=1$ and $v_1=0$ then the word $0^{n-1}1$ is a bridge between $u$ and $v$. Finally, if $u_{\ell(u)}=0$ and $v_1 =1$ the word $10^{n-1}$ connect $u$ and $v$. If $n=k$, then $\mathcal{F} = \{0^j , 1^n\}.$ For this case similar bridges can be constructed. By \cite[Theorem 1.3]{samuel} the subshift $(\Sigma_{(\boldsymbol{\alpha},\boldsymbol{\beta})},\sigma_{(\boldsymbol{\alpha},\boldsymbol{\beta})})$ corresponding to $\boldsymbol{\alpha}= (1^{k-1}0)^\infty$ and $\boldsymbol{\beta} = (0^{j-1}1)^{\infty}$ is a subshift of finite type. Observe that $\Sigma_{(\boldsymbol{\alpha},\boldsymbol{\beta})} \neq \emptyset$ since $(0^{j-1}1^{k-1})^{\infty} \in \Sigma_{(\boldsymbol{\alpha},\boldsymbol{\beta})}$. Also, note that for every $\boldsymbol{x} \in  \Sigma_{(\boldsymbol{\alpha},\boldsymbol{\beta})}$, neither $0^j$ nor $1^n$ are factors of $\boldsymbol{x}$, which implies that $\Sigma_{(\boldsymbol{\alpha},\boldsymbol{\beta})} \subset \Sigma_{\mathcal{F}}$. Let $\boldsymbol{x} \in \Sigma_{\mathcal{F}}$ and assume that $\boldsymbol{x} \notin \Sigma_{(\boldsymbol{\alpha},\boldsymbol{\beta})}$. Then there exists $n \in \mathbb{N}$ such that $\sigma^n(\boldsymbol{x}) \prec \boldsymbol{\beta}$ or $\sigma^n(\boldsymbol{x}) \succ \boldsymbol{\alpha}$. Without losing generality assume that implies that $\sigma^n(\boldsymbol{x}) \prec \boldsymbol{\beta}$.  Then there exists $j^{\prime} \in \mathbb{N}$ such that $(\sigma^n(\boldsymbol{x}))_{j^{\prime}} = 0$ and $b_{j^{\prime}} = 1$, which implies that $0^j$ is a factor of $x$, a 
contradiction. Therefore $\Sigma_{\mathcal{F}} = \Sigma_{(\boldsymbol{\alpha},\boldsymbol{\beta})}$ and $(\Sigma_{(\boldsymbol{\alpha},\boldsymbol{\beta})}, \sigma_{(\boldsymbol{\alpha},\boldsymbol{\beta})})$ is a subshift of finite type.
\end{proof}
 
\begin{lemma}
Let $(\boldsymbol{\alpha},\boldsymbol{\beta}) \in \mathcal{LW}$ be renormalisable by $\omega$ and $\nu$. Then the sequences $\{n^{\omega}_i\}_{i=1}^{\infty}$, $\{n^{\nu}_i\}_{i=1}^{\infty}$, $\{m^{\omega}_j\}_{j=1}^{\infty}$ and $\{m^{\nu}_j\}_{j=1}^{\infty}$ are bounded if neither $n_i = \infty$ nor $m_j = \infty$ for $i,j \in \mathbb{N}$. Moreover, $$\max \{n^{\omega}_j\}_{j=1}^{\infty} \leq \max \{m^{\omega}_i\}_{i=1}^{\infty} = m^{\omega}_1$$ and $$\max \{m^{\nu}_i\}_{i=1}^{\infty} \leq \max \{n^{\nu}_j\}_{j=1}^{\infty} = n^{\nu}_1.$$ \label{boundedsequences}
\end{lemma}
\begin{proof}
Firstly, assume that $\{n^{\nu}_i\}_{i=1}^{\infty}$ is not bounded. Then for every $i$ there exists $i^{\prime}$ such that $n^{\nu}_{i} < n^{\nu}_{i^{\prime}}$. In particular, there exists $i^{\prime}_1$ such that $n_1 < n_{i^{\prime}_1}$. Take $n$ sufficiently large such that $\sigma^n(\boldsymbol{\alpha}) = \omega\nu^{n^{\nu}_{i^{\prime}_1}}$. Observe that $a_{(\ell(\omega)-1)n_1\ell(\nu)+1}=0$ and $\sigma^n(\boldsymbol{\alpha})_{(\ell(\omega)-1)n_1\ell(\nu)+1}=1$, which implies that $\boldsymbol{\alpha} \notin P$ which is a contradiction. Therefore $\{n^{\nu}_i\}_{i=1}^{\infty}$ is bounded and $\max \{n^{\nu}_i\}_{i=1}^{\infty} = n^{\nu}_1$.

Secondly, assume that $\{m^{\omega}_j\}_{j=1}^{\infty}$ is not bounded. Similarly, for every $j$ there exists $j^{\prime}$ such that $m^{\omega}_{j} < m^{\omega}_{j^{\prime}}$. Let $j^{\prime}_1$ be such that $m_1 < n_{j^{\prime}_1}$. Take $m$ sufficiently large such that $\sigma^m(\boldsymbol{\alpha}) = \nu\omega^{m^{\omega}_{j^{\prime}_1}}$. Observe that $b_{(\ell(\nu)-1)m_1\ell(\omega)+1}=0$ and $\sigma^m(\boldsymbol{\alpha})_{(\ell(\nu)-1)m_1\ell(\omega)+1}=0$, which implies that $\boldsymbol{\beta} \notin \bar P$ which is a contradiction. Therefore $\{m^{\omega}_j\}_{j=1}^{\infty}$ is bounded and $\max \{m^{\omega}_j\}_{j=1}^{\infty} = m^{\omega}_1$.

Now, assume that $\{n^{\omega}_i\}_{i=1}^{\infty}$ is not bounded. Since $\{m^{\omega}_j\}_{j=1}^{\infty}$ is bounded there exists $j$ such that $m^{\omega}_1 < n^{\omega}_j$. Since $(\boldsymbol{\alpha},\boldsymbol{\beta})$ is renormalisable, there exists $n^{\prime}$ such that $\sigma^{n^{\prime}}(\boldsymbol{\alpha}) = \nu \omega^{n^{\omega}_j}\nu\ldots$. Then $\sigma^{n^{\prime}}
(\boldsymbol{\alpha}) \prec \boldsymbol{\beta}$ which contradicts that $(\boldsymbol{\alpha},\boldsymbol{\beta}) \in \mathcal{LW}$. Therefore $\{n^{\omega}_i\}_{i=1}^{\infty}$ is bounded by $m^{\omega}_1$. Similarly,$\{m^{\nu}_j\}_{j=1}^{\infty}$ is bounded by $n^{\nu}_1$. The proof is omitted.
\end{proof}

Observe that $n^{\nu}_i = \infty$ for some $i \in \mathbb{N}$ if and only if $i =1$. Also, $m^{\omega}_j = \infty$ for some $j \in \mathbb{N}$ if and only if $j =1$. This is a direct corollary of Lemma \ref{boundedsequences}. 

\vspace{1em}Note that if a pair $(\boldsymbol{\alpha}, \boldsymbol{\beta})$ is renormalisable, then it is possible to define the substitution $\rho_{\omega, \nu}$ as follows: $\rho_{\omega, \nu}(\omega) = 0$ and $\rho_{\omega, \nu}(\nu) = 1$. It is possible to define the pair $\left(R_{\omega,\nu}(\boldsymbol{\alpha}), R_{\omega,\nu}(\boldsymbol{\beta})\right)$ to be $\left(R_{\omega,\nu}(\boldsymbol{\alpha}), R_{\omega,\nu}(\boldsymbol{\beta})\right) = (\sigma(\rho_{\omega, \nu}(0\boldsymbol{\alpha})), \sigma(\rho_{\omega, \nu}(1\boldsymbol{\beta})))$. As a consequence of Lemma \ref{boundedsequences} it is clear $\left(R_{\omega,\nu}(\boldsymbol{\alpha}), R_{\omega,\nu}(\boldsymbol{\beta})\right) \in \mathcal{LW}$. Observe that $\left(R_{\omega,\nu}(\boldsymbol{\alpha}), R_{\omega,\nu}(\boldsymbol{\beta})\right)$ can be a renormalisable pair by words $(\omega_1,\nu_1)$. In such a case we denote by $\left(R^2(\boldsymbol{\alpha}), R^2(\boldsymbol{\beta})\right) = \left(R_{\omega^{\prime},\nu^{\prime}}(R_{\omega,\nu}(\boldsymbol{\alpha})), R_{\omega^{\prime},\nu^{\prime}}(R_{\omega,\nu}(\boldsymbol{\beta}))\right)$. Then, for $n \in \mathbb{N}$ we say that $(\boldsymbol{\alpha}, \boldsymbol{\beta})$ is \textit{$n$-renormalisable} if for every $0 \leq k \leq n -1$ the pair $\left(R^k(\boldsymbol{\alpha}), R^k(\boldsymbol{\beta})\right)$ is renormalisable and $\left(R^n(\boldsymbol{\alpha}), R^n(\boldsymbol{\beta})\right)$ is not renormalisable. If $(\boldsymbol{\alpha}, \boldsymbol{\beta})$ is renormalisable for every $n \geq 0$ then we say that $(\boldsymbol{\alpha}, \boldsymbol{\beta})$ is \textit{infinitely renormalisable}.  

\vspace{1em}In \cite[Lemma 2]{glendinning1} the authors proved that $n_1^{\omega}$ and $m_1^{\nu}$ cannot be simultaneously equal to $1$ based in a more restrictive definition of the lexicographic world $\mathcal{LW}$. However using Definition \ref{renormdef} it is possible to consider such a case. Given $(\boldsymbol{\alpha},\boldsymbol{\beta}) \in \mathcal{LW}$, we say that $\Sigma_{(\boldsymbol{\alpha},\boldsymbol{\beta})}$ is a \textit{cyclic subshift} if there exists a finite word $\omega$ such that $\boldsymbol{\alpha} = \sigma(0\hbox{\rm{-max}}_{\omega}^{\infty})$ and $\boldsymbol{\beta}= \sigma(1\hbox{\rm{-min}}_{\omega}^{\infty})$. Note that from \cite[Theorem 1.3]{samuel} we know that every cyclic subshift is a subshift of finite type. Furthermore, for every cyclic subshift $0_{\boldsymbol{\beta}} = 0_{\boldsymbol{\alpha}}$ and $1_{\boldsymbol{\beta}} = 1_{\boldsymbol{\alpha}}$.

\subsection*{Transitivity in $D_1 \setminus D_2$: The cyclically balanced case}

\noindent Firstly we will study subshifts such that $(\boldsymbol{\alpha}, \boldsymbol{\beta})$ is renormalisable by $\omega$ and $\nu$ where $\omega$ and $\nu$ are cyclically balanced words with $\omega = 0\hbox{\rm{-max}}_{\omega}$, $\nu = 1\hbox{\rm{-min}}_{\nu}$, $\ell(\omega) = \ell(\nu)$ and $1(\omega) = 1(\nu) = \frac{p}{q}$ with $\gcd(p,q) = 1$. This implies that the corresponding hole $(a,b) \in D_1 \setminus D_2$. Note that Remark \ref{laobservacion} and Theorem \ref{lemadenik} imply that the pair $(\sigma(\omega^{\infty}), \sigma(\nu^{\infty}))$ is not renormalisable. Besides, $(\pi(\omega^{\infty}),\pi(\nu^{\infty})) \in D_0$. Observe that Remark \ref{laobservacion} and Theorem \ref{lemadenik} that the pair $(\sigma(\omega^{\infty}), \sigma(\nu^{\infty}))$ will also imply that the proof of Theorem \ref{renorm1} remains valid in $D_1$. However, in Theorem \ref{balanceadas} we show that the pairs $(\boldsymbol{\alpha}, \boldsymbol{\beta})$ such that $0\boldsymbol{\alpha} = \omega\nu^{\infty}$ and $1\boldsymbol{\beta} = \nu\omega^{\infty}$ and $\omega$ and $\nu$ satisfy to be cyclically balanced words with $\omega = 0\hbox{\rm{-max}}_{\omega}$, $\nu = 1\hbox{\rm{-min}}_{\nu}$, $\ell(\omega) = \ell(\nu)$ and $1(\omega) = 1(\nu) = \frac{p}{q}$ with $\gcd(p,q) = 1$ then $(\Sigma_{(\boldsymbol{\alpha}, \boldsymbol{\beta})}, \sigma_{(\boldsymbol{\alpha}, \boldsymbol{\beta})})$ is a transitive subshift. This shows that all the hypothesis of Theorem \ref{renorm1prime} are necessary.

\begin{proposition}
Let $\omega$ and $\nu$ be cyclically balanced words such that $\omega = 0\hbox{\rm{-max}}_{\omega}$, $\nu = 1\hbox{\rm{-min}}_{\nu}$, $\ell(\omega) = \ell(\nu)$ and $1(\omega) = 1(\nu) = \frac{p}{q}$ with $\gcd(p,q) = 1$. Then 
\begin{align*}
\sigma(\nu\omega^{\infty}) &\prec \sigma((\nu\omega^{n+1})^{\infty}) \prec \sigma((\nu\omega^{n})^{\infty}) \prec \sigma(\nu^{\infty}) \preccurlyeq \omega^{\infty} \\ 
&\prec \nu^{\infty} \preccurlyeq \sigma(\omega^{\infty}) \prec \sigma((\omega\nu^{n})^{\infty}) \prec \sigma((\omega\nu^{n+1})^{\infty}) \prec \sigma(\omega \nu^{\infty})
\end{align*} for every $n \in \mathbb{N}$.  \label{balanced1}
\end{proposition}
\begin{proof}
Firstly observe that for every $n \in \mathbb{N}$, $\omega^{\infty} \prec (\omega\nu^n)^{\infty}$ since $\omega = 0\hbox{\rm{-max}}_{\omega}$, $\nu = 1\hbox{\rm{-min}}_{\nu}$ and $\omega^{\infty}_{\ell(\omega)+1} = 0$ and $(\omega\nu^{n})^{\infty}_{\ell(\omega+1)} = 1$. Consider $n \in \mathbb{N}$, observe that $(\omega\nu^{n})^{\infty}_{(n+1)\ell(\omega)+1} = 0$ and $(\omega\nu^{n+1})^{\infty}_{(n+1)\ell(\omega)+1}=1$. Similarly, for every $n \in \mathbb{N}$, $(\omega\nu^{n})^{\infty}_{(n+1)\ell(\omega)+1} = 0$ and $\omega\nu^{\infty}_{(n+1)\ell(\omega)+1} =1$. By a similar argument it is possible to show that for every $n \in \mathbb{N}$, $(\nu\omega^n)^{\infty} \prec \nu^{\infty}$ as well as $\nu\omega^{\infty} \prec (\nu\omega^{n+1})^{\infty} \prec (\nu\omega^n)^{\infty}$. This gives $$\sigma(\nu\omega^{\infty}) \prec \sigma((\nu\omega^{n+1})^{\infty}) \prec \sigma((\nu\omega^{n})^{\infty}) \prec \sigma(\nu^{\infty}) \hbox{\rm{ for every }} n \in \mathbb{N}$$ and $$\sigma(\omega^{\infty}) \prec \sigma((\omega\nu^{n})^{\infty}) \prec \sigma((\omega\nu^{n+1})^{\infty}) \prec \sigma(\omega\nu^{\infty}) \hbox{\rm{ for every }} n \in \mathbb{N}.$$ We have that $\sigma(\nu^{\infty}) \preccurlyeq \omega^{\infty} \prec \nu^{\infty} \prec \sigma(\omega)$ because $\omega = 0\hbox{\rm{-max}}_{\omega}$, $\nu = 1\hbox{\rm{-min}}_{\nu}$, $\nu$ is a cyclic permutation of $\omega$ , $\sigma(\nu^{\infty})_1 = 0$ and $\sigma(\omega^{\infty})_1 = 1$. This completes the proof.
\end{proof}

We would like to remark that $\sigma(\nu^{\infty}) = \omega^{\infty}$ and $\sigma^{\omega^{\infty}} = \nu^{\infty}$ if and only if $1(\omega) = 1(\nu) = \frac{1}{p}$ for every $p \geq 2$.

\vspace{1em}Let $$\left\{\omega, \nu \right\}^{\infty} = \left\{\boldsymbol{x} \in \Sigma_{2} \mid \boldsymbol{x} = \sigma^{n}((\epsilon_i)_{i=1}^{\infty}) \hbox{\rm{ where }} \epsilon_{i} \in \{\omega, \nu\} \hbox{\rm{ and }} n \geq 0 \right\}.$$ It is clear that $(\left\{\omega, \nu\right\}^{\infty}, \sigma_{\left\{\omega, \nu\right\}^{\infty}})$ is a subshift.

\begin{lemma}
Let $\omega$ and $\nu$ be cyclically balanced words such that $\omega = 0\hbox{\rm{-max}}_{\omega}$, $\nu = 1\hbox{\rm{-min}}_{\nu}$, $\ell(\omega) = \ell(\nu)$ and $1(\omega) = 1(\nu) = \frac{p}{q}$ with $\gcd(p,q) = 1$. Then for every $n \geq 0$, $$\Sigma_{({\boldsymbol{\alpha}}_n, {\boldsymbol{\beta}}_n)} \subset \left\{\omega, \nu \right\}^{\infty},$$ where $({\boldsymbol{\alpha}}_n, {\boldsymbol{\beta}}_n) = (\sigma((\omega\nu^n)^{\infty}), \sigma((\nu\omega^{n})^{\infty}))$.\label{balanced2}
\end{lemma}
\begin{proof}
For $n = 0$ note that $\Sigma_{({\boldsymbol{\alpha}}_0, {\boldsymbol{\beta}}_0)} = \left\{\sigma^n(\omega^{\infty})\right\}_{n=0}^{q-1} \subset \left\{\omega, \nu\right\}^{\infty}$. Let us assume that $\Sigma_{({\boldsymbol{\alpha}}_k, {\boldsymbol{\beta}}_k)} = \left\{\sigma^n(\omega^{\infty})\right\}_{n=0}^{q-1} \subset \left\{\omega, \nu\right\}^{\infty}$ for $k \in \mathbb{N}$. Let $n = k+1$. From Proposition \ref{balanced1} we have that $\Sigma_{({\boldsymbol{\alpha}}_k, {\boldsymbol{\beta}}_k)} \subset \Sigma_{({\boldsymbol{\alpha}}_{k+1}, {\boldsymbol{\beta}}_{k+1})}$ then by induction we just need to show that $\boldsymbol{x} \in \left\{\omega, \nu\right\}^{\infty}$ if $\boldsymbol{x} \in \Sigma_{({\boldsymbol{\alpha}}_{k+1}, {\boldsymbol{\beta}}_{k+1})} \setminus \Sigma_{({\boldsymbol{\alpha}}_{k}, {\boldsymbol{\beta}}_{k})}$. Consider $\boldsymbol{x} \in \Sigma_{({\boldsymbol{\alpha}}_{k+1}, {\boldsymbol{\beta}}_{k+1})} \setminus \Sigma_{({\boldsymbol{\alpha}}_{k}, {\boldsymbol{\beta}}_{k})}$. Then 
\begin{align*}
\boldsymbol{x} \in \left[\sigma((\nu\omega^{k+1})^{\infty}), \sigma((\nu\omega^{k})^{\infty})\right]_{\prec} &\cup \left[(\omega\nu^{k})^{\infty},(\omega\nu^{k+1})^{\infty},\right]_{\prec}\\ 
&\cup \left[(\nu\omega^{k+1})^{\infty},(\nu\omega^{k})^{\infty}\right]_{\prec} \cup \left[\sigma((\omega\nu^{k})^{\infty}), \sigma((\omega\nu^{k+1})^{\infty})\right]_{\prec}.   
\end{align*}
Moreover, since $\left\{\omega, \nu\right\}^{\infty}$ contains $$\sigma((\nu\omega^{k+1})^{\infty}), \sigma((\nu\omega^{k})^{\infty}),(\omega\nu^{k})^{\infty},(\omega\nu^{k+1})^{\infty},(\nu\omega^{k+1})^{\infty},(\nu\omega^{k})^{\infty}, \sigma((\omega\nu^{k})^{\infty})$$ and $\sigma((\omega\nu^{k+1})^{\infty}),$ then it suffices to consider 
\begin{align*}
\boldsymbol{x} \in \left(\sigma((\nu\omega^{k+1})^{\infty}), \sigma((\nu\omega^{k})^{\infty})\right)_{\prec} &\cup \left((\omega\nu^{k})^{\infty},(\omega\nu^{k+1})^{\infty},\right)_{\prec}\\ 
&\cup \left((\nu\omega^{k+1})^{\infty},(\nu\omega^{k})^{\infty}\right)_{\prec} \cup \left(\sigma((\omega\nu^{k})^{\infty}), \sigma((\omega\nu^{k+1})^{\infty})\right)_{\prec}.   
\end{align*}
Assume that $\boldsymbol{x} \in \left((\omega\nu^{k})^{\infty},(\omega\nu^{k+1})^{\infty},\right)_{\prec}$ or $\boldsymbol{x} \in \left((\nu\omega^{k+1})^{\infty},(\nu\omega^{k})^{\infty}\right)_{\prec}$. Let us assume the former. Then  $x_1 \ldots x_{\ell(\omega)+n\ell(\nu)} = w_1\ldots w_{\ell(\omega)}(v_1\ldots v_{\ell(\nu)})^n$. If $x_{\ell(\omega)+n\ell(\nu)+1} = 0$ then $\boldsymbol{x} = (\omega\nu^{k})^{\infty}$ and if $x_{\ell(\omega)+n\ell(\nu)+1} = 1$ then $x_{\ell(\omega)+n\ell(\nu)+1}\ldots x_{\ell(\omega)+n+1\ell(\nu)} = \nu$. Then $\boldsymbol{x} \in \left\{\omega, \nu \right\}^{\infty}$. Changing $\omega$ by $\nu$ and vice versa and using a similar argument if $\boldsymbol{x} \in \left((\nu\omega^{k+1})^{\infty},(\nu\omega^{k})^{\infty}\right)_{\prec}$ then $\boldsymbol{x} \in \left\{\omega, \nu \right\}^{\infty}$. Observe that $\sigma\left(\left((\omega\nu^{k})^{\infty},(\omega\nu^{k+1})^{\infty}\right)_{\prec}\right) = \left(\sigma((\omega\nu^{k})^{\infty}), \sigma((\omega\nu^{k+1})^{\infty})\right)_{\prec}$ and $\sigma\left(\left((\nu\omega^{k+1})^{\infty},(\nu\omega^{k})^{\infty}\right)_{\prec}\right) = \left(\sigma((\nu\omega^{k+1})^{\infty}), \sigma((\nu\omega^{k})^{\infty})\right)_{\prec}$ and the proof is complete.  
\end{proof}

\begin{lemma}
Let $\omega$ and $\nu$ be cyclically balanced with $1(\omega) = 1(\nu)$ and $\ell(\omega) = \ell(\nu)$ then $(\Sigma_{(\boldsymbol{\alpha}, \boldsymbol{\beta})}, \sigma_{(\boldsymbol{\alpha}, \boldsymbol{\beta})}) = \left\{\omega, \nu \right\}^{\infty}$ where $(\boldsymbol{\alpha}, \boldsymbol{\beta}) = (\sigma(\omega\nu^{\infty}),\sigma(\nu\omega^{\infty}))$ .\label{lemacomoeldenik} 
\end{lemma}
\begin{proof}
Firstly we show that $\left\{\omega, \nu \right\}^{\infty} \subset (\Sigma_{(\boldsymbol{\alpha}, \boldsymbol{\beta}}, \sigma_{\boldsymbol{\alpha}, \boldsymbol{\beta})})$. Note that $\omega\nu^{\infty}$ is the lexicographically largest word starting with $0$ in $\left\{\omega, \nu \right\}^{\infty}$ and $\nu\omega^{\infty}$ is the lexicographically smallest word starting with $1$ in $\left\{\omega, \nu \right\}^{\infty}$. Then $\sigma(\omega\nu^{\infty})$ is the lexicographically largest word in $\left\{\omega, \nu \right\}^{\infty}$ and $\sigma(\nu\omega^{\infty})$ is the lexicographically smallest word in $\left\{\omega, \nu \right\}^{\infty}$. Since $\left\{\omega, \nu \right\}^{\infty}$ is a subshift, then $$\boldsymbol{\beta} = \sigma(\nu\omega^{\infty}) \preccurlyeq \sigma^n(\boldsymbol{x}) \preccurlyeq \sigma(\omega\nu^{\infty}) = \boldsymbol{\alpha}$$ for every $\boldsymbol{x} \in \left\{\omega, \nu \right\}^{\infty}$. This implies that $\left\{\omega, \nu \right\}^{\infty} \subset (\Sigma_{(\boldsymbol{\alpha}, \boldsymbol{\beta})}, \sigma_{(\boldsymbol{\alpha}, \boldsymbol{\beta})})$.

\vspace{0.5em}Let us show now that $\Sigma_{(\boldsymbol{\alpha}, \boldsymbol{\beta})} \subset \left\{\omega, \nu \right\}^{\infty}$. Observe that $\Sigma_{(\boldsymbol{\alpha}, \boldsymbol{\beta})} = \overline{\mathop{\bigcup}\limits_{n=0}^{\infty} \Sigma_{({\boldsymbol{\alpha}}_n, {\boldsymbol{\beta}}_n)}}$ where $({\boldsymbol{\alpha}}_n, {\boldsymbol{\beta}}_n) = (\sigma((\omega\nu^n)^{\infty}), \sigma((\nu\omega^{n})^{\infty}))$. Let $\boldsymbol{x} \in \Sigma_{(\boldsymbol{\alpha}, \boldsymbol{\beta})}$. Note that if $\boldsymbol{x} \in  \Sigma_{(\boldsymbol{\alpha}, \boldsymbol{\beta})} \setminus \mathop{\Sigma_{(\boldsymbol{\alpha}, \boldsymbol{\beta})}}\limits^{\circ}$ then $\boldsymbol{x} \in \left\{\sigma(\omega\nu^{\infty}),\sigma(\nu\omega^{\infty})\right\} \subset \left\{\omega, \nu \right\}^{\infty}$. If $\boldsymbol{x} \in  \mathop{\Sigma_{(\boldsymbol{\alpha}, \boldsymbol{\beta})}}\limits^{\circ}$ then there exist $n \geq 0$ such that $\boldsymbol{x} \in \Sigma_{({\boldsymbol{\alpha}}_n, {\boldsymbol{\beta}}_n)}$. Then, by Lemma \ref{balanced2}, $\boldsymbol{x} \in \left\{\omega, \nu\right\}^{\infty}$. Thus $\Sigma_{(\boldsymbol{\alpha}, \boldsymbol{\beta})} \subset \left\{\omega, \nu \right\}^{\infty}$ which finishes the proof.
\end{proof}

We emphasise that Lemma \ref{lemacomoeldenik} is false if $\omega$ and $\nu$ are not cyclically balanced words or if $\omega$ and $\nu$ are cyclically balanced but $1(\omega) \neq 1(\nu)$. For example if $\omega = 01110$ and $\nu = 100001$ then the period $2$ cycle $(01)^{\infty} \in \Sigma_{(\sigma(\omega\nu^{\infty}),\sigma(\nu\omega^{\infty}))}$ and $(01)^{\infty} \neq \left\{\omega, \nu\right\}^{\infty}$.  

\begin{theorem}
If $\omega$ and $\nu$ are cyclically balanced with $1(\omega) = 1(\nu)$ and $\ell(\omega) = \ell(\nu)$ then $(\Sigma_{\boldsymbol{\alpha}, \boldsymbol{\beta}}, \sigma_{(\boldsymbol{\alpha}, \boldsymbol{\beta})})$ is a transitive subshift where $\boldsymbol{\alpha} = \sigma(\omega\nu^{\infty})$ and $\boldsymbol{\beta} = \sigma(\nu\omega^{\infty})$ \label{balanceadas}
\end{theorem}
\begin{proof}
From Lemma \ref{lemacomoeldenik} it is just needed to show that $\left\{\omega,\nu\right\}^{\infty}$ is a transitive subshift. From \cite[Theorem 2.1]{fiebig} $\left\{\omega,\nu\right\}^{\infty}$ is coded. Thus, $\left\{\omega,\nu\right\}^{\infty}$ is transitive. 
\end{proof}

Note that the transitivity of $\left\{\omega,\nu\right\}^{\infty}$ does not rely on the fact that $\omega$ and $\nu$ satisfy the hypothesis of Theorem \ref{balanceadas}. In fact, for any two words $\omega$ and $\nu$, $\left\{\omega,\nu\right\}^{\infty}$ is transitive. 

\vspace{1em}Let $\omega$, $\nu$ be finite words. Observe that $\left\{\omega, \nu\right\}^{\infty}$ is a one-sided version of a \textit{uniquely decipherable renewal system} - see \cite{lind} for a general definition. By an easy modification of the proof of \cite[Theorem 3.3]{lind} we can be sure that $\left\{\omega, \nu\right\}^{\infty}$ is a transitive sofic subshift. Let $W$ be the set of all bi-infinite concatenations of $\omega$ and $\nu$. Then, it is clear that $B_n(W) = B_n(\left\{\omega, \nu\right\}^{\infty})$. This implies that $h_{top}(W) = h_{top}(\left\{\omega, \nu\right\}^{\infty})$. Moreover $$h_{top}(\left\{\omega, \nu\right\}^{\infty}) = \log \lambda$$ where $\frac{1}{\lambda}$ is the unique root of $1-t^{\ell(\omega)}+t^{\ell(\nu)}$ in $[0,1]$ \cite[Remark 2.1]{hong} \cite[p. 1008]{hall}. 

\subsection*{Transitivity in $D_2$}

Now we restrict our attention to study transitivity for subshifts corresponding to $(a,b) \in D_2$. Observe that from \cite[Theorem 3.8]{sidorov6} and \cite[Theorem 2.13]{sidorov1} we are sure that if $(a,b) \in D_2$ then the associated pair $(\boldsymbol{\alpha},\boldsymbol{\beta})$ is not renormalisable by $\omega_r$ and $\nu_r$ for $r \in \mathbb{Q} \cap (0,1)$.    

\vspace{1em}Let us now begin with the proof of Theorem \ref{renorm1} and Theorem \ref{renorminfty}. For this purpose we need to show the following technical lemmas.

\begin{proposition}
Let $(\boldsymbol{\alpha}, \boldsymbol{\beta}) \in \mathcal{LW}$. Suppose that $(\boldsymbol{\alpha}, \boldsymbol{\beta})$ is renormalisable by $\omega$ and $\nu$ and $(a,b) = (\pi(0\boldsymbol{\alpha}),\pi(1\boldsymbol{\beta})) \in D_2$. Then ${\max_{\omega}}^{\infty} \prec {\max_{\nu}}^{\infty}$. \label{vyw1}
\end{proposition}
\begin{proof}
Assume that ${\max_{\omega}}^{\infty} = {\max_{\nu}}^{\infty}$. Then $\omega$ and $\nu$ are cyclic permutations of each other. Then by Proposition \ref{endings} there exist words $\upsilon$ and $\varpi$ such that $\omega = \upsilon\varpi$ and $\nu = \varpi\upsilon$ which contradicts our assumption on the length of $\omega$ and $\nu$.

\vspace{0.5em}Assume that ${\max_{\omega}}^{\infty} \prec {\max_{\nu}}^{\infty}$. From Proposition \ref{endings} we have that ${\max_{\omega}}_{\ell(\omega)} = {\max_{\omega}}_{\ell(\nu)} = 0$. This implies that $\omega^{\infty} \prec \nu^{\infty}$ which contradicts that $(\boldsymbol{\alpha}, \boldsymbol{\beta})$ is renormalisable.
\end{proof}

\begin{lemma}
Let $(\boldsymbol{\alpha},\boldsymbol{\beta}) \in \mathcal{LW}$ be renormalisable by $\omega$ and $\nu$. Then $\omega^{\infty}, \nu^{\infty} \in \Sigma_{(\boldsymbol{\alpha},\boldsymbol{\beta})}$.\label{vyw}
\end{lemma}
\begin{proof}
Firstly, note that if $0\boldsymbol{\alpha}= \omega \nu^{\infty}$ and $1\boldsymbol{\beta} = \nu \omega^{\infty}$ the result is automatically true. Let 
\begin{center}
$0\boldsymbol{\alpha} = \omega \nu^{n^{\nu}_1}\omega^{n^{\omega}_1}\nu^{n^{\nu}_2}\omega^{n^{\omega}_2}\nu^{n^{\nu}_3}\ldots$ and $1\boldsymbol{\beta}= \nu \omega^{m^{\omega}_1}\nu^{m^{\nu}_1}\omega^{m^{\omega}_2}\nu^{m^{\nu}_2}\omega^{n^{\omega}_3}\ldots$. 
\end{center}
Note that $(\omega\nu)^{\infty}, (\nu\omega)^{\infty} \in \Sigma_{(\boldsymbol{\alpha},\boldsymbol{\beta})}$ since $(\boldsymbol{\alpha},\boldsymbol{\beta}) \in \mathcal{LW}$, $$\sigma^n((\omega \nu)^{\infty}) \preccurlyeq \omega \nu^{n^{\nu}_1}\omega^{n^{\omega}_1}\nu^{n^{\nu}_2}\omega^{n^{\omega}_2}\nu^{n^{\nu}_3}\ldots,$$ and $$\sigma^n((\nu\omega)^{\infty}) \succcurlyeq \nu \omega^{m^{\omega}_1}\nu^{m^{\nu}_1}\omega^{m^{\omega}_2}\nu^{m^{\nu}_2}\omega^{n^{\omega}_3}\ldots$$ for any sequences $\{n^{\omega}_i\}_{i=1}^{\infty}$, $\{n^{\nu}_i\}_{i=1}^{\infty}$, $\{m^{\omega}_j\}_{j=1}^{\infty}$, $\{m^{\nu}_j\}_{j=1}^{\infty}$ and for every $n \geq 0$. This implies that $\omega \nu , \nu \omega \in \mathcal{L}(\Sigma_{(\boldsymbol{\alpha},\boldsymbol{\beta})})$.  Note that $(\omega^{\infty})_{\ell(\omega)+1} = 0$ and $((\omega\nu)^{\infty})_{\ell(\omega)+1} = 1$. Also, $(\nu^{\infty})_{\ell(\nu)+1} = 1$ and $((\nu\omega)^{\infty})_{\ell(\nu)+1} = 0$. This implies that 
$$\omega^{\infty} \prec  (\omega\nu)^{\infty} \preccurlyeq 0\boldsymbol{\alpha}$$ and $$1\boldsymbol{\beta} \preccurlyeq (\nu\omega)^{\infty} \prec \nu^{\infty}.$$ Since $(\sigma(\omega^{\infty}), \sigma(\nu)^{\infty}) \in \mathcal{LW}$ and $\Sigma_{(\sigma(\omega^{\infty}), \sigma(\nu)^{\infty})} \subset \Sigma_{(\sigma((\omega\nu)^{\infty}), \sigma((\nu\omega)^{\infty}))} \subset \Sigma_{(\boldsymbol{\alpha}, \boldsymbol{\beta})}$ we conclude that $\omega^{\infty}, \nu^{\infty} \in \Sigma_{(\boldsymbol{\alpha},\boldsymbol{\beta})}$. 

\vspace{0.5em}Using the same arguments as before we can show that $\omega^{\infty}, \nu^{\infty} \in \Sigma_{(\boldsymbol{\alpha},\boldsymbol{\beta})}$ if  
\begin{center}
$0\boldsymbol{\alpha} = \omega\nu^{\infty}$ and $1\boldsymbol{\beta}= \nu\omega^{m^{\omega}_1}\nu^{m^{\nu}_1}\omega^{m^{\omega}_2}\nu^{m^{\nu}_2}\omega^{n^{\omega}_3}\ldots$,
\end{center} 
or 
\begin{center}
$0\boldsymbol{\alpha} = \omega \nu^{n^{\nu}_1}\omega^{n^{\omega}_1}\nu^{n^{\nu}_2}\omega^{n^{\omega}_2}\nu^{n^{\nu}_3}\ldots$ and $1\boldsymbol{\beta} = \nu\omega^{\infty}.$ 
\end{center}
\end{proof}

\begin{theorem}
Let $(\boldsymbol{\alpha},\boldsymbol{\beta}) \in \mathcal{LW}$ be a renormalisable pair by $\omega$ and $\nu$ such that $\ell(\omega)+ \ell(\nu) > 4$, $0\boldsymbol{\alpha} \neq \omega\nu^{\infty}$ and $1\boldsymbol{\beta} \neq \nu\omega^{\infty}$. Then $(\Sigma_{(\boldsymbol{\alpha},\boldsymbol{\beta})}, \sigma_{(\boldsymbol{\alpha},\boldsymbol{\beta})})$ is not transitive.\label{renorm1}
\end{theorem}
\begin{proof}
Let $(\boldsymbol{\alpha},\boldsymbol{\beta}) \in \mathcal{LW}$ be renormalisable by $\omega$ and $\nu$. By Lemma \ref{vyw}, $\omega^{\infty}$ and $\nu^{\infty} \in \Sigma_{(\boldsymbol{\alpha},\boldsymbol{\beta})}$.  Also, $(\omega \nu)^m$ and $(\nu \omega)^m \in \mathcal{L}(\Sigma_{(\boldsymbol{\alpha},\boldsymbol{\beta})})$ for every $m \in \mathbb{N}$. Note that $\sigma(\omega)1 \in B_{\ell(\omega)}$ is the maximal admissible word of length $\ell(\omega)$ and $\sigma(\nu)0 \in B_{\ell(\nu)}$ is the minimal admissible word of length $\ell(\nu)$. It is needed to consider two cases.

\vspace{0.5em}\noindent \textit{Case 1)}: Assume that 
\begin{center}
$0\boldsymbol{\alpha} = \omega \nu^{n^{\nu}_1}\omega^{n^{\omega}_1}\nu^{n^{\nu}_2}\omega^{n^{\omega}_2}\nu^{n^{\nu}_3}\ldots$ and $1\boldsymbol{\beta} = \nu \omega^{m^{\omega}_1}\nu^{m^{\nu}_1}\omega^{m^{\omega}_2}\nu^{m^{\nu}_2}\omega^{n^{\omega}_3}\ldots$.
\end{center}
From Lemma \ref{vyw}, $\omega^{m^{\omega}_1 + 1}$ and $\nu^{n^{\nu}_1 +1} \in \mathcal{L}(\Sigma_{(\boldsymbol{\alpha},\boldsymbol{\beta})})$.  We want to show that there are no bridges between $\omega1$ and $\nu^{n_1^{\nu}+1}$. Note that Lemma \ref{boundedsequences} implies that $\omega\nu^{n_1^{\nu}+1} \notin \mathcal{L}(\Sigma_{(\boldsymbol{\alpha},\boldsymbol{\beta})})$. Also $\omega1\nu^{n_1^{\nu}+1} \notin \mathcal{L}(\Sigma_{(\boldsymbol{\alpha},\boldsymbol{\beta})})$ since $\sigma(\omega1\nu^{n_1^{\nu}+1}) \succ \boldsymbol{\alpha}$. Suppose that there exists a bridge $\upsilon$ such that $\omega1\upsilon\nu^{n_1^{\nu}+1} \in \mathcal{L}(\Sigma_{(\boldsymbol{\alpha},\boldsymbol{\beta})})$. Since $\sigma(\omega)1$ is the maximal admissible word of length $\ell(\omega)$ the first $\ell(\nu)-2$ digits of $\upsilon$ satisfy that $u_i = \nu_{i+1}$ and the following digit is free. If $u_{\ell(\nu)-1} \neq \nu_{\ell(\nu)}$ then $\sigma(\omega)1\upsilon \succ \boldsymbol{\alpha}$ or $1\upsilon \prec \nu$ which is a contradiction. Then $u_{\ell(\nu)-1} = u_{\ell(\nu)}$. Also, if $\ell(\upsilon) = \ell(\nu)-1$ then $\sigma(\omega)1\upsilon\nu^{n_1^{\nu}+1} \succ \boldsymbol{\alpha}$ which is a contradiction. This implies that $\ell(\upsilon) > \ell(\nu)-1$ and that $u_1 \ldots u_{\ell(\nu)-1} = \sigma(\nu)$. Then $u_{\ell(\nu)} = 0$ then the following $\ell(\omega)$ digits coincide with $\omega$ and we reach a contradiction. If $u_{\ell(\nu)} = 1$ then the following $\ell(\nu)$ digits will coincide with $\nu$. This implies that $\ell(\upsilon) = \infty$ which is a contradiction. Then $(\Sigma_{(\boldsymbol{\alpha},\boldsymbol{\beta})},\sigma_{(\boldsymbol{\alpha},\boldsymbol{\beta})})$ is not transitive.

\vspace{0.5em}\noindent \textit{Case 2)}: Assume that $n^{\nu}_1 = \infty$ or $m^{\omega}_1 = \infty$ but not both. Without losing generality assume that $n^{\nu}_1 = \infty$. Then $0\boldsymbol{\alpha} = \omega \nu^{\infty}$ and $1\boldsymbol{\beta} = \nu\omega^{m^{\omega}_1}\nu^{m^{\nu}_1}\omega^{m^{\omega}_2}\nu^{m^{\nu}_2}\omega^{n^{\omega}_3}\ldots$. We observe that there are no bridges between $\nu0$ and $\omega^{m^{\omega}_1 +1}$ as follows. Firstly, from Lemma \ref{boundedsequences} we now that $\nu0\omega^{m^{\omega}_1 +1} \notin \mathcal{L}(\Sigma_{(\boldsymbol{\alpha},\boldsymbol{\beta})})$. Assume that the bridge $\upsilon$ exists. Since $\sigma(\nu)0$ is the minimal admissible word of length $\ell(\nu)$ then $\upsilon$ satisfies that $u_i = \omega_{i+1}$ for $i\in \{1, \ldots \ell(\omega)-1\}$ and the following digit is free. If $u_{\ell(\omega)} = 1$ then $u = \sigma(\omega)1$, this implies that $u_{\ell(\omega) + i} = \nu_{i+1}$ for $i \in \{1 \ldots \ell(\nu)-1\}$. Note that if $u_{\ell(\omega)+\ell(\nu)} = 0$ then we fall in our starting case. If $u_{\ell(\omega)+\ell(\nu)} = 1$ then $u_{\ell(\omega)+\ell(\nu)+i} = \nu_{i}$. Then $u_{\ell(\omega)} = 0$. If $u_{\ell(\omega)} = 0$ then $u_{\ell(\omega)+i} = \omega_{i+1}$ for $i \in \{1 \ldots \ell(\omega)-1\}$. Note that $u_{2\ell(\omega)} = 1$ we get the same conclusion from the case when $u_{\ell(\omega)} = 1$. Thus, $u_{2\ell(\omega)} = 0$. This implies that $u_1 \ldots u_{m^{\omega}_1(\ell(\omega)) - 1} = \sigma(\omega)^{m^{\omega}_1}$ and $u_{m^{\omega}_1\ell(\omega)} = 1$ which takes us to the starting case. Then $(\Sigma_{(\boldsymbol{\alpha},\boldsymbol{\beta})}, \sigma_{(\boldsymbol{\alpha},\boldsymbol{\beta})})$ is not transitive.
\end{proof}

\begin{corollary}
If $(\Sigma_{(\boldsymbol{\alpha},\boldsymbol{\beta})}, \sigma_{(\boldsymbol{\alpha},\boldsymbol{\beta})})$ is a cyclic subshift, then $(\Sigma_{(\boldsymbol{\alpha},\boldsymbol{\beta})}, \sigma_{(\boldsymbol{\alpha},\boldsymbol{\beta})})$ is not transitive. \label{cyclicsft}
\end{corollary}
\begin{proof}
Let $(\Sigma_{(\boldsymbol{\alpha},\boldsymbol{\beta})},\sigma_{(\boldsymbol{\alpha},\boldsymbol{\beta})})$ be a cyclic subshift. Observe that $\boldsymbol{\alpha}$ and $\boldsymbol{\beta}$ are cyclic permutations of each other. Then Proposition \ref{palabrasciclicas} implies there exist $\omega, \nu \in \mathcal{L}(\Sigma_2)$ such that $\omega$ starts with $0$, $\nu$ starts with $1$, $\omega = 0\max(\omega)$, $\nu = 1\min(\nu)$, $0\boldsymbol{\alpha} =(\omega\nu)^{\infty}$ and $1\boldsymbol{\beta} = (\nu\omega)^\infty$ which implies that $(\boldsymbol{\alpha},\boldsymbol{\beta})$ is renormalisable. Then by Theorem \ref{renorm1}, $(\Sigma_{(\boldsymbol{\alpha},\boldsymbol{\beta})},\sigma_{(\boldsymbol{\alpha},\boldsymbol{\beta})})$ is not transitive.
\end{proof}

\begin{corollary}
If $$(a,b) \in \left(D_1 \setminus  D_2\right) \setminus \mathop{\bigcup}\limits_{r \in \mathbb{Q}} (\pi(\omega_r\nu_r^{\infty}),\pi(\nu_{r}\omega_r^{\infty}))$$ then $(\Lambda_{(a,b)}, f_{(a,b)})$ is not transitive.\label{lemafrontera1}
\end{corollary}
\begin{proof}
This is a direct consequence of \cite[Theorem 2.13]{sidorov1}, \cite[Theorem 3.8]{sidorov6}, Theorem \ref{renorm1} and Theorem \ref{balanceadas}.
\end{proof}

\begin{theorem}
Let $(\boldsymbol{\alpha},\boldsymbol{\beta}) \in \mathcal{LW}$ be a renormalisable pair by $\omega$ and $\nu$ with $\ell(\omega)+ \ell(\nu) > 4$ and $(0\boldsymbol{\alpha}, 1\boldsymbol{\beta}) = (\omega\nu^{\infty}, \nu\omega^{\infty})$. Then, if the hole $(a,b)$ corresponding to $(\boldsymbol{\alpha},\boldsymbol{\beta})$ satisfies that $(a,b) \in D_2$ then $(\Sigma_{(\boldsymbol{\alpha},\boldsymbol{\beta})}, \sigma_{(\boldsymbol{\alpha},\boldsymbol{\beta})})$ is not transitive.\label{renorm1prime}
\end{theorem}
\begin{proof}
From Proposition \ref{vyw} we know that $\omega^{\infty}$ and $\nu^{\infty} \in \Sigma_{(\boldsymbol{\alpha},\boldsymbol{\beta})}$. Also, from Proposition \ref{vyw1} we have that ${\max_{\nu}}^{\infty} \prec {\max_{\omega}}^{\infty}$. Then $\max_{\omega}$ and $\max_{\nu} \in \mathcal{L}
(\Sigma_{(\boldsymbol{\alpha},\boldsymbol{\beta})})$. Note that $(\max_{\omega}\max_{\nu})^{\infty} \in 
\Sigma_{(\boldsymbol{\alpha},\boldsymbol{\beta})}$ since ${\max_{\omega}}^{\infty} \prec \boldsymbol{\alpha}$, ${\max_{\nu}}^{\infty} \prec \boldsymbol{\alpha}$, ${\max_{\omega}}_{\ell(\omega)} = 0$ and ${\max_{\nu}}_{1}= 1$. Moreover, 
\begin{center} 
$\sigma^{n}((\max_{\omega}\max_{\nu})^{\infty}) \succ \nu^{\infty}$ and $\sigma^{n}((\max_{\omega}\max_{\nu})^{\infty}) \prec \omega^{\infty}$
\end{center}
for every $n \in \mathbb{N}$. Note that 
\begin{equation}
\nu_2\ldots \nu_{\ell(\nu)}0{\max}_{\omega}{\max}_{\nu} \notin \mathcal{L}(\Sigma_{(\boldsymbol{\alpha},\boldsymbol{\beta})})\label{lapincheecuacion}
\end{equation} since $$\nu_2\ldots\nu_{\ell(\nu)}0{\max}_{\omega}{\max}_{\nu} \prec \boldsymbol{\beta}.$$ Finally, we want to show that there is no bridge between $$a_1 \ldots a_{\ell(\omega)+\ell(\nu)-1}1 \hbox{\rm{ and }} ({\max}_{\omega}{\max}_{\nu})^2.$$ Assume that such a bridge $v$ exists, i.e. $$a_1 \ldots a_{\ell(\omega)+\ell(\nu)-1}1 v ({\max}_{\omega}{\max}_{\nu})^2 \in\mathcal{L}(\Sigma_{(\boldsymbol{\alpha},\boldsymbol{\beta})})$$ and $v \in \mathcal{L}(\Sigma_{(\boldsymbol{\alpha},\boldsymbol{\beta})})$. Observe that $v_1 \ldots v_{\ell(\nu)-1} = \nu_2 \ldots \nu_{\ell(\nu)}$ since if $v_i \neq \nu_{i-1}$ then $$a_1 \ldots a_{\ell(\omega)+\ell(\nu)-1)}1v_1\ldots v_i \succ \boldsymbol{\alpha}$$ if $v_i =1$ or  $v_1 \ldots v_{\ell(\nu)-1} \prec b$ if $v_{i} = 0$. Then it is necessary to consider two sub-cases. Assume that $v_{\ell(\nu)} = 1$. If $v_{\ell(\nu)-1} = 1$ then $$a_1 \ldots a_{\ell(\omega)+\ell(\nu)-1}1v_1 \ldots v_{\ell(\nu)-1} = a_1 \ldots a_{\ell(\omega)+2\ell(\nu)-1}1,$$ which implies $v_{i-1} = (\nu)^{\infty}_i$ for every $i \in \{1 \ldots 2\ell(\nu)-1\}$. Then $v_{2\ell(\nu)}$ is either $0$ or $1$. If $v_{2\ell(\nu)} = 1$ we get the previous case. Therefore $v_{2\ell(\nu)} = 0$. Note that $v_{2\ell(\nu)} = 0$ is equivalent to considering $v_{\ell(\nu)} = 0$. Consider now that $v_{\ell(\nu)} = 0$. From (\ref{lapincheecuacion}), $v_{1}\ldots v_{\ell(\nu)}\max_{\omega}\max_{\nu}$ is not admissible. Then $v_1 \ldots v_{\ell(\nu)+\ell(\omega)-1} = b_1 \ldots b_{\ell(\nu)}$. Then $v_{\ell(\nu)+\ell(\omega)} \ldots v_{\ell(\nu)+\ell(\omega)} = \omega_1\ldots \omega_{\ell(\omega)-1}$. Then if $v_{\ell(\nu)+\ell(\omega)+1} = 1$ then $a_1 \ldots a_{\ell(\nu)+\ell(\omega)+1} \prec a_1 \ldots a_{\ell(\omega)+\ell(\nu)-1}1v$ which is a contradiction.  
\end{proof}

Note that the case $\ell(\omega)+\ell(\nu) = 4$ is contained in Theorem \ref{balanceadas}. Observe that in the proof of Theorem \ref{renorm1} we did not consider the case when $(\boldsymbol{\alpha},\boldsymbol{\beta})$ is renormalisable by an infinite sequence.  

\begin{proposition}
Let $(\boldsymbol{\alpha},\boldsymbol{\beta}) \in \mathcal{LW}$ such that $\boldsymbol{\beta} = 0^n\boldsymbol{\alpha}$ for $n > 0_{\boldsymbol{\alpha}}$. Then $(\Sigma_{(\boldsymbol{\alpha},\boldsymbol{\beta})}, \sigma_{(\boldsymbol{\alpha},\boldsymbol{\beta})})$ is not transitive. \label{mueveb}
\end{proposition}
\begin{proof}
It is evident that $0^n \in \mathcal{L}(\Sigma_{(\boldsymbol{\alpha},\boldsymbol{\beta})}, \sigma_{(\boldsymbol{\alpha},\boldsymbol{\beta})})$ since $\boldsymbol{\beta} = 0^n\boldsymbol{\alpha}$. Assume that $(\Sigma_{(\boldsymbol{\alpha},\boldsymbol{\beta})}, \sigma_{(\boldsymbol{\alpha},\boldsymbol{\beta})})$ is transitive. Then, there exist $\omega \in \mathcal{L}(\Sigma_{(\boldsymbol{\alpha},\boldsymbol{\beta})})$ such that $0^n\omega0^n \in \mathcal{L}(\Sigma_{(\boldsymbol{\alpha},\boldsymbol{\beta})})$. Since $\boldsymbol{\beta} = 0^n\boldsymbol{\alpha}$ then $w_i = a_i$ for all $i \in \{1 \ldots \ell(\omega)\}$. Note that for every $i \in \mathbb{N}$ the block $a_1 \ldots a_i0^n$ since $n  > 0_{\boldsymbol{\alpha}}$. This implies that the block $0^na_1\ldots a_i 0^n \prec \boldsymbol{\beta}$ for every $i \in \mathbb{N}$. Therefore $0^na_1\ldots a_i 0^n \notin \mathcal{L}(\Sigma_{(\boldsymbol{\alpha},\boldsymbol{\beta})})$ which is a contradiction. Therefore $(\Sigma_{(\boldsymbol{\alpha},\boldsymbol{\beta})}, \sigma_{(\boldsymbol{\alpha},\boldsymbol{\beta})})$ is not transitive.  
\end{proof}

\begin{theorem}
If $(\boldsymbol{\alpha},\boldsymbol{\beta}) \in \mathcal{LW}$ is renormalisable by an infinite sequence then $(\Sigma_{(\boldsymbol{\alpha},\boldsymbol{\beta})}, \sigma_{(\boldsymbol{\alpha},\boldsymbol{\beta})})$ is not transitive.\label{renorminfty}
\end{theorem}
\begin{proof}
There is no loss of generality in assuming that $0\boldsymbol{\alpha} = \boldsymbol{\omega}$ and $1\boldsymbol{\beta} = \nu\boldsymbol{\omega}$ where $\boldsymbol{\omega}$ is an infinite sequence. Moreover, from Proposition \ref{mueveb} we can assume that $\nu \neq 10^{n}$ for any $n \geq 0_{\boldsymbol{\alpha}}$. Moreover, by Theorem \ref{renorm1} we know that if $\boldsymbol{\omega} = \omega^{\infty}$ then $(\Sigma_{(\boldsymbol{\alpha}, \boldsymbol{\beta})}, \sigma_{(\boldsymbol{\alpha},\boldsymbol{\beta})})$ is not transitive. Then we can assume that $\boldsymbol{\omega}$ is not a periodic sequence. Observe that $\nu^{\infty} \prec \boldsymbol{\omega}$ since $(\boldsymbol{\alpha}, \boldsymbol{\beta}) \in \mathcal{LW}$. Note that if $0_{\boldsymbol{\beta}} > 0_{\boldsymbol{\alpha}}$ then an analogous argument as the one used in the proof of Theorem \ref{mueveb} will hold. Let us assume that $0_{\boldsymbol{\beta}} = 0_{\boldsymbol{\alpha}}$. Let $n > 1_{\boldsymbol{\alpha}}$ such that $a_n = 1$, $\upsilon = a_1 \ldots a_{n-1}0 \in \mathcal{L}(\Sigma_{(\boldsymbol{\alpha}, \boldsymbol{\beta})})$, $\upsilon^{\infty} = (a_1 \ldots a_{n-1}0)^{\infty} \in \Sigma_{(\boldsymbol{\alpha}, \boldsymbol{\beta})}$ and $\upsilon$ occurs finitely many times in $\boldsymbol{\alpha}$. There is no loss of generality in assuming that $$m = \mathop{\max}\left\{n \in \mathbb{N} \mid  \upsilon^n \hbox{\rm{ is a factor of }} \boldsymbol{\alpha} \right\}$$ exists. We claim that there are no bridges between $\nu0$ and $\upsilon^{m+1}$. Assume that such bridge $\varpi$ exists, i.e $\nu0\varpi\upsilon^{m+1} \in \mathcal{L}(\Sigma_{(\boldsymbol{\alpha},\boldsymbol{\beta})})$. If $\varpi$ is a factor of $\boldsymbol{\alpha}$ then $\nu0\varpi\upsilon^{n+1} \prec \boldsymbol{\beta}$ since $\nu0\varpi\upsilon^{n+1}_{\ell(\nu)+1+\ell(\varpi)+n\ell(\upsilon)} = 0$ and $b_{\ell(\nu)+1+\ell(\varpi)+n\ell(\upsilon)} = 1$. Thus, $\varpi$ is not a factor of $\boldsymbol{\alpha}$. Since $\nu0\varpi\upsilon^{n+1}_i = b_i$ for every $1 \leq i \leq \ell(\nu)+1$ then there exist $n \in \mathbb{N}$ with $n+1 \leq \ell(\varpi)$ such that $\varpi_i = a_i$ for every $1 \leq i \leq n$ and $\varpi_{n+1} \neq a_{n+1}$. Note that if $\varpi_{n+1} = 0$ and $a_{n+1} = 1$ then $\nu0\varpi\upsilon^{n+1} \prec \boldsymbol{\beta}$ then $\varpi_{n+1} = 1$. This gives that $a_{n+1} = 0$ then $\varpi\nu^{n+1} \succ \boldsymbol{\alpha}$ which is a contradiction. This completes the proof.
\end{proof}

To characterise transitivity via the renormalisability of $(\boldsymbol{\alpha},\boldsymbol{\beta})$ it is necessary to show that $(\Sigma_{(\boldsymbol{\alpha},\boldsymbol{\beta})}, \sigma_{(\boldsymbol{\alpha},\boldsymbol{\beta})})$ is transitive provided that $(\boldsymbol{\alpha}, \boldsymbol{\beta})$ is not renormalisable.

\vspace{1em}Let $(\boldsymbol{\alpha},\boldsymbol{\beta}) \in \mathcal{LW}$. We say that $(\boldsymbol{\alpha},\boldsymbol{\beta})$ is \textit{essential} if $\boldsymbol{\alpha}$ and $\boldsymbol{\beta}$ are periodic sequences and $(\boldsymbol{\alpha},\boldsymbol{\beta})$ is not renormalisable. From \cite[Theorem 1.3]{samuel} $(\Sigma_{(\boldsymbol{\alpha},\boldsymbol{\beta})}, \sigma_{(\boldsymbol{\alpha},\boldsymbol{\beta})})$ is a subshift of finite type provided that $(\boldsymbol{\alpha},\boldsymbol{\beta})$ is an essential pair. From Proposition \ref{palabrasciclicas} we obtain that every for essential pair $(\boldsymbol{\alpha},\boldsymbol{\beta}) \in \mathcal{LW}$ there exists a pair of finite words $(\omega, \nu)$ such that $\omega = 0-{\max}_{\omega}$ , $\nu = 1-{\min}_{\nu}$ and $0\boldsymbol{\alpha} = \omega^{\infty}$ and $1\boldsymbol{\beta} = \nu^{\infty}$. The pair $(\omega, \nu)$ is called \textit{the associated pair}.

\vspace{1em}To prove that the subshift of finite type corresponding to an essential pair is transitive we need to prove a technical lemma. Let $(\boldsymbol{\alpha}, \boldsymbol{\beta})$ be an essential pair with associated pair $(\omega, \nu)$. Consider the cyclic subshifts corresponding to $\omega$ and $\nu$ respectively i.e $$(\Sigma_{(\sigma(\omega^{\infty}), \sigma_{(({1-{\min}_{\omega}})^{\infty}))}}, \sigma_{(\sigma(\omega^{\infty}), \sigma_{(({1-{\min}_{\omega}})^{\infty}))}})$$ and $$(\Sigma_{(\sigma((0-{\max}_{\nu})^{\infty}), \sigma(\nu^{\infty}))}, \sigma_{(\sigma(({0-{\max}_{\nu})^{\infty}), \sigma(\nu^{\infty}))}}).$$ Then by Proposition \ref{endings} there exist words $\omega_{\boldsymbol{\alpha}}, \nu_{\boldsymbol{\alpha}}, \omega_{\boldsymbol{\beta}}$ and $\nu_{\boldsymbol{\beta}}$ such that $\omega =  \omega_{\boldsymbol{\alpha}}\nu_{\boldsymbol{\alpha}}$, $\nu = \omega_{\boldsymbol{\beta}}\nu_{\boldsymbol{\beta}}$, $\omega_{\boldsymbol{\alpha}} = 0-{\max}_{\omega_{\boldsymbol{\alpha}}}$, $\omega_{\boldsymbol{\beta}} = 0-{\max}_{\omega_{\boldsymbol{\beta}}}$, $\nu_{\boldsymbol{\alpha}} = 1-{\min}_{\nu_{\boldsymbol{\alpha}}}$ and $\nu_{\boldsymbol{\beta}} = 1-{\min}{\nu_{\boldsymbol{\beta}}}$. Let $${p_1}_{(\boldsymbol{\alpha}, \boldsymbol{\beta})} = \mathop{\max}\left\{\omega_{\boldsymbol{\alpha}}^{\infty}, \omega_{\boldsymbol{\beta}}^{\infty} \right\}$$ and $${p_2}_{(\boldsymbol{\alpha}, \boldsymbol{\beta})} = \mathop{\min}\left\{\nu_{\boldsymbol{\alpha}}^{\infty}, \nu_{\boldsymbol{\beta}}^{\infty} \right\}.$$   

\begin{lemma}
If $(\boldsymbol{\alpha}, \boldsymbol{\beta})$ is an essential pair, then $${p_1}_{(\boldsymbol{\alpha}, \boldsymbol{\beta})}{p_2}_{(\boldsymbol{\alpha}, \boldsymbol{\beta})} \hbox{\rm{ and }} {p_2}_{(\boldsymbol{\alpha}, \boldsymbol{\beta})}{p_1}_{(\boldsymbol{\alpha}, \boldsymbol{\beta})} \in \mathcal{L}(\Sigma_{(\boldsymbol{\alpha},\boldsymbol{\beta})}).$$ Moreover, $({p_1}_{(\boldsymbol{\alpha}, \boldsymbol{\beta})}{p_2}_{(\boldsymbol{\alpha}, \boldsymbol{\beta})})^{\infty} \in \Sigma_{(\boldsymbol{\alpha}, \boldsymbol{\beta})}$.\label{p1p2}
\end{lemma}
\begin{proof} 
From the construction of ${p_1}_{\boldsymbol{\alpha}, \boldsymbol{\beta}}$ and ${p_2}_{\boldsymbol{\alpha}, \boldsymbol{\beta}}$ it is clear that $$\boldsymbol{\beta} \prec ({p_1}_{(\boldsymbol{\alpha}, \boldsymbol{\beta})}{p_2}_{(\boldsymbol{\alpha}, \boldsymbol{\beta})})^{\infty} \prec \boldsymbol{\alpha}$$ and $$\boldsymbol{\beta} \prec ({p_2}_{(\boldsymbol{\alpha}, \boldsymbol{\beta})}{p_1}_{(\boldsymbol{\alpha}, \boldsymbol{\beta})})^{\infty} \prec \boldsymbol{\alpha}.$$ Suppose that $({p_1}_{(\boldsymbol{\alpha}, \boldsymbol{\beta})}{p_2}_{(\boldsymbol{\alpha}, \boldsymbol{\beta})})^{\infty} \notin \Sigma_{(\boldsymbol{\alpha}, \boldsymbol{\beta})}.$ Then there exists $n \in \mathbb{N}$ such that $$\sigma^n(({p_1}_{(\boldsymbol{\alpha}, \boldsymbol{\beta})}{p_2}_{(\boldsymbol{\alpha}, \boldsymbol{\beta})})^{\infty}) \succ \boldsymbol{\alpha} \hbox{\rm{ or }} \sigma^n(({p_1}_{(\boldsymbol{\alpha}, \boldsymbol{\beta})}{p_2}_{(\boldsymbol{\alpha}, \boldsymbol{\beta})})^{\infty}) \prec \boldsymbol{\beta}.$$ Without losing generality let us assume the former. Then there exists $i \in \mathbb{N}$ such that $\sigma^n(({p_1}_{(\boldsymbol{\alpha}, \boldsymbol{\beta})}{p_2}_{(\boldsymbol{\alpha}, \boldsymbol{\beta})})^{\infty})_i = 1$ and $a_i = 0$. . Thus ${p_2}_{(\boldsymbol{\alpha}, \boldsymbol{\beta})} \succ \nu_{\boldsymbol{\alpha}}$, which contradicts that ${p_2}_{(\boldsymbol{\alpha}, \boldsymbol{\beta})} = \mathop{\min}\left\{(\nu_{\boldsymbol{\alpha}})^{\infty}, \nu_{\boldsymbol{\beta}})^{\infty} \right\}$. Then the only possibility is $\sigma^m(({p_1}_{(\boldsymbol{\alpha}, \boldsymbol{\beta})}{p_2}_{(\boldsymbol{\alpha}, \boldsymbol{\beta})})^{\infty}) \prec \boldsymbol{\beta}.$ Similarly there is $i^{\prime} \in \mathbb{N}$ such that $\sigma^m(({p_1}_{(\boldsymbol{\alpha}, \boldsymbol{\beta})}{p_2}_{(\boldsymbol{\alpha}, \boldsymbol{\beta})})^{\infty})_{i^{\prime}} = 0$ and $b_{i^{\prime}} = 1$. Then ${p_1}_{(\boldsymbol{\alpha}, \boldsymbol{\beta})} \prec \omega_{\boldsymbol{\alpha}}$, which contradicts that ${p_1}_{(\boldsymbol{\alpha}, \boldsymbol{\beta})} = \mathop{\min}\left\{\nu_{\boldsymbol{\alpha}},\nu_{\boldsymbol{\beta}}\right\}$. This completes the proof.
\end{proof}

\begin{theorem}
If $(\boldsymbol{\alpha},\boldsymbol{\beta}) \in \mathcal{LW}$ is an essential pair then 
$(\Sigma_{(\boldsymbol{\alpha},\boldsymbol{\beta})}, \sigma_{(\boldsymbol{\alpha},\boldsymbol{\beta})})$ is transitive. 
\label{renorm2}
\end{theorem}
\begin{proof}
Let $(\boldsymbol{\alpha},\boldsymbol{\beta}) \in \mathcal{LW}$ be an essential pair with associated pair $(\omega, \nu)$. Let ${p_1}_{(\boldsymbol{\alpha}, \boldsymbol{\beta})}$ and ${p_2}_{(\boldsymbol{\alpha}, \boldsymbol{\beta})}$ be given by Lemma \ref{p1p2}. Note that $\Sigma_{\mathcal{F}} \subset \Sigma_{(\boldsymbol{\alpha},\boldsymbol{\beta})}$ where $\mathcal{F} = \{0^{0_{\boldsymbol{\alpha}}},1^{1_{\boldsymbol{\beta}}}\}$. From Proposition \ref{ceroandone}, $(\Sigma_{\mathcal{F}}, \sigma_{\mathcal{F}})$ is a transitive subshift of finite type. Then it is needed to show that for $\upsilon, \nu \in \mathcal{L}(\Sigma_{(\boldsymbol{\alpha},\boldsymbol{\beta})})$ such that $1_{\boldsymbol{\alpha}}$ or $0_{\beta}$ is a factor of $\upsilon$ or $\nu$ there exist a bridge $\omega$ from $\upsilon$ to $\nu$. We claim that for any pair $\upsilon , \nu \in \mathcal{L}(\Sigma_{(\boldsymbol{\alpha},\boldsymbol{\beta})})$, ${p_1}_{(\boldsymbol{\alpha},\boldsymbol{\beta})}{p_2}_{(\boldsymbol{\alpha},\boldsymbol{\beta})}$ and ${p_2}_{(\boldsymbol{\alpha},\boldsymbol{\beta})}{p_1}_{(\boldsymbol{\alpha},\boldsymbol{\beta})}$ are bridges between $\upsilon$ and $\nu$. 

Consider $\upsilon, \nu \in \mathcal{L}(\Sigma_{(\boldsymbol{\alpha},\boldsymbol{\beta})})$. Observe that there exist $\varpi, \varpi^{\prime} \in \mathcal{L}(\Sigma_{(\boldsymbol{\alpha},\boldsymbol{\beta})})$ such that the words $\upsilon\varpi, \upsilon\varpi^{\prime} \in \mathcal{L}(\Sigma_{(\boldsymbol{\alpha},\boldsymbol{\beta})})$ and ${p_1}_{(\boldsymbol{\alpha},\boldsymbol{\beta})}$ is a factor of $\upsilon\varpi$ such that $\sigma^j(\upsilon\varpi) = {p_1}_{(\boldsymbol{\alpha},\boldsymbol{\beta})}$ and ${p_2}_{(\boldsymbol{\alpha},\boldsymbol{\beta})}$ is a factor of $\upsilon\varpi^{\prime}$ such that $\sigma^k(\upsilon\varpi) = {p_2}_{(\boldsymbol{\alpha},\boldsymbol{\beta})}$ for some $j \ell(\upsilon) < j < \ell(\upsilon\varpi)$ and $\ell(\upsilon)< k < \ell(\upsilon\varpi^{\prime})$. Also, there exist words $\epsilon, \epsilon^{\prime} \in \mathcal{L}(\Sigma_{(\boldsymbol{\alpha},\boldsymbol{\beta})})$ such that the words $\epsilon\nu, \epsilon^{\prime}\nu \in \mathcal{L}(\Sigma_{(\boldsymbol{\alpha},\boldsymbol{\beta})})$ and  ${p_1}_{(\boldsymbol{\alpha},\boldsymbol{\beta})}$ is a factor of $\epsilon\nu$ such that ${\epsilon\nu}_{i} =  {{p_1}_{(\boldsymbol{\alpha},\boldsymbol{\beta})}}_i$ for every $1 \leq i \leq \ell({p_1}_{(\boldsymbol{\alpha},\boldsymbol{\beta})})$ and ${p_2}_{(\boldsymbol{\alpha},\boldsymbol{\beta})}$ is a factor of $\epsilon^{\prime}\nu$ such that ${\epsilon^{\prime}\nu}_{j} =  {{p_2}_{(\boldsymbol{\alpha},\boldsymbol{\beta})}}_j$ for every $1 \leq j \leq \ell({p_2}_{(\boldsymbol{\alpha},\boldsymbol{\beta})})$. Thus, the words $\upsilon\varpi\epsilon^{\prime}\nu, \upsilon\varpi^{\prime}\epsilon\nu \in \mathcal{L}(\Sigma_{(\boldsymbol{\alpha},\boldsymbol{\beta})})$, which implies that $(\Sigma_{(\boldsymbol{\alpha},\boldsymbol{\beta})}, \sigma_{(\boldsymbol{\alpha},\boldsymbol{\beta})})$ is a transitive subshift of finite type.
\end{proof}

\subsection*{Transitivity of limits of essential pairs}

To finish this section, we prove that every non renormalisable pair $(\boldsymbol{\alpha},\boldsymbol{\beta}) \in \mathcal{LW}$ corresponds to a transitive lexicographic subshift. 
It is clear that given a sequence $\left\{({\boldsymbol{\alpha}}_i,{\boldsymbol{\beta}}_i)\right\}_{i=1} \subset \mathcal{LW}$ such that for every $i \in \mathbb{N}$, ${\boldsymbol{\alpha}}_{i+1} \prec {\boldsymbol{\alpha}}_i$, $\boldsymbol{\beta}_{i+1} \succ \boldsymbol{\beta}_i$, $\boldsymbol{\alpha}_i \to \boldsymbol{\alpha}$, $\boldsymbol{\beta}_i \to \boldsymbol{\beta}$ and $(\boldsymbol{\alpha}_i,\boldsymbol{\beta}_i)$ is an essential pair then $(\Sigma_{(\boldsymbol{\alpha},\boldsymbol{\beta})}, \sigma_{(\boldsymbol{\alpha}, \boldsymbol{\beta})})$ is transitive. In particular, $(\Sigma_{(\boldsymbol{\alpha},\boldsymbol{\beta})}, \sigma_{(\boldsymbol{\alpha}, \boldsymbol{\beta})})$ is coded. We show now that every non renormalisable pair $(\boldsymbol{\alpha},\boldsymbol{\beta}) \in \mathcal{LW}$ is a coded system.

\begin{theorem}
If $(\boldsymbol{\alpha},\boldsymbol{\beta}) \in \mathcal{LW}$ is non renormalisable, then 
$(\Sigma_{(\boldsymbol{\alpha},\boldsymbol{\beta})}, \sigma_{(\boldsymbol{\alpha},\boldsymbol{\beta})})$ is coded. 
\label{lemaaproxtrans1}  
\end{theorem}
\begin{proof}
From Theorem \ref{renorm2} we just need to show the case when $\boldsymbol{\alpha}$ or $\boldsymbol{\beta}$ are not periodic sequences. Firstly, assume that $\boldsymbol{\alpha}$ is not a periodic sequence and $\boldsymbol{\beta}$ is. Let $i_1$ such that $i_1 > 1_{\boldsymbol{\alpha}}$ and $a_{i_1} = 1$. Let $i_{2} > i_1$ such that $a_{i_2} = 1$. Then we define inductively the sequence $\{i_m\}_{m=1}^{\infty}$ as $i_m > i_{m-1}$ and $a_{i_m} = 1$. We define $\{\boldsymbol{\alpha}_m\}_{m=1}^{\infty}$ as $\boldsymbol{\alpha}_m = (a_1 \ldots a_{i_m-1}0)^{\infty}$. From the construction it is clear that $\{\boldsymbol{\alpha}_m\}_{m=1}^{\infty} \in Per(\sigma) \cap P$, $\sigma^n(\boldsymbol{\alpha}_m) \prec \boldsymbol{\alpha}$ for every $n, m \in \mathbb{N}$, $\boldsymbol{\alpha}_m \prec \boldsymbol{\alpha}_{m+1}$ for every $m \in \mathbb{N}$, and $\boldsymbol{\alpha}_m \mathop{\longrightarrow}\limits_{m \to \infty} \boldsymbol{\alpha}$. Let $\{\boldsymbol{\alpha}_{m_k}\}_{k=1}^{\infty} \subset \{\boldsymbol{\alpha}_m\}_{m=1}^{\infty}$ such that $\boldsymbol{\alpha}_{m_k} \in \Sigma_{(\boldsymbol{\alpha},\boldsymbol{\beta})}$ for every $k \in \mathbb{N}$ and $\boldsymbol{\alpha}_{m_k} \mathop{\longrightarrow}\limits_{k \to \infty} \boldsymbol{\alpha}$. Then $(\boldsymbol{\alpha}_{m_k},\boldsymbol{\beta}) \in \mathcal{LW}$. We claim that there exist $K \in \mathbb{N}$ such that, for every $k \geq K$, $(\Sigma_{(\boldsymbol{\alpha}_{m_k},\boldsymbol{\beta})}, \sigma_{(\boldsymbol{\alpha}_{m_k},\boldsymbol{\beta})})$ is transitive. The claim implies that $\Sigma_{(\boldsymbol{\alpha},\boldsymbol{\beta})} = \overline{\mathop{\bigcup}\limits_{k=K}^{\infty}\Sigma_{(\boldsymbol{\alpha}_{m_k},\boldsymbol{\beta})}}$ and hence, $(\Sigma_{(\boldsymbol{\alpha},\boldsymbol{\beta})}, \sigma_{(\boldsymbol{\alpha},\boldsymbol{\beta})})$ is coded. To prove the claim, assume that it is not true. This implies that there exist infinitely many $k \in \mathbb{N}$ such that $(\Sigma_{(\boldsymbol{\alpha}_{m_k},\boldsymbol{\beta}}, \sigma_{(\boldsymbol{\alpha}_{m_k},\boldsymbol{\beta})})$ is not transitive for infinitely many $k \in \mathbb{N}$. Then from Theorem \ref{renorm2}, $(\boldsymbol{\alpha}_{m_k},\boldsymbol{\beta})$ is renormalisable. Then $(\boldsymbol{\alpha},\boldsymbol{\beta})$ is renormalisable, which is a contradiction. 

Note that if $\boldsymbol{\beta}$ is not a periodic sequence and $\boldsymbol{\alpha}$ is, it is possible to develop a similar construction to the previous one considering the sequence $\{j_n\}_{n=1}^{\infty}$ defined as follows: Let $j_1$ such that $j_1 > 0_b$ and $b_{j_1} = 0$. Let $j_2 > j_1$ such that $b_{j_2} = 0$. Inductively, we define $\{j_n\}_{n=1}^{\infty}$ as $j_n > j_{n-1}$ and $b_{j_n} = 0$. Then $\boldsymbol{\beta}_n = (b_1 \ldots b_{j_n}1)^{\infty}$. Then $\{\boldsymbol{\beta}_n\}_{n=1}^{\infty} \in Per(\sigma) \cap \bar P$, $\sigma^m(\boldsymbol{\beta}_n) \succ \boldsymbol{\beta}$ for every $n,m \in \mathbb{N}$, $\boldsymbol{\beta}_{n} < 
\boldsymbol{\beta}_{n-1}$ and $\boldsymbol{\beta}_n \mathop{\longrightarrow}\limits_{n \to \infty} 
\boldsymbol{\beta}$.

\vspace{1em}Assume now that $\boldsymbol{\alpha}$ and $\boldsymbol{\beta}$ are not periodic. Observe that the sequences $\{j_m\}_{m=1}^{\infty}$ and  $\{i_n\}_{n=1}^{\infty}$ can be defined as we showed before. Then there exist a sequence $\left\{\boldsymbol{\alpha}_k,\boldsymbol{\beta}_k\right\}_{k=1}^{\infty} \subset \left\{\boldsymbol{\alpha}_m, \boldsymbol{\beta}_n\right\}_{n,m \in \mathbb{N}}$ such that $(\boldsymbol{\alpha}_k,\boldsymbol{\beta}_k) \in \mathcal{LW}$, $(\boldsymbol{\alpha}_k,\boldsymbol{\beta}_k) \mathop{\longrightarrow}\limits_{n\to \infty} (\boldsymbol{\alpha},\boldsymbol{\beta})$, $(\Sigma_{(\boldsymbol{\alpha}_k,\boldsymbol{\beta}_k)}, \sigma_{(\boldsymbol{\alpha}_k,\boldsymbol{\beta}_k)})$ is a transitive subshift of finite type and $\Sigma_{(\boldsymbol{\alpha},\boldsymbol{\beta})} = \overline{\mathop{\bigcup}\limits_{k=1}^{\infty} \Sigma_{(\boldsymbol{\alpha}_k,\boldsymbol{\beta}_k)}}.$ 
\end{proof}

\begin{theorem}
If $(\boldsymbol{\alpha},\boldsymbol{\beta})$ is not renormalisable then $(\Sigma_{(\boldsymbol{\alpha},\boldsymbol{\beta})}, \sigma_{(\boldsymbol{\alpha},\boldsymbol{\beta})})$ is transitive. 
\label{aproximaporabajogeneral}
\end{theorem}
\begin{proof}
This is an immediate consequence of Theorem \ref{renorm2} and Theorem \ref{lemaaproxtrans1}.
\end{proof}

From \cite[Proposition 4.5.10 (4)]{lindmarcus}, \cite[Proposition 2.6]{sidorov6} and Theorem \ref{asociar} we obtain that if $(\boldsymbol{\alpha},\boldsymbol{\beta}) \in \mathcal{LW}$ is not renormalisable and the associated attractor $(\Lambda_{(a,b)}, f_{(a,b)})$ satisfies that $(a, b) \in D_3$ then $(\Sigma_{(\boldsymbol{\alpha},\boldsymbol{\beta})}, \sigma_{(\boldsymbol{\alpha},\boldsymbol{\beta})})$ is topologically mixing. Also, From \cite[Theorem 2.3]{sidorov6} and the structure of $\partial D_3$ -see \cite[p. 353]{sidorov6} we are sure that the associated $(\boldsymbol{\alpha},\boldsymbol{\beta})$ to $(a,b) \in \partial D_3$ are not renormalisable, then by Theorem \ref{renorm2} and Theorem \ref{aproximaporabajogeneral} $(\Lambda_{(a,b)},f_{(a,b)})$ is transitive.

\section{Specification Properties}
\label{specification}

\noindent In \cite[Proposition 2.1]{buzzi}, Buzzi proved a criterion which determines when a piecewise monotonic map has the specification property. It is clear that we can apply this criterion to expanding Lorenz maps in order to determine when they have the specification property. However, as topological transitivity, it is not immediate that $(\Sigma_{(\boldsymbol{\alpha},\boldsymbol{\beta})}, \sigma_{(\boldsymbol{\alpha},\boldsymbol{\beta})})$ has the specification property if the associated Lorenz dynamical system $([0,1], g_{(a,b)})$ has it. 

\vspace{1em}During this section, we give sufficient condition to determine when a lexicographic subshift has the specification property. Also, we construct a family of asymmetric subshifts with no specification.

\vspace{1em}Recall that a transitive lexicographic subshift $(\Sigma_{(\boldsymbol{\alpha},\boldsymbol{\beta})}, \sigma_{(\boldsymbol{\alpha},\boldsymbol{\beta})})$ has specification if there exist $M \in \mathbb{N}$ such that for every $\omega, \nu \in \mathcal{L}(\Sigma_{(\boldsymbol{\alpha},\boldsymbol{\beta})})$ there is $\upsilon \in \mathcal{L}(\Sigma_{(\boldsymbol{\alpha},\boldsymbol{\beta})})$ such that $\omega\upsilon\nu \in \mathcal{L}(\Sigma_{(\boldsymbol{\alpha},\boldsymbol{\beta})})$ and $\ell(\upsilon) = M$. We can rephrase this definition as follows: Let 
\begin{align*}
m_n = \inf\{k \in \mathbb{N} \mid &\hbox{\rm{ for every }} \omega, \nu \in B_n(\Sigma_{(\boldsymbol{\alpha},\boldsymbol{\beta})})\\ 
&\hbox{\rm{ there exist }} \upsilon \in B_k(\Sigma_{(\boldsymbol{\alpha},\boldsymbol{\beta})})\hbox{\rm{ such that }} \omega \upsilon \nu \in \mathcal{L}(\Sigma_{(\boldsymbol{\alpha},\boldsymbol{\beta})})\}.
\end{align*}

Then $(\Sigma_{(\boldsymbol{\alpha},\boldsymbol{\beta})}, \sigma_{(\boldsymbol{\alpha},\boldsymbol{\beta})})$ has the specification property if and only if $\mathop{\lim}\limits_{n \to \infty} m_n < \infty$. Recall that transitive subshifts of finite type have the specification property \cite{parry}. Then for every essential pair $(\boldsymbol{\alpha},\boldsymbol{\beta})$ we define \textit{specification number of $(\Sigma_{\boldsymbol{\alpha},\boldsymbol{\beta}}, \sigma_{\boldsymbol{\alpha},\boldsymbol{\beta}})$} denoted by $s_{(\boldsymbol{\alpha},\boldsymbol{\beta})}$ to be $\mathop{\lim}\limits_{n \to \infty} m_n$.

\begin{lemma}
Let $(a,b) \in D_1$ and $(\boldsymbol{\alpha}, \boldsymbol{\beta}) \in \mathcal{LW}$ be such that $(\Lambda_{(a,b)}, f_{(a,b)})$ is conjugated to $(\Sigma_{(\boldsymbol{\alpha}, \boldsymbol{\beta})},\sigma_{(\boldsymbol{\alpha}, \boldsymbol{\beta})})$. Then for every $m \in \mathbb{N}$ there exists $N \in \mathbb{N}$ such that $B_m(\Sigma_{(\boldsymbol{\alpha}, \boldsymbol{\beta})}) = B_m(\Sigma_{({\boldsymbol{\alpha}}_n, {\boldsymbol{\beta}}_n)}) = B_m(\Sigma_{({\boldsymbol{\alpha}}_N, {\boldsymbol{\beta}}_N)})$ for every $n \geq N$, where $({\boldsymbol{\alpha}}_n, {\boldsymbol{\beta}}_n)$ satisfies: 
\begin{enumerate}[$i)$]
\item $(\Sigma_{({\boldsymbol{\alpha}}_n, {\boldsymbol{\beta}}_n)}, \sigma_{({\boldsymbol{\alpha}}_n, {\boldsymbol{\beta}}_n)})$ is a shift of finite type for every $n \in \mathbb{N}$; 
\item ${\boldsymbol{\alpha}}_n  \preccurlyeq {\boldsymbol{\alpha}}_{n+1} \preccurlyeq \boldsymbol{\alpha}$ and ${\boldsymbol{\beta}}_n  \succcurlyeq {\boldsymbol{\beta}}_{n+1} \succcurlyeq \boldsymbol{\beta};$ 
\item ${\boldsymbol{\alpha}}_n \underset{n \to \infty}\longrightarrow \boldsymbol{\alpha}$ and ${\boldsymbol{\beta}}_n \underset{n \to \infty}\longrightarrow \boldsymbol{\beta}.$ 
\end{enumerate}
\label{stabilityproperty}
\end{lemma}
\begin{proof}
Assume that $(a,b) \in D_1$ and consider the $(\boldsymbol{\alpha}, \boldsymbol{\beta})$ satisfying the hypothesis of the Lemma.  Observe that if $(\Sigma_{(\boldsymbol{\alpha}, \boldsymbol{\beta})},\sigma_{(\boldsymbol{\alpha}, \boldsymbol{\beta})})$ then the constant sequence $(\boldsymbol{\alpha}_n, \boldsymbol{\beta}_n) = (\boldsymbol{\alpha}, \boldsymbol{\beta})$ for every $n \in \mathbb{N}$ gives us the desired conclusion.  

\vspace{0.5em}Let $(\boldsymbol{\alpha}, \boldsymbol{\beta})$ be such that $(\Sigma_{(\boldsymbol{\alpha}, \boldsymbol{\beta})}, \sigma_{(\boldsymbol{\alpha}, \boldsymbol{\beta})})$ is not a subshift of finite type. Let $$n_1 = \mathop{\min}\left\{k > 1_{\boldsymbol{\alpha}} \mid a_k =1 \hbox{\rm{ and }} (a_1\ldots a_{k-1}0)^{\infty} \in \Sigma_{(\boldsymbol{\alpha},\boldsymbol{\beta})}\right\}$$ and $$m_1 = \mathop{\min}\left\{j >  0_{\boldsymbol{\beta}} \mid b_j = 0 \hbox{\rm{ and }}(b_1\ldots b_{j-1}1)^{\infty} \in \Sigma_{(\boldsymbol{\alpha},\boldsymbol{\beta})}\right\}.$$ Let consider ${\boldsymbol{\alpha}}_1 = (a_1 \ldots a_{n_1 - 1}0)^{\infty}$ and $$m^{1}_1 = \left\{
\begin{array}{clrr}      
m_1 & \hbox{\rm{ if }}  (b_1\ldots b_{m_1 - 1}1)^{\infty} \in \Sigma_{({\boldsymbol{\alpha}}_1,\boldsymbol{\beta})};\\

\mathop{\min}\left\{j >  m_1 \mid b_j = 0 \hbox{\rm{ and }}(b_1\ldots b_{j-1}1)^{\infty} \in \Sigma_{({\boldsymbol{\alpha}}_1,\boldsymbol{\beta})}\right\} & \hbox{\rm{ otherwise.}} \\
\end{array}
\right.$$

Let ${\boldsymbol{\beta}}^1_1 = (b_1 \ldots b_{m^1_1 - 1}1)^{\infty}.$ Observe that $(\Sigma_{(\boldsymbol{\alpha_1})},\sigma_{{\boldsymbol{\beta}}^1_1})$ is a subshift of finite type. Inductively, let \begin{align}
m^1_n = \mathop{\min}\left\{j >  m^1_{n-1} \mid b_j = 0 \hbox{\rm{ and }}(b_1\ldots b_{j-1}1)^{\infty} \in \Sigma_{({\boldsymbol{\alpha}}_1,\boldsymbol{\beta})}\right\} \label{elno1}
\end{align}
Let ${\boldsymbol{\beta}}^1_n = (b_1 \ldots b_{m^1_n - 1}1)^{\infty}$. It is clear that or every $n \in \mathbb{N}$ $(\Sigma_{(\boldsymbol{\alpha_1},{\boldsymbol{\beta}}^1_n)},\sigma_{(\boldsymbol{\alpha_1},{\boldsymbol{\beta}}^1_n)})$ is a subshift of finite type. Now consider for every $l \geq 2$, $$n_{l} = \mathop{\min}\left\{k > n_{l-1} \mid a_k =1 \hbox{\rm{ and }} (a_1\ldots a_{k-1}0)^{\infty} \in \Sigma_{(\boldsymbol{\alpha},\boldsymbol{\beta})}\right\}$$ and $$m_{l} = \mathop{\min}\left\{j >  m_{l-1} \mid b_j = 0 \hbox{\rm{ and }}(b_1\ldots b_{j-1}1)^{\infty} \in \Sigma_{(\boldsymbol{\alpha},\boldsymbol{\beta})}\right\}.$$ Then, using a similar argument as in (\ref{elno1}) we construct for every $l \in \mathbb{N}$ get a sequences of subshifts of finite type $(\Sigma_{(\boldsymbol{\alpha_l},{\boldsymbol{\beta}}^l_n)}, \sigma_{(\boldsymbol{\alpha_l},{\boldsymbol{\beta}}^l_n)})$. Let $(\boldsymbol{\alpha_n},{\boldsymbol{\beta}}_n) = (\boldsymbol{\alpha_n},{\boldsymbol{\beta}}^n_n)$ for every $n \in \mathbb{N}$. From the construction we see that the sequence $\left\{(\boldsymbol{\alpha_n},{\boldsymbol{\beta}}_n)\right\}_{n=1}^{\infty}$ satisfies $i)$, $ii)$ and $iii)$.

\vspace{0.5em}Observe that $\Sigma_{(\boldsymbol{\alpha},\boldsymbol{\beta})} = \overline{\mathop{\bigcup}\limits_{n=1}^{\infty} \Sigma_{({\boldsymbol{\alpha}}_n,{\boldsymbol{\beta}}_n)}}$. Let $M_1 = \mathop{\min}\left\{\ell(\boldsymbol{\alpha}_1-1), \ell(\boldsymbol{\beta}_1-1)\right\}$. Observe that for every $n \in \mathbb{N}$ and $1 \leq m < M_1$, $B_m(\Sigma_{({\boldsymbol{\alpha}},{\boldsymbol{\beta}})}) = B_m(\Sigma_{({\boldsymbol{\alpha}}_n,{\boldsymbol{\beta}}_n)}) = B_m(\Sigma_{({\boldsymbol{\alpha}}_1,{\boldsymbol{\beta}}_1)})$ since $a_m = {a_n}_m = {a_1}_m$ and $a_m = {a_n}_m = {a_1}_m$. Assume that $m \geq M_1$. 

Let $N \in \mathbb{N}$ be such that $$\mathop{\max}\left\{\ell(\boldsymbol{\alpha}_{N})-1, \ell(\boldsymbol{\beta}_{N})-1\right\} < m \leq \mathop{\min}\left\{\ell(\boldsymbol{\alpha}_{N+1})-1, \ell(\boldsymbol{\beta}_{N+1})-1\right\}.$$ Let $n > N$. We will show that $|B_m(\Sigma_{({\boldsymbol{\alpha}}_n,{\boldsymbol{\beta}}_n)})| = |B_m(\Sigma_{({\boldsymbol{\alpha}}_N,{\boldsymbol{\beta}}_N)})|$. Assume that the former does not hold. Since $\Sigma_{({\boldsymbol{\alpha}}_N,{\boldsymbol{\beta}}_N)} \subset \Sigma_{({\boldsymbol{\alpha}}_n,{\boldsymbol{\beta}}_n)}$ then $|B_m(\Sigma_{({\boldsymbol{\alpha}}_n,{\boldsymbol{\beta}}_n)})| > |B_m(\Sigma_{({\boldsymbol{\alpha}}_N,{\boldsymbol{\beta}}_N)})|$. Then there exists $$\upsilon \in B_m(\Sigma_{({\boldsymbol{\alpha}}_{n},{\boldsymbol{\beta}}_{n})}) \setminus B_m(\Sigma_{({\boldsymbol{\alpha}}_{N},{\boldsymbol{\beta}}_{N})})$$ such that $\upsilon^{m_N}_{\max} \prec \upsilon \prec \upsilon^{m_n}_{\max}$ and $\upsilon^{m_n}_{\min} \prec \upsilon \prec \upsilon^{m_N}_{\min}$ where $\upsilon^{m_N}_{\max}$ is the maximal admissible word of length $m$ in $\mathcal{L}(\Sigma_{({\boldsymbol{\alpha}}_{N},{\boldsymbol{\beta}}_{N})})$, $\upsilon^{m_N}_{\min}$ is the minimal admissible word of length $m$ in $\mathcal{L}(\Sigma_{({\boldsymbol{\alpha}}_{N},{\boldsymbol{\beta}}_{N})})$, $\upsilon^{m_n}_{\max}$ is the maximal admissible word of length $m$ in $\mathcal{L}(\Sigma_{({\boldsymbol{\alpha}}_{n},{\boldsymbol{\beta}}_{n})})$ and $\upsilon^{m_n}_{\min}$ is the minimal admissible word of length $m$ in $\mathcal{L}(\Sigma_{({\boldsymbol{\alpha}}_{n},{\boldsymbol{\beta}}_{n})})$. This implies that there exist $i \leq m$ such that either ${\upsilon^{m_n}_{\max}}_i = 1$ and $\upsilon_i = 0$ or ${\upsilon^{m_n}_{\min}}_1 = 0$ and $\upsilon_i = 1$. Note that the first $m$ symbols of $\upsilon^{m_N}_{\max}$ and $\upsilon^{m_N}_{\min}$ coincide with $\upsilon^{m_n}_{\max}$ and $\upsilon^{m_n}_{\min}$ respectively. Then $\upsilon \prec \upsilon^{m_n}_{\min}$ or $\upsilon \succ \upsilon^{m_n}_{\max}$ which is a contradiction. Then our result holds. 
\end{proof}

Note that there are different sequences than $\left\{(\boldsymbol{\alpha}_n, \boldsymbol{\beta}_n)\right\}_{n=1}^{\infty}$ satisfying the properties of Lemma \ref{stabilityproperty}. In particular, if $(\Sigma_{({\boldsymbol{\alpha}},{\boldsymbol{\beta}})}, \sigma_{({\boldsymbol{\alpha}}_,{\boldsymbol{\beta}})})$ is coded, then the sequence constructed in Theorem \ref{lemaaproxtrans1} satisfies the properties stated in Lemma \ref{stabilityproperty}.  

\subsection*{Examples with the specification property}

\begin{theorem}
Let $(\Sigma_{({\boldsymbol{\alpha}},{\boldsymbol{\beta}})}, \sigma_{({\boldsymbol{\alpha}}_,{\boldsymbol{\beta}})})$ be a coded system and $\left\{({\boldsymbol{\alpha}}_n, {\boldsymbol{\beta}}_n)\right\}_{n=1}^{\infty}$ be the sequence constructed in Theorem \ref{lemaaproxtrans1}. If there is $N \in \mathbb{N}$ such that for every $n \geq N$, $${p_1}_{(\boldsymbol{\alpha}_n,\boldsymbol{\beta}_n)}{p_2}_{(\boldsymbol{\alpha}_n,\boldsymbol{\beta}_n)} = {p_1}_{(\boldsymbol{\alpha}_N,\boldsymbol{\beta}_N)}{p_2}_{(\boldsymbol{\alpha}_N,\boldsymbol{\beta}_N)}$$ then $(\Sigma_{(\boldsymbol{\alpha},\boldsymbol{\beta})},\sigma_{(\boldsymbol{\alpha},\boldsymbol{\beta})})$ has specification. \label{sispec}
\end{theorem}
\begin{proof}
The modification of the definition of specification combined with Lemma \ref{stabilityproperty} gives that $m_n = s_{({\boldsymbol{\alpha}}_{N(n)}, {\boldsymbol{\beta}}_{N(n)})}$ for every $n \in \mathbb{N}$. Thus, a coded system $(\Sigma_{({\boldsymbol{\alpha}},{\boldsymbol{\beta}})}, \sigma_{({\boldsymbol{\alpha}}_,{\boldsymbol{\beta}})})$ has specification if and only if $\mathop{\lim}\limits_{n\to \infty}s_{({\boldsymbol{\alpha}}_{n}, {\boldsymbol{\beta}}_{n})}$ is bounded. Observe that from the proof of Theorem \ref{renorm2} we have that $s_{({\boldsymbol{\alpha}}_{n}, {\boldsymbol{\beta}}_{n})} \leq \ell({p_1}_{({\boldsymbol{\alpha}}_{n}, {\boldsymbol{\beta}}_{n})}{p_2}_{({\boldsymbol{\alpha}}_{n}, {\boldsymbol{\beta}}_{n})})$. From the hypothesis there is $N \in \mathbb{N}$ such that for every $n \leq N$, $${p_1}_{(\boldsymbol{\alpha}_n,\boldsymbol{\beta}_n)}{p_2}_{(\boldsymbol{\alpha}_n,\boldsymbol{\beta}_n)} = {p_1}_{(\boldsymbol{\alpha}_N,\boldsymbol{\beta}_N)}{p_2}_{(\boldsymbol{\alpha}_N,\boldsymbol{\beta}_N)}.$$ Then $s_{({\boldsymbol{\alpha}}_{n}, {\boldsymbol{\beta}}_{n})} \leq \ell({p_1}_{({\boldsymbol{\alpha}}_{N}, {\boldsymbol{\beta}}_{N})}{p_2}_{({\boldsymbol{\alpha}}_{N}, {\boldsymbol{\beta}}_{N})})$ for every $n \geq N$. Thus $\mathop{\lim}\limits_{n \to \infty}s_{({\boldsymbol{\alpha}}_{n}, {\boldsymbol{\beta}}_{n})} < \infty$. Then $(\Sigma_{(\boldsymbol{\alpha},\boldsymbol{\beta})},\sigma_{(\boldsymbol{\alpha},\boldsymbol{\beta})})$ has specification.
\end{proof}

We now construct a family of examples satisfying Theorem \ref{sispec}. To perform this construction it is needed to show the following result.

\begin{lemma}
Let $(\boldsymbol{\alpha}, \boldsymbol{\beta}), ({\boldsymbol{\alpha}}^{\prime}, {\boldsymbol{\beta}}^{\prime}) \in \mathcal{LW}$ be essential pairs with associated pairs $(\omega,\nu)$ and $(\omega^{\prime},\nu^{\prime})$ respectively such that $(\boldsymbol{\alpha}, \boldsymbol{\beta}) \neq ({\boldsymbol{\alpha}}^{\prime}, {\boldsymbol{\beta}}^{\prime})$. Suppose that ${\omega^{\prime}}^{\infty} \prec \omega^{\infty}$ and $\nu^{\infty} \prec {\nu^{\prime}}^{\infty}$. Then the pair $({\boldsymbol{\alpha}}^{\prime \prime}, {\boldsymbol{\beta}}^{\prime \prime})$ given by $0{\boldsymbol{\alpha}}^{\prime \prime} = (\omega\nu^{\prime})^{\infty}$ and $1{\boldsymbol{\beta}}^{\prime \prime} = (\nu\omega^{\prime})^{\infty}$ is an essential pair. \label{constructionspec1}
\end{lemma} 
\begin{proof}
Since ${\omega^{\prime}}^{\infty} \prec \omega^{\infty}$ and $\nu^{\infty} \prec {\nu^{\prime}}^{\infty}$ it is clear that $(\boldsymbol{\alpha}^{\prime\prime}, \boldsymbol{\beta}^{\prime \prime}) \in \mathcal{LW}$. Also $\omega^{\infty} \prec (\omega\nu^{\prime})^{\infty}$ and $(\nu\omega^{\prime})^{\infty} \prec \nu^{\infty}$. Assume that $({\boldsymbol{\alpha}}^{\prime \prime}, {\boldsymbol{\beta}}^{\prime \prime})$ is renormalisable. Then there exist sequences $\{n^{\varpi}_i\}_{i=1}^{\infty}$ $\{n^{\epsilon}_i\}_{i=1}^{\infty}$, $\{m^{\varpi}_j\}_{j=1}^{\infty}$ and $\{m^{\epsilon}_j\}_{j=1}^{\infty} \subset \mathbb{N}$ and words $\varpi, \epsilon$ such that $\varpi = {0-{\max}}_{\varpi}$, $\epsilon = {1-{\min}}_{\epsilon}$, $$(\omega \nu^{\prime})^{\infty} = \varpi\epsilon^{n_1^{\epsilon}}\varpi^{n_1^{\varpi}}\epsilon^{n_2^{\epsilon}}\varpi^{n_2^{\varpi}} \ldots$$ and $$(\nu\omega^{\prime})^{\infty} = \epsilon\varpi^{m_1^{\varpi}}\epsilon^{m_1^{\epsilon}}\varpi^{m_2^{\varpi}}\epsilon^{m_2^{\epsilon}} \ldots.$$ Then $$\varpi^{\infty} \prec (\varpi\epsilon)^{\infty} \preccurlyeq (\omega\nu^{\prime})^{\infty} \prec \varpi\epsilon^{\infty}$$ and $$\epsilon\varpi^{\infty} \prec (\nu\omega^{\prime})^{\infty} \preccurlyeq (\epsilon\varpi)^{\infty} \prec \epsilon^{\infty}.$$ Note that $\omega\nu^{\prime} \neq \varpi\epsilon$ and $\nu\omega^{\prime} \neq \epsilon\varpi$. If the former holds then $\sigma^{\ell(\epsilon)}(\nu\omega^{\prime})= \omega\nu^{\prime}$, which implies that $\nu\omega^{\prime}$ is a cyclic permutation of $\omega\nu^{\prime}$ which contradicts that ${\omega^{\prime}}^{\infty} \prec \omega^{\infty}$ and $\nu^{\infty} \prec {\nu^{\prime}}^{\infty}$. Since ${\boldsymbol{\alpha}}^{\prime \prime}$ and ${\boldsymbol{\beta}}^{\prime \prime}$ are periodic sequences we are sure that $n^{\varpi}_i, n^{\epsilon}_i, m^{\varpi}_j, m^{\epsilon}_j  \neq \infty$. Moreover, from Lemma \ref{boundedsequences} $n_1^{\epsilon}, m_1^{\varpi}$ are the maximal exponents occurring in $\omega\nu^{\prime}$ and $\nu \omega^{\prime}$ respectively. Then $$(\omega\nu^{\prime})^{\infty} = (\varpi\epsilon^{n_1^{\epsilon}}\varpi^{n_1^{\varpi}}\epsilon^{n_2^{\epsilon}}\varpi^{n_2^{\varpi}} \ldots \epsilon^{n_i^{\epsilon}}\varpi^{n_i^{\varpi}})^{\infty}$$ and  $$(\nu\omega^{\prime})^{\infty} = (\epsilon\varpi^{m_1^{\varpi}}\epsilon^{m_1^{\epsilon}}\varpi^{m_2^{\varpi}}\epsilon^{m_2^{\epsilon}} \ldots \varpi^{m_j^{\varpi}}\epsilon^{m_j^{\epsilon}})^{\infty}.$$ Note that if $\varpi^{n_i^{\varpi}} = 0$ or $\epsilon^{m_j^{\epsilon}} = 0$ will not affect our argument. Consider $0-{\hbox{\rm{max}}}_{\nu\omega^{\prime}}$. Observe that $0-{\hbox{\rm{max}}}_{\nu\omega^{\prime}}= \varpi\epsilon^{m_i^{\epsilon}}$ for some $1 \leq i \leq j$. Since $(\boldsymbol{\alpha}^{\prime \prime}, \boldsymbol{\beta}^{\prime \prime}) \in \mathcal{LW}$ and the hypothesis that ${\omega^{\prime}}^{\infty} \prec \omega^{\infty}$ and $\nu^{\infty} \prec {\nu^{\prime}}^{\infty}$ we have that $(0-{\hbox{\rm{max}}}_{\nu\omega^{\prime}})^{\infty} \prec (\omega\nu^{\prime})^{\infty}$. This implies that $\varpi_i = \omega_i$ for every $0\leq i \leq k-1$ where $k = \min\left\{j \mid (\omega\nu^{\prime})_j \neq (0-{\hbox{\rm{max}}}_{\nu\omega^{\prime}})_j\right\}$. Since $(\boldsymbol{\alpha}^{\prime \prime}, \boldsymbol{\beta}^{\prime \prime})$ is renormalisable then $(\omega\nu^{\prime})_k = (0-{\hbox{\rm{max}}}_{\nu\omega^{\prime}})_k$ which contradicts the minimality of $k$. Then $(\boldsymbol{\alpha}^{\prime \prime}, \boldsymbol{\beta}^{\prime \prime})$ is not renormalisable.
\end{proof}

\subsubsection*{A construction for $(\boldsymbol{\alpha}, \boldsymbol{\beta})$ with the specification property}

Consider an essential pair $(\boldsymbol{\alpha}_1, \boldsymbol{\beta}_1) \in \mathcal{LW}$ with associated pair $(\omega_1,\nu_1)$. Let $({\boldsymbol{\alpha}}_1^{\prime}, {\boldsymbol{\beta}}_1^{\prime})$ be an essential pair with associated pair $(\omega^{\prime}_1, \nu^{\prime}_1)$ such that $${\omega^{\prime}_1}^{\infty} \prec {\omega_1}^{\infty} \hbox{\rm{ and }}{\nu_1}^{\infty} \prec {\nu_1^{\prime}}^{\infty}.$$ Let $$(\boldsymbol{\alpha}_{2},\boldsymbol{\beta}_{2})= (\sigma({(\omega_1\nu^{\prime}_1)}^{\infty}), \sigma({(\nu_1\omega^{\prime}_1)}^{\infty})) = (\sigma(\omega_2^{\infty}), \sigma(\nu_2^{\infty})).$$ Consider $(\boldsymbol{\alpha}^{\prime}_2, \boldsymbol{\beta}^{\prime}_2)$ an essential pair with associated pair $(\omega^{\prime}_2,\nu^{\prime}_2)$ such that ${\omega^{\prime}_2}^{\infty} \prec \omega_2^{\infty}$, $\nu_2^{\infty} \prec {\nu^{\prime}_2}^{\infty}$, ${\omega^{\prime}_2}^{\infty} \prec {\omega^{\prime}_1}^{\infty}$ and ${\nu^{\prime}_1}^{\infty} \prec {\nu^{\prime}_2}^{\infty}$. Let $$(\boldsymbol{\alpha}_3, \boldsymbol{\beta}_3) = (\sigma({(\omega_2\nu_2^{\prime})}^{\infty}), \sigma({(\nu_2\omega_2^{\prime})}^{\infty})) = (\sigma(\omega_3^{\infty}), \sigma(\nu_3^{\infty})).$$ Then, for every $n \in \mathbb{N}$, let $$(\boldsymbol{\alpha}_{n},\boldsymbol{\beta}_{n})= (\sigma({(\omega_{n-1}\nu^{\prime}_{n-1})}^{\infty}), \sigma({(\nu_{n-1}\omega^{\prime}_{n-1})}^{\infty})) = (\sigma(\omega_n^{\infty}),\sigma(\nu_n^{\infty}))$$ where $(\omega^{\prime}_{n-1},\nu^{\prime}_{n-1})$ is the associated pair of an essential pair $(\boldsymbol{\alpha}^{\prime}_{n-1}, \boldsymbol{\beta}^{\prime}_{n-1})$ satisfying ${\omega^{\prime}_{n-1}}^{\infty} \prec \omega_{n-1}^{\infty}$, $\nu_{n-1}^{\infty} \prec {\nu^{\prime}_{n-1}}^{\infty}$, ${\omega^{\prime}_{n-1}}^{\infty} \prec {\omega^{\prime}_{n-2}}^{\infty}$ and ${\nu^{\prime}_{n-2}}^{\infty} \prec {\nu^{\prime}_{n-1}}^{\infty}$. Let $$(\boldsymbol{\alpha}, \boldsymbol{\beta}) = \left(\mathop{\lim}\limits_{n \to \infty} \boldsymbol{\alpha}_n, \mathop{\lim}\limits_{n \to \infty}\boldsymbol{\beta}_n \right).$$ From Lemma \ref{constructionspec1} we have that $(\Sigma_{(\boldsymbol{\alpha}, \boldsymbol{\beta})}, \sigma_{(\boldsymbol{\alpha}, \boldsymbol{\beta})})$ is a coded system. We claim that $(\Sigma_{(\boldsymbol{\alpha}, \boldsymbol{\beta})}, \sigma_{(\boldsymbol{\alpha}, \boldsymbol{\beta})})$ has the specification property. 

To prove this we will show that $${p_1}_{(\boldsymbol{\alpha}_n,\boldsymbol{\beta}_n)}{p_2}_{(\boldsymbol{\alpha}_n,\boldsymbol{\beta}_n)} = {p_1}_{(\boldsymbol{\alpha}_2,\boldsymbol{\beta}_2)}{p_2}_{(\boldsymbol{\alpha}_2,\boldsymbol{\beta}_2)}$$ for every $n \geq 2$. Observe that the associated pair of $(\boldsymbol{\alpha}_2,\boldsymbol{\beta}_2)$ is $(\omega_1\nu_1^{\prime}, \nu\omega^{\prime})$. Since ${\omega_1^{\prime}}^{\infty}\prec \omega_1^{\infty}$ and $\nu_1^{\infty} \prec {\nu_2^{\prime}}^{\infty}$ then ${p_1}_{(\boldsymbol{\alpha}_2,\boldsymbol{\beta}_2)} = \omega_1$ and ${p_2}_{(\boldsymbol{\alpha}_2,\boldsymbol{\beta}_2)} = \nu_1$. Let $n \geq 2$. Note that the associated pair of $(\boldsymbol{\alpha}_n,\boldsymbol{\beta}_n)$ is $(\omega_1\nu_1^{\prime}\nu_2^{\prime}\ldots \nu_{n-1}^{\prime}, \nu\omega_1^{\prime}\ldots\omega_{n-1}^{\prime})$. Since $$\omega_{n-1}^{\prime} \prec \ldots \prec {\omega_2^{\prime}}^{\infty} \prec {\omega_1^{\prime}}^{\infty} \prec {\omega_1}^{\infty}$$ and $${\nu_1}^{\infty} \prec {\nu_1^{\prime}}^{\infty} \prec \ldots \prec {\nu_{n-1}^{\prime}}^{\infty}$$ then ${p_1}_{(\boldsymbol{\alpha}_n,\boldsymbol{\beta}_n)} = \omega_1$ and ${p_2}_{(\boldsymbol{\alpha}_n,\boldsymbol{\beta}_n)}$. Then $${p_1}_{(\boldsymbol{\alpha}_n,\boldsymbol{\beta}_n)}{p_2}_{(\boldsymbol{\alpha}_n,\boldsymbol{\beta}_n)} = {p_1}_{(\boldsymbol{\alpha}_2,\boldsymbol{\beta}_2)}{p_1}_{(\boldsymbol{\alpha}_2,\boldsymbol{\beta}_2)}$$ for every $n \geq 2$. This proves our assertion. 

\subsection*{Examples with no specification}

\begin{theorem}
Let $(\Sigma_{({\boldsymbol{\alpha}},{\boldsymbol{\beta}})}, \sigma_{({\boldsymbol{\alpha}}_,{\boldsymbol{\beta}})})$ be a coded system and $\left\{({\boldsymbol{\alpha}}_n, {\boldsymbol{\beta}}_n)\right\}_{n=1}^{\infty}$ be the sequence constructed in Theorem \ref{lemaaproxtrans1}. If there exists a subsequence $\left\{({\boldsymbol{\alpha}_n}_k,{\boldsymbol{\beta}_n}_k)\right\}_{k=1}^{\infty} \subset \left\{(\boldsymbol{\alpha}_n,\boldsymbol{\beta}_n)\right\}_{n=1}^{\infty}$ such that for every $k \in \mathbb{N}$, $$\ell({p_1}_{({\boldsymbol{\alpha}}_{n_k},{\boldsymbol{\beta}}_{n_k})} {p_2}_{({\boldsymbol{\alpha}}_{n_k},{\boldsymbol{\beta}}_{n_k})}) < \ell({p_1}_{({\boldsymbol{\alpha}}_{n_{k+1}},{\boldsymbol{\beta}}_{n_{k+1}})}{p_2}_{({\boldsymbol{\alpha}}_{n_{k+1}},{\boldsymbol{\beta}}_{n_{k+1}})}) $$ and $$\ell({p_1}_{({\boldsymbol{\alpha}}_{n_k},{\boldsymbol{\beta}}_{n_k})} {p_2}_{({\boldsymbol{\alpha}}_{n_k},{\boldsymbol{\beta}}_{n_k})}) \leq s_{({\boldsymbol{\alpha}}_{n_{k+1}},{\boldsymbol{\beta}}_{n_{k+1}})}$$ then $(\Sigma_{(\boldsymbol{\alpha},\boldsymbol{\beta})},\sigma_{(\boldsymbol{\alpha},\boldsymbol{\beta})})$ does not have specification.\label{nospec}
\end{theorem}
\begin{proof}
Observe that if $s_{({\boldsymbol{\alpha}}_{n_k}, {\boldsymbol{\beta}}_{n_k})} < s_{({\boldsymbol{\alpha}}_{n_{k+1}}, {\boldsymbol{\beta}}_{n_{k+1}})}$ for every $k \in \mathbb{N}$ then $(\Sigma_{(\boldsymbol{\alpha},\boldsymbol{\beta})},\sigma_{(\boldsymbol{\alpha},\boldsymbol{\beta})})$ does not have specification. From the proof of Theorem \ref{renorm2} we have that $$s_{({\boldsymbol{\alpha}}_{n_k}, {\boldsymbol{\beta}}_{n_k})} \leq \ell({p_1}_{({\boldsymbol{\alpha}}_{n_k}, {\boldsymbol{\beta}}_{n_k})}{p_2}_{({\boldsymbol{\alpha}}_{n_k}, {\boldsymbol{\beta}}_{n_k})})$$ and $$s_{({\boldsymbol{\alpha}}_{n_{k+1}}, {\boldsymbol{\beta}}_{n_{k+1}})} \leq \ell({p_1}_{({\boldsymbol{\alpha}}_{n_{k+1}}, {\boldsymbol{\beta}}_{n_{k+1}})}{p_2}_{({\boldsymbol{\alpha}}_{n_{k+1}}, {\boldsymbol{\beta}}_{n_{k+1}})}).$$ From hypothesis we have that $\ell({p_1}_{({\boldsymbol{\alpha}}_{n_k}, {\boldsymbol{\beta}}_{n_k})}{p_2}_{({\boldsymbol{\alpha}}_{n_k}, {\boldsymbol{\beta}}_{n_k})}) < s_{({\boldsymbol{\alpha}}_{n_{k+1}}, {\boldsymbol{\beta}}_{n_{k+1}})}$ for every $k \in \mathbb{N}$. Then $$s_{({\boldsymbol{\alpha}}_{n_{k}}, {\boldsymbol{\beta}}_{n_{k}})} \leq \ell({p_1}_{({\boldsymbol{\alpha}}_{n_k}, {\boldsymbol{\beta}}_{n_k})}{p_2}_{({\boldsymbol{\alpha}}_{n_k}, {\boldsymbol{\beta}}_{n_k})}) < s_{({\boldsymbol{\alpha}}_{n_{k+1}}, {\boldsymbol{\beta}}_{n_{k+1}})}$$ for every $k \in \mathbb{N}$ which concludes the proof.
\end{proof}

To finish this section we construct a family of examples satisfying Theorem \ref{nospec}. For this purpose we show the following technical lemma.

\begin{lemma}
Let $(\boldsymbol{\alpha}, \boldsymbol{\beta}), ({\boldsymbol{\alpha}}^{\prime}, {\boldsymbol{\beta}}^{\prime}) \in \mathcal{LW}$ be essential pairs with associated pairs $(\omega,\nu)$ and $(\omega^{\prime},\nu^{\prime})$ respectively such that $(\boldsymbol{\alpha}, \boldsymbol{\beta}) \neq ({\boldsymbol{\alpha}}^{\prime}, {\boldsymbol{\beta}}^{\prime})$. Suppose that ${\omega^{\prime}}^{\infty} \prec \omega^{\infty}$ and $\nu^{\infty} \prec {\nu^{\prime}}^{\infty}$. For every $n,m \in \mathbb{N}$ the pair $({\boldsymbol{\alpha}}^{\prime \prime}, {\boldsymbol{\beta}}^{\prime \prime})$ given by $0{\boldsymbol{\alpha}}^{\prime \prime} = (\omega\nu^n\nu^{\prime})^{\infty}$ and $1{\boldsymbol{\beta}}^{\prime \prime} = (\nu\omega^m\omega^{\prime})^{\infty}$ is an essential pair. \label{constructionspec2}
\end{lemma} 
\begin{proof}
Let $n,m \in \mathbb{N}$. Note that since ${\omega^{\prime}}^{\infty} \prec \omega^{\infty}$ and $\nu^{\infty} \prec {\nu^{\prime}}^{\infty}$ then $(\boldsymbol{\alpha}^{\prime \prime}, \boldsymbol{\beta}^{\prime \prime}) \in \mathcal{LW}$. Also, $$\omega\nu^{\infty} \prec (\omega\nu^n\nu^{\prime})^{\infty} \prec (\omega \nu^{\prime})^{\infty}$$ and $$\nu\omega^{\infty} \prec (\nu\omega^m\omega^{\prime})^{\infty} \prec (\nu\omega^{\prime})^{\infty}.$$ Let us assume that $(\boldsymbol{\alpha}^{\prime \prime}, \boldsymbol{\beta}^{\prime \prime})$ is renormalisable. Then there exist sequences $\{n^{\varpi}_i\}_{i=1}^{\infty}$ $\{n^{\epsilon}_i\}_{i=1}^{\infty}$, $\{m^{\varpi}_j\}_{j=1}^{\infty}$ and $\{m^{\epsilon}_j\}_{j=1}^{\infty} \subset \mathbb{N}$ and words $\varpi, \epsilon$ such that $\varpi = {0-{\max}}_{\varpi}$, $\epsilon = {1-{\min}}_{\epsilon}$, $$(\omega\nu^{n} \nu^{\prime})^{\infty} = \varpi\epsilon^{n_1^{\epsilon}}\varpi^{n_1^{\varpi}}\epsilon^{n_2^{\epsilon}}\varpi^{n_2^{\varpi}} \ldots$$ and $$(\nu\omega^m\omega^{\prime})^{\infty} = \epsilon\varpi^{m_1^{\varpi}}\epsilon^{m_1^{\epsilon}}\varpi^{m_2^{\varpi}}\epsilon^{m_2^{\epsilon}} \ldots.$$ Observe that $0-\hbox{\rm{max}}_{\nu\omega^m\omega^{\prime}} = \omega^m\omega^{\prime}\nu$. Moreover $(0-\hbox{\rm{max}}_{\nu\omega^m\omega^{\prime}})^{\infty} \prec (\omega\nu^{n} \nu^{\prime})^{\infty}$ since $(\boldsymbol{\alpha}^{\prime \prime},\boldsymbol{\beta}^{\prime \prime}) \in \mathcal{LW}$. Furthermore, $0-\hbox{\rm{max}}_{\nu\omega^m\omega^{\prime}} = \omega^m\omega^{\prime}\nu$ and $(\omega^m\omega^{\prime}\nu)^{\infty} = \varpi\epsilon^{m^{\epsilon}_i}\varpi^{m^{\varpi}_{i+1}}\ldots$ for $1\leq i < j$. This implies that $\varpi = \omega$ which gives $(\omega^m \omega^{\prime}\nu)^{\infty} = \varpi^m\epsilon^{m^{\epsilon}_i}\varpi^{m^{\varpi}_{i+1}}\ldots$ which contradicts that $\omega^m\omega^{\prime}\nu = 0-\hbox{\rm{max}}_{\nu\omega^m\omega^{\prime}}$. Thus $(\boldsymbol{\alpha},\boldsymbol{\beta})$ is not renormalisable.    
\end{proof}

\subsubsection*{A construction for $(\boldsymbol{\alpha}, \boldsymbol{\beta})$ with no specification}

Let $(\boldsymbol{\alpha}_1, \boldsymbol{\beta}_1) \in \mathcal{LW}$ be an essential pair with associated pair $(\omega_1,\nu_1)$. Consider an essential pair $({\boldsymbol{\alpha}}_1^{\prime}, {\boldsymbol{\beta}}_1^{\prime})$ with associated pair $(\omega^{\prime}_1, \nu^{\prime}_1)$ such that $${\omega^{\prime}_1}^{\infty} \prec {\omega_1}^{\infty} \hbox{\rm{ and }} {\nu_1}^{\infty} \prec {\nu_1^{\prime}}^{\infty}.$$ Let $$(\boldsymbol{\alpha}_{2},\boldsymbol{\beta}_{2})= (\sigma{(\omega_1{\nu_1}^{j_1}\nu^{\prime}_1)}^{\infty}, \sigma{(\nu_1{\omega_1}^{k_1}\omega^{\prime}_1)}^{\infty}) = (\sigma(\omega_2^{\infty}), \sigma(\nu_2^{\infty}))$$ where $j_1, k_1 \in \mathbb{N}$ and $j_1, k_1 \geq 2$. Consider $(\boldsymbol{\alpha}^{\prime}_2, \boldsymbol{\beta}^{\prime}_2)$ an essential pair with associated pair $(\omega^{\prime}_2,\nu^{\prime}_2)$ such that ${\omega^{\prime}_2}^{\infty} \prec \omega_2^{\infty}$, $\nu_2^{\infty} \prec {\nu^{\prime}_2}^{\infty}$, ${\omega^{\prime}_2}^{\infty} \prec {\omega^{\prime}_1}^{\infty}$ and ${\nu^{\prime}_1}^{\infty} \prec {\nu^{\prime}_2}^{\infty}$. Let $$(\boldsymbol{\alpha}_3, \boldsymbol{\beta}_3) = (\sigma({(\omega_2\nu_2^{j_2}\nu_2^{\prime})}^{\infty}), \sigma({(\nu_2\omega_2^{k_2}\omega_2^{\prime})}^{\infty})) = (\sigma(\omega_3^{\infty}), \sigma(\nu_3^{\infty}))$$ with and $j_2, k_2 \in \mathbb{N}$ and $j_2, k_2 \geq 2$. Then, for every $n \in \mathbb{N}$, let $$(\boldsymbol{\alpha}_{n},\boldsymbol{\beta}_{n})= (\sigma({(\omega_{n-1}\nu^{j_{n-1}}_{n-1}\nu^{\prime}_{n-1})}^{\infty}), \sigma({(\nu_{n-1}\omega_{n-1}^{k_{n-1}}\omega^{\prime}_{n-1})}^{\infty})) = (\sigma(\omega_n^{\infty}),\sigma(\nu_n^{\infty}))$$ with $j_n, k_n \in \mathbb{N}$ and $j_n, k_n \geq 2$where $(\omega^{\prime}_{n-1},\nu^{\prime}_{n-1})$ is the associated pair of an essential pair $(\boldsymbol{\alpha}^{\prime}_{n-1}, \boldsymbol{\beta}^{\prime}_{n-1})$ satisfying ${\omega^{\prime}_{n-1}}^{\infty} \prec \omega_{n-1}^{\infty}$, $\nu_{n-1}^{\infty} \prec {\nu^{\prime}_{n-1}}^{\infty}$, ${\omega^{\prime}_{n-1}}^{\infty} \prec {\omega^{\prime}_{n-2}}^{\infty}$ and ${\nu^{\prime}_{n-2}}^{\infty} \prec {\nu^{\prime}_{n-1}}^{\infty}$. Let $$(\boldsymbol{\alpha}, \boldsymbol{\beta}) = \left(\mathop{\lim}\limits_{n \to \infty} \boldsymbol{\alpha}_n, \mathop{\lim}\limits_{n \to \infty}\boldsymbol{\beta}_n \right).$$ From Lemma \ref{constructionspec2} we have that $(\Sigma_{(\boldsymbol{\alpha}, \boldsymbol{\beta})}, \sigma_{(\boldsymbol{\alpha}, \boldsymbol{\beta})})$ is a coded system. We claim that $(\Sigma_{(\boldsymbol{\alpha}, \boldsymbol{\beta})}, \sigma_{(\boldsymbol{\alpha}, \boldsymbol{\beta})})$ does not have the specification property. 

To prove our assertion, we will show that the sequence $\left\{(\boldsymbol{\alpha}_n, \boldsymbol{\beta}_n)\right\}_{n=1}^{\infty}$ satisfies that $$\ell({p_1}_{({\boldsymbol{\alpha}}_{n},{\boldsymbol{\beta}}_{n})} {p_2}_{({\boldsymbol{\alpha}}_{n},{\boldsymbol{\beta}}_{n})}) < \ell({p_1}_{({\boldsymbol{\alpha}}_{n+1},{\boldsymbol{\beta}}_{n+1})}{p_2}_{({\boldsymbol{\alpha}}_{n+1},{\boldsymbol{\beta}}_{n+1})})$$ and $$\ell({p_1}_{({\boldsymbol{\alpha}}_{n},{\boldsymbol{\beta}}_{n})} {p_2}_{({\boldsymbol{\alpha}}_{n},{\boldsymbol{\beta}}_{n})}) \leq s_{({\boldsymbol{\alpha}}_{n+1},{\boldsymbol{\beta}}_{n+1})}$$ for every $n \in \mathbb{N}$. Consider $n \in \mathbb{N}$. Note that the associated pairs of $({\boldsymbol{\alpha}}_{n},{\boldsymbol{\beta}}_{n})$ and $({\boldsymbol{\alpha}}_{n+1},{\boldsymbol{\beta}}_{n+1})$ are $(\omega_{n-1}{\nu_{n-1}}^{j_{n-1}}\nu_{n-1}^{\prime}, \nu_{n-1}{\omega_{n-1}}^{k_{n-1}}\omega_{n-1}^{\prime}),$ $(\omega_{n}{\nu_{n}}^{j_{n}}\nu_{n}^{\prime}, \nu_{n}{\omega_{n}}^{k_{n}}\omega_{n}^{\prime})$ respectively. Since $\omega_n = \omega_{n-1}{\nu_{n-1}}^{j_{n-1}}\nu_{n-1}^{\prime}$ and $\nu_n = \nu_{n-1}{\omega_{n-1}}^{k_{n-1}}\omega_{n-1}^{\prime}$ then the associated pair of $(\boldsymbol{\alpha}_{n+1},\boldsymbol{\beta}_{n+1})$ is $$\left(\omega_{n-1}{\nu_{n-1}}^{j_{n-1}}\nu_{n-1}^{\prime}{(\nu_{n-1}{\omega_{n-1}}^{k_{n-1}}\omega_{n-1}^{\prime})}^{j_{n}}\nu_{n}^{\prime}, \nu_{n-1}{\omega_{n-1}}^{k_{n-1}}\omega_{n-1}^{\prime}{(\omega_{n-1}{\nu_{n-1}}^{j_{n-1}}\nu_{n-1}^{\prime})}^{k_{n}}\omega_{n}^{\prime}\right).$$ Observe that $p_1(\boldsymbol{\alpha}_n,\boldsymbol{\beta}_n) = \omega_{n-1}\nu_{n-1}^{j_{n-1}}$ and $p_2(\boldsymbol{\alpha}_n,\boldsymbol{\beta}_n) = \nu_{n-1}\omega_{n-1}^{k_{n-1}}$ since $\omega_{n-1}^{\prime} \prec \omega_{n-1}$ and $\nu_{n-1}^{\prime} \prec \nu_{n-1}$. Moreover $p_1(\boldsymbol{\alpha}_{n+1},\boldsymbol{\beta}_{n+1}) = \omega_{n}\nu_{n}^{j_{n+1}}$ and $p_2(\boldsymbol{\alpha}_{n+1},\boldsymbol{\beta}_{n+1}) = \nu_{n}\omega_{n}^{k_{n+1}}$ since $\omega_{n}^{\prime} \prec \omega_{n}$ and $\nu_{n}^{\prime} \prec \nu_{n}$. This gives that $$\ell({p_1}_{({\boldsymbol{\alpha}}_{n},{\boldsymbol{\beta}}_{n})} {p_2}_{({\boldsymbol{\alpha}}_{n},{\boldsymbol{\beta}}_{n})}) < \ell({p_1}_{({\boldsymbol{\alpha}}_{n+1},{\boldsymbol{\beta}}_{n+1})}{p_2}_{({\boldsymbol{\alpha}}_{n+1},{\boldsymbol{\beta}}_{n+1})}).$$ Consider $\upsilon, \nu \in \mathcal{L}(\Sigma_{({\boldsymbol{\alpha}}_{n+1},{\boldsymbol{\beta}}_{n+1})}) \setminus \mathcal{L}(\Sigma_{({\boldsymbol{\alpha}}_{n},{\boldsymbol{\beta}}_{n})})$. In particular, observe that $\omega_n1 \notin \mathcal{L}(\Sigma_{(\boldsymbol{\alpha}_{n+1}, \boldsymbol{\beta}_{n+1})}) \setminus \mathcal{L}(\Sigma_{(\boldsymbol{\alpha}_{n}, \boldsymbol{\beta}_{n})})$. Since $(\boldsymbol{\alpha}_{n+1}, \boldsymbol{\beta}_{n+1})$ is an essential pair, there exist a bridge $\varpi$ between $\omega_n1$ and $\nu^{\prime}_n$. Since $\nu_n \prec \nu_n^{\prime}$ then it is clear that $\ell(\nu_n) - 1 < \ell(\varpi)$. Moreover $\ell(\nu_n) -1 = \ell(\nu_{n-1}{\omega_{n-1}}^{k_{n-1}}\omega_{n-1}^{\prime})-1$. Since $j_n \geq 2$ then $\sigma^{\ell(\nu_n) -1}(\varpi)$ contains either $\nu_n$ or $\omega_n$. If $\ell(\nu_n) \leq \ell(\omega_n)$ then $$\ell(\varpi) = j_n\ell(\nu_{n-1}{\omega_{n-1}}^{k_{n-1}}\omega_{n-1}^{\prime})-1 = j_n(\ell(\nu_n-1)+k_{n-1}\ell(\omega_{n-1})+\ell(\omega_{n-1}^{\prime})),$$ which gives us the desired conclusion.

\vspace{1em}It is worth to mention that it is possible to perform a similar construction for subshifts with no specification by considering the same hypothesis as in the previous example and the sequence $\left\{(\boldsymbol{\alpha}_{n},\boldsymbol{\beta}_{n}) \right\}_{n=1}^{\infty}$ $$(\boldsymbol{\alpha}_{n},\boldsymbol{\beta}_{n})= (\sigma({(\omega_{n-1}{\omega_{n-1}^{\prime}}^{j_{n-1}}_{n-1}\nu^{\prime}_{n-1})}^{\infty}), \sigma({(\nu_{n-1}{\nu_{n-1}^{\prime}}^{k_{n-1}}\omega^{\prime}_{n-1})}^{\infty})).$$

\section*{Acknowledgements}

\noindent The author is indebted to Nikita Sidorov for his encouragement and guidance. Also, the author wishes to thank Lyndsey Clark for stimulating conversations and their useful remarks. In addition, the author wishes to express his gratitude to Jean Paul Allouche for his hospitality in the author's visit to Institut de Math\'ematiques de Jussieu-PRG, Universit\'e Pierre et Marie Curie. The visit was sponsored by the Royal Society International Exchanges grant IE130940: Open Dynamical Systems and the Lexicographic World.

%\bibliography{resultados1}
%\bibliographystyle{plain}

\end{document}